\documentclass[11pt]{article}
\usepackage{amsmath, amssymb, theorem, latexsym, epsfig}
\numberwithin{equation}{section}

\theoremstyle{plain}
\theorembodyfont{\itshape}
\newtheorem{theorem}{Theorem}[section]
\newtheorem{proposition}[theorem]{Proposition}
\newtheorem{lemma}[theorem]{Lemma}
\newtheorem{corollary}[theorem]{Corollary}
\theorembodyfont{\rmfamily}
\newtheorem{definition}[theorem]{Definition}

\newtheorem{example}[theorem]{Example}
\newtheorem{remark}[theorem]{Remark}

\newenvironment{proof}{{\noindent \textbf{Proof}\,\,}}{\hspace*{\fill}$\Box$\medskip}

\def\mcv{\mathcal V}

\def\wt#1{\widetilde#1}
\def\rr{\mathbb R}
\def\var{\varepsilon}

\def\wh#1{\widehat#1}

 \def\mcx{\mathcal X}
\def\mcy{\mathcal Y}

\def\nn{\mathbb N}
 \def\mct{\mathcal T}
 
 \def\mcb{\mathcal B}
 \def\zz{\mathbb Z}
 \def\rp{\mathbb{RP}}
 \def\mcb{\mathcal B}
 
\def\La{\Lambda}

\def\mca{\mathcal A}
  \def\diag{\operatorname{diag}}
 \def\la{\lambda}
\def\mcc{\mathcal C}

\def\mcu{\mathcal U}

\def\mcd{\mathcal D}
\def\Xi{\mathcal Z}
\def\mcl{\mathcal L}
\def\mcd{\mathcal D}
\def\mcf{\mathcal F}
\def\sh{\sinh}
\def\mco{\mathcal O}
\def\sign{\operatorname{sign}}
\title{On curves with Poritsky property}
\author{Alexey Glutsyuk\thanks{ CNRS, France (UMR 5669 (UMPA, ENS de Lyon) and UMI 2615 (Interdisciplinary Scientific Center J.-V.Poncelet)), 
Lyon, France. 
E-mail:
aglutsyu@ens-lyon.fr}
\thanks{National Research University Higher School of Economics (HSE), Russian Federation}
\thanks{The author is partially supported by Laboratory of Dynamical Systems and Applications NRU HSE of the Ministry of science and higher education of the RF grant ag. No 075-15-2019-1931}
 \thanks{Supported by part by RFBR grants 16-01-00748 and 16-01-00766}
 \thanks{Partially supported by RFBR and JSPS (research project 19-51-50005)}
 \thanks{This material is based upon work supported by the National Science Foundation under Grant No. DMS-1440140, while the author was in residence at the Mathematical Sciences Research Institute in Berkeley, California, during the fall semester 2018}}
\begin{document}
\maketitle
\begin{abstract} Reflection in planar billiard  acts on the space of oriented lines. 
For a given closed 
convex planar curve $\gamma$ 
the string construction 
yields a one-parameter family of nested billiards containing $\gamma$ for which 
$\gamma$ is a {\it caustic:} 
each tangent line to $\gamma$ is reflected  to a line tangent to $\gamma$. 
Thus, the reflections in these billiards  act on the tangent 
lines to $\gamma$ and hence, on the tangency points, inducing a family of {\it string  diffeomorphisms} $\mathcal T_p:\gamma\to\gamma$.   
We say that $\gamma$ has {\it string Poritsky property,} if it admits a parameter $t$ (called {\it Poritsky  
string length}) in which all the transformations $\mathcal T_p$ with small $p$ 
are translations $t\mapsto t+c_p$.   These definitions also make sense for germs 
of curves $\gamma$. Poritsky property is closely related to the famous Birkhoff Conjecture. 
 It is classically known that each conic has string Poritsky property. 
 In 1950 H.Poritsky proved the converse: {\it each germ of  planar 
curve with Poritsky property is a conic.} 
In the present paper we extend this Poritsky's result to germs of curves in 
simply connected complete 
Riemannian surfaces of constant curvature and to outer billiards on all these surfaces. In the general case of 
curves with Poritsky property on any two-dimensional surface with Riemannian metric we prove the following two results: {\it 1) the Poritsky string length 
coincides with Lazutkin parameter, introduced by V.F.Lazutkin in 1973, up to additive and multiplicative 
constants;  2) a germ of $C^5$-smooth curve with Poritsky property is uniquely 
determined by its 4-th jet.} In the Euclidean case the latter statement follows from the above-mentioned 
Poritsky's result.  
\end{abstract}
\tableofcontents
\section{Introduction and main results}

Consider the billiard in a bounded planar domain $\Omega\subset\rr^2$ with a strictly convex smooth boundary. The 
billiard dynamics $T$ acts on the space of oriented lines intersecting 
$\Omega$. Namely, let $L$ be an oriented line 
intersecting $\Omega$, and let $A$ be its last point (in the sense of orientation) of its intersection with 
$\partial\Omega$. By definition $T(L)$ is the image of the line $L$ under the symmetry with respect to the tangent 
line $T_A\partial\Omega$, being oriented from the point $A$ inside the domain $\Omega$. A  curve $\gamma\subset\rr^2$  is a 
{\it caustic} of the billiard $\Omega$, if each line tangent to $\gamma$ is reflected from the boundary $\partial\Omega$ 
again to a line tangent to $\gamma$; in other words, if the curve formed by oriented lines tangent to $\gamma$ is 
invariant under the billiard transformation $T$. In what follows we consider only {\it smooth caustics} (in particular, 
without cusps). 

It is well-known that each planar billiard with sufficiently smooth strictly convex boundary has a Cantor family of caustics 
\cite{laz}. Analogous statement for outer billiards was proved in \cite{amiran2}. 
Every elliptic billiard is {\it Birkhoff caustic integrable,} that is, an inner neighborhood of its boundary is foliated by closed 
caustics. The famous Birkhoff Conjecture states the converse: the only Birkhoff caustic integrable planar billiards are 
elllipses. The Birkhoff Conjecture together with its extension to billiards on surfaces of constant curvature and its 
version (due to Sergei Tabachnikov) for outer billiards on the latter surfaces are big open problems, see, e.g.,  
\cite{hess, kalsor} and references therein for history and related results.

It is  well-known that each smooth convex planar curve $\gamma$ is a caustic for a family of  billiards 
$\Omega=\Omega_p$, $p\in\rr_+$, whose boundaries $\Gamma=\Gamma_p=\partial\Omega_p$ are given 
by the $p$-th string constructions, see \cite[p.73]{tab}.  Namely, let $|\gamma|$ denote the length of the curve $\gamma$. 
Take an arbitrary number $p>0$ and 
a string of length $p+|\gamma|$ enveloping the curve $\gamma$. Let us put  a pencil between the curve 
$\gamma$ and the string, and let us push it out of $\gamma$ until a position, when the string, which envelopes 
$\gamma$ and the pencil,  becomes stretched. Then let us move  the pencil around the curve $\gamma$ so that the 
string remains  stretched. Thus moving pencil draws a convex curve that 
is called the {\it $p$-th string construction,} see Fig. 1.
\begin{figure}[ht]
  \begin{center}
   \epsfig{file=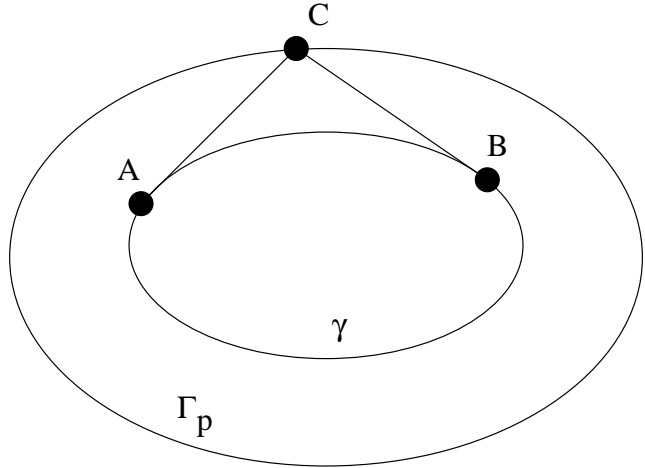}
    \caption{The string construction.}
    \label{fig:01}
  \end{center}
\end{figure}
 
For every $A\in\gamma$ by $G_A$ we denote the line 
tangent to $\gamma$ at $A$. If $\gamma$ is oriented by a 
vector in $T_A\gamma$, then 
we orient $G_A$ by the same vector.
The billiard reflection $T_p$  from the curve  $\Gamma_p$ 
 acts on the oriented lines tangent to $\gamma$. It induces the mapping $\mct_p:\gamma\to\gamma$ 
 acting on tangency points and called {\it string diffeomorphism.} It sends 
 each point $A\in\gamma$ 
 to the point of tangency of the curve $\gamma$ with  the line $T_p(G_A)$.  
 
 Consider the special case, when $\gamma$ is an ellipse. Then for every 
 $p>0$ the 
 curve $\Gamma_p$ given by the $p$-th string construction is an ellipse confocal to $\gamma$.  Every ellipse  
$\gamma$ admits a canonical bijective parametrization by the circle $S^1=\rr\slash2\pi\zz$ equipped with parameter $t$ 
 such that for every $p>0$ small enough one has $\mct_p(t)=t+c_p$, $c_p=c_p(\gamma)$, 
 see \cite[the discussion before corollary 4.5]{tab}. The property of existence 
 of the above parametrization 
 will be called the {\it string Poritsky property,} and the parameter $t$ will be called {\it Poritsky--Lazutkin  string length.}
 
 In his seminal paper \cite{poritsky} Hillel Poritsky proved the Birkhoff Conjecture under the following 
 additional assumption called {\it Graves} (or {\it evolution}) {\it property}: for every two nested caustics $\gamma_{\lambda}$, $\gamma_{\mu}$ 
 of the billiard under question the smaller caustic  $\gamma_{\lambda}$ is also a caustic 
 of the billiard in the bigger caustic $\gamma_{\mu}$. His beautiful geometric proof was based on his remarkable 
 theorem  stating that in Euclidean plane only conics have string Poritsky 
 property, see \cite[section  7]{poritsky}. 
 
 In the present paper we extend the above Poritsky's 
 theorem to billiards on simply connected complete 
 surfaces of constant curvature (Subsection 1.1 and Section 4) 
 and prove its version for 
 outer billiards and area construction on these surfaces (Subsection 1.2 and Section 5). 
 All the results of the present paper will be stated and proved for germs of curves, and thus, in Subsection 1.1 (1.2) 
 we state the definitions of Poritsky string (area) property for germs.  
 We also study Poritsky property on arbitrary surfaces equipped 
 with a Riemannian metric. 
 In this  general case we show that 
 the Poritsky  string length coincides with the Lazutkin parameter 
 \begin{equation}t_L(s)=\int_{s_0}^s\kappa^{\frac23}(\zeta)d\zeta\label{lazparl}\end{equation}
  introduced in \cite[formula (1.3)]{laz}, 
 up to multiplicative and additive constants  
 (Theorem \ref{ttk} in Subsection 1.3, proved in Section 6). Here $\kappa$ is the geodesic 
 curvature. 
 This explains the name "Poritsky--Lazutkin length".
 
 In the same famous paper \cite{laz}, 
 for a given curve $\gamma\subset\rr^2$ 
 V.F.Lazutkin introduced 
 remarkable coordinates $(x,y)$ on the space of oriented geodesics,  in which the billiard ball map given 
 by reflection from the curve $\gamma$ takes the form 
 $$(x,y)\mapsto(x+y+o(y), y+o(y^2));$$
 \centerline{the $x$-axis coincides with the set of the geodesics tangent to $\gamma$;} 
 $$  \ x=t_L(s) 
 \text{ on the } x-\text{axis up to multiplicative and additive constants.}$$ 
 In \cite{mm} Melrose and Marvizi studied 
 planar billiard ball map with reflection from a 
 $C^{\infty}$-smooth curve and proved their famous 
 theorem on existence of an interpolating Hamiltonian. 
 Namely, they have shown that  the billiard ball map 
 coincides with a unit time flow map   of 
 appropriately "time-rescaled" 
 smooth Hamiltonian vector field, up to a flat correction. 
 
 We state and prove the above-mentioned 
  Lazutkin's result (in slightly different 
 form) for a more 
 general class of symplectic maps, the so-called 
 "weakly billiard-like maps", which include 
 billiard ball maps in arbitrary Riemannian surface. 
 Using it, we extend Theorem \ref{ttk} on coincidence 
 of Poritsky and Lazutkin parameters to families 
 of weakly billiard-like maps with appropriate 
 converging family of invariant curves 
 (Theorem \ref{tsympor} stated and proved in Section 7). We retrieve Theorem \ref{ttk} (for $C^6$-smooth curves) from Theorem \ref{tsympor} at the end of Section 7. The proof of Theorem \ref{tsympor} is based on 
  Lemma \ref{plog} on asymptotic behavior of orbits of a weakly 
 billiard-like map, which  may have an independent interest.

 For curves on arbitrary  surface equipped with a $C^6$-smooth 
 Riemannian metric 
 we show that a $C^5$-smooth germ of 
 curve with string  Poritsky property is uniquely determined by its 4-jet (Theorem \ref{uniq4} stated in Subsection 1.4 and proved in Section 8). 
 
 Theorem \ref{thsm}  in Subsection 1.1 (proved in Section 3) states that if a metric and a germ of 
  curve $\gamma$ are both $C^k$-smooth, then the  string curve foliation 
 is tangent to a line field $C^{[\frac k2]-1}$-smooth on a domain adjacent 
 to $\gamma$ (including $\gamma$). 
 
 In Section 2 we present a Riemannian-geometric 
background material on normal coordinates, equivalent definitions of 
geodesic curvature etc. used in the proofs of main results.

\subsection{Poritsky property for string construction and Poritsky--Lazutkin string length} 
Let $\Sigma$ be a two-dimensional surface equipped with a Riemannian metric. 
Let $\gamma\subset\Sigma$ be a  smooth curve (a germ of smooth curve at a point $O\in\Sigma$). 
We consider it to be {\it convex:} 
its geodesic curvature should be non-zero. For every given two points $A,B\in\gamma$ 
close enough 
by $C_{AB}$ we will denote the unique point (close to them) of intersection  of the geodesics $G_A$ and $G_B$ tangent to $\gamma$ at $A$ and 
$B$ respectively. (Its existence will be proved in Subsection 2.1.) 
Set 
$$\la(A,B):= \text{ the length of the arc } AB \text{ of the curve } \gamma, $$
\begin{equation}L(A,B):=|AC_{AB}|+|BC_{AB}|-\la(A,B).\label{lab}\end{equation}
Here for $X,Y\in\Sigma$ close enough and lying in a compact subset in $\Sigma$  
by $|XY|$ we denote the length of small geodesic segment connecting $X$ and $Y$. 

\begin{definition} \label{dstring} (equivalent definition of string construction) 
Let $\gamma\subset\Sigma$ be a germ of curve with non-zero geodesic curvature. For every $p\in\rr_+$ small enough 
the subset  
$$\Gamma_p:=\{ C_{AB} \ | \  L(A,B)=p\}\subset\Sigma$$
is  called the {\it $p$-th string construction,} see \cite[p.73]{tab}.
\end{definition}
\begin{remark} For every $p>0$ small enough  $\Gamma_p$ is a well-defined smooth curve,  we set $\Gamma_0=\gamma$.  
The curve $\gamma$ is a caustic for the billiard transformation 
acting by reflection from the curve $\Gamma_p$: a line tangent to $\gamma$ is reflected from 
the curve $\Gamma_p$  to a line tangent to $\gamma$ \cite[theorem 5.1]{tab}. 
In Section 3 we will prove the following theorem.
\end{remark}
\begin{theorem} \label{thsm} Let $k\geq2$, $\Sigma$ be a $C^k$-smooth surface equipped 
with a $C^k$-smooth Riemannian metric, and let $\gamma\subset\Sigma$ be a germ of $C^k$-smooth curve at $O\in\Sigma$ with positive geodesic curvature. Let $\mcu\subset\Sigma$ denote the domain adjacent 
  to $\gamma$ from the concave side. For every $C\in\mcu$ let  
$\Lambda(C)\subset T_C\Sigma$ denote the one-dimensional subspace that  is the exterior bisector of the angle 
formed by the two geodesics through $C$ that are tangent to $\gamma$.
  Then the following statements hold.
  
1) The subspaces $\La(C)$ form a germ at $O$ of 
 line field $\La$ that is 
$C^{k-1}$-smooth on 
$\mcu$ and $C^{r(k)}$-smooth on $\overline{\mcu}$, 
$$r(k)=[\frac{k}2]-1.$$ 

2) The string  curves $\Gamma_p$ are tangent to  $\La$ and $C^{r(k)+1}$-smooth. Their 
$(r(k)+1)$-jets at base points $C$ depend 
continuously on $C\in\overline{\mcu}$. 
\end{theorem}

\begin{definition} \label{dpor} We say that a germ of oriented curve $\gamma\subset\Sigma$ with non-zero geodesic 
curvature has {\it string Poritsky property,} if it admits a 
$C^1$-smooth 
parametrization  by a parameter $t$  (called {\it Poritsky--Lazutkin string length}) such that 
for every $p>0$ small enough there exists a $c=c_p>0$  such that for every  pair $B,A\in\gamma$ 
ordered by  orientation with $L(A,B)=p$ one has $t(A)-t(B)=c_p$.
\end{definition}

\begin{example} \label{expor} It is classically known that 

(i) for every planar conic $\gamma\subset\rr^2$ and every $p>0$ the $p$-th string construction $\Gamma_p$ 
is a conic confocal to $\gamma$;

(ii) all the conics confocal to $\gamma$ and lying inside a given string construction conic 
$\Gamma_p$ are caustics of the billiard inside the conic $\Gamma_p$; 

(iii) each planar conic $\gamma$ has string Poritsky property  \cite[section  7]{poritsky},  \cite[p.58]{tab};  

(iv) conversely, {\it each planar curve with string 
Poritsky property is a conic}, by a theorem of H.Poritsky  \cite[section  7]{poritsky}. 
\end{example}

Two  results of the present paper extend statement (iv) to billiards on simply connected complete 
surfaces of constant curvature 
(by adapting Poritsky's arguments from \cite[section  7]{poritsky}) and to outer billiards on the latter surfaces. To state them, 
let us recall the notion of a conic on a surface of constant curvature.

Without loss of generality we consider  simply connected complete surfaces $\Sigma$ of constant curvature  0, $\pm1$ and realize each of them in its standard model in the space $\rr^3_{(x_1,x_2,x_3)}$ equipped with appropriate quadratic form 
$$<Qx,x>, \ Q\in\{\diag(1,1,0),\diag(1,1,\pm1)\}, \ < x,x >=x_1^2+x_2^2+x_3^2.$$

- Euclidean plane: $\Sigma=\{x_3=1\}$, $Q=\diag(1,1,0)$. 

- The unit sphere: $\Sigma=\{ x_1^2+x_2^2+x_3^2=1\}$, $Q=Id$.

- The hyperbolic plane: $\Sigma=\{ x_1^2+x_2^2-x_3^2=-1\}\cap\{ x_3>0\}$, $Q=\diag(1,1,-1)$. 

The metric of constant curvature on the surface $\Sigma$ under question is induced by the  quadratic form $<Qx,x>$. 
The {\it geodesics} on $\Sigma$ are its intersections with two-dimensional vector subspaces in $\rr^3$. The {\it conics} 
on $\Sigma$ are its intersections with quadrics $\{<Cx,x>=0\}\subset\rr^3$, where $C$ is a real symmetric $3\times 3$-matrix, see \cite{izm, veselov2}.  

\begin{proposition} \label{p-bir} On every surface of constant curvature each conic has string Poritsky property.
\end{proposition}

\begin{theorem} \label{t-bir} Conversely, on every surface of constant curvature each germ of 
$C^2$-smooth curve with string Poritsky property is a conic. 
\end{theorem}

Proposition \ref{p-bir} and Theorem \ref{t-bir} will be proved in Section 4. 

\begin{remark}
In the case, when the surface under question  is Euclidean plane, Proposition \ref{p-bir} was proved in  
\cite[formula (7.1)]{poritsky}, 
and  Theorem \ref{t-bir} 
was proved in \cite[section  7]{poritsky}. 
\end{remark}

\def\mcu{\mathcal U}
\subsection{Poritsky property for outer billiards and area construction}

Let $\gamma\subset\rr^2$ be a smooth strictly convex closed curve. Let $\mcu$ be the exterior connected component 
of the complement $\rr^2\setminus\gamma$. Recall that the {\it outer billiard map} $T:\mcu\to\mcu$ 
associated to the curve $\gamma$ acts as follows. Take  a point $A\in\mcu$. There are two tangent lines to $\gamma$ through 
$A$. Let $L_A$ denote the right tangent line (that is, the image of the line $L_A$ under a small clockwise rotation around 
the point $A$ is disjoint from the curve $\gamma$). Let $B\in\gamma$ denote its tangency point. 
By definition, the image $T(A)$ is the point of the line $L_A$ that is 
central-symmetric  to $A$ with respect to the point $B$.

It is well-known that if $\gamma$ is an ellipse, then the corresponding outer billiard map is {\it integrable:} 
that is, an  exterior neighborhood of the curve $\gamma$ 
is foliated by invariant closed curves for the outer billiard map so that $\gamma$ is a leaf of this foliation. 
 The analogue of Birkhoff Conjecture for the outer billiards, which was suggested by 
 S.Tabachnikov \cite[p.101]{tab08}, states 
 the converse: if $\gamma$ generates an integrable outer billiard, then it is an ellipse. Its polynomially integrable version was studied in \cite{tab08} and 
 recently solved in \cite{gs}. For a survey on outer billiards see 
 \cite{tab95, tab, tabdog} and references therein.

 For a given strictly convex smooth curve $\Gamma$ there exists a one-parametric family of curves $\gamma_p$  
 such that $\gamma_p$ lies in the interior component $\Omega$ of the complement $\rr^2\setminus\Gamma$, 
 and the curve $\Gamma$ is invariant under the outer billiard map $T_p$ generated  by $\gamma_p$. The curves 
 $\gamma_p$ are given by the following {\it area construction} analogous to the string construction.  
 Let $\mca$ denote the area of the domain $\Omega$. For every oriented line $\ell$ intersecting 
 $\Gamma$ let $\Omega_-(\ell)$ denote the connected component of the complement $\Omega\setminus \ell$ 
 for which $\ell$ is a negatively oriented part of boundary.  Let now $L$ be a class of parallel and co-directed oriented lines. 
 For every $p>0$, $p<\frac12\mca$,  let $L_p$ denote the oriented line representing $L$ that intersects 
 $\Gamma$ and such that $Area(\Omega_-(L_p))=p$. For every given $p$, 
 the lines $L_p$ corresponding 
 to different classes $L$ form a one-parameter family parametrized by the circle: the  azimuth of the line is the parameter. 
 Let  $\gamma_p$ denote the enveloping curve  of the latter family, and let $T_p$ denote the  outer billiard 
  map generated by $\gamma_p$. It is well-known that the curve $\Gamma$ 
 is $T_p$-invariant  for every $p$ as above 
 \cite[corollary 9.5]{tab}. See Fig. 2. 
 \begin{figure}[ht]
  \begin{center}
   \epsfig{file=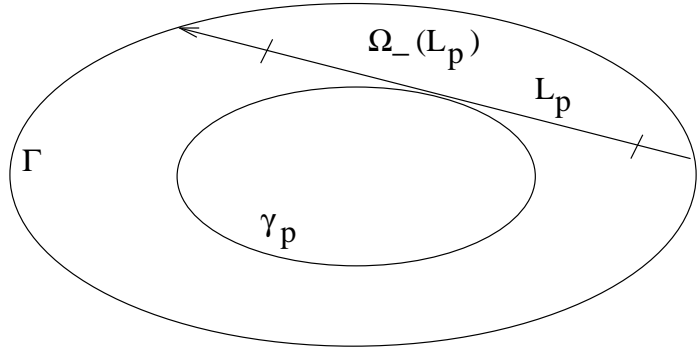}
    \caption{The area construction: $Area(\Omega_-(L_p))\equiv p$.}
    \label{fig:01}
  \end{center}
\end{figure}

 \begin{remark} For every $p>0$ small enough the curve $\gamma_p$ given by the area construction is 
 smooth. But for big $p$ it may have singularities (e.g., cusps). 
 \end{remark}
 
 For $\Gamma$ being an ellipse, all the  $\gamma_p$'s are ellipses homothetic to $\Gamma$ with respect 
 to its center. In this case there exists a parametrization of the curve $\Gamma$ by  circle $S^1=\rr\slash2\pi\zz$ 
 with parameter $t$ 
 in which $T_p:\Gamma\to\Gamma$ is a translation $t\mapsto t+c_p$ 
 for every $p$. This 
 follows from the area-preserving property of outer billiards, 
 see \cite[corollary 1.2]{tabcom}, and $T_q$-invariance of the ellipse $\gamma_p$ for  $q>p$, 
 analogously to the arguments in \cite[section  7]{poritsky}, \cite[the discussion before corollary 4.5]{tab}. Similar statements 
 hold for all conics, as in loc. cit.  
 
In our paper we prove the converse statement given by the following theorem, which will be stated in local context, 
for germs of smooth curves. To state it, let us introduce the following definition. 

\begin{definition} Let $\Sigma$ be a surface with smooth Riemannian metric, $O\in\Sigma$. 
Let $\Gamma\subset\Sigma$ be a 
germ of smooth strictly convex curve at a point $O$ (i.e., with positive 
 geodesic curvature). Let 
$U\subset\Sigma$ be a disk centered at $O$ that is split by $\Gamma$ into two components. One of these components 
is convex; let us denote it by $V$.  Consider 
the curves $\gamma_p$ given by the above area construction with $p>0$ small enough and lines replaced by geodesics. 
The curves $\gamma_p$ form a germ at $O$ of foliation  in the domain $V$, and its boundary curve $\Gamma=\gamma_0$ is 
a leaf of this foliation. We say that the curve $\Gamma$ has {\it area Poritsky property,} 
if it admits a local $C^1$-smooth 
parametrization by parameter $t$ called the 
{\it area Poritsky parameter} such that for every $p>0$ small enough the mapping 
$T_p:\Gamma\to\Gamma$ is a translation $t\mapsto t+c_p$ in the coordinate $t$. 
\end{definition}

\begin{proposition} \label{p-outer} (see \cite[lemma 3]{tabponc} for the hyperbolic case; \cite[lemma 5.1]{tabcom} for planar conics). 
On every surface of constant curvature each 
conic  has area Poritsky property.
\end{proposition} 

\begin{remark} (S.Tabachnikov) The area Poritsky property for conics on the 
sphere follows 
from their string Poritsky property and the fact that  the spherical 
outer billiards  are 
dual to the spherical Birkhoff billiards  \cite[subsection 4.1, lemma 5]{tab95}: 
the  duality is given by orthogonal polarity. Analogous duality  
holds on hyperbolic plane realized as the half-pseudo-sphere  of radius -1 
in 3-dimensional Minkovski space \cite[section 2, remark 2]{bm2}. 
\end{remark}

\begin{theorem} \label{t-outer} Conversely, on every surface of constant curvature each germ of 
$C^2$-smooth  curve with area Poritsky property 
 is a conic\footnote{For {\it planar} curves with area Poritsky property the statement of Theorem \ref{t-outer} for 
 $C^4$-smooth curves was earlier proved by Sergei 
 Tabachnikov (unpublished paper, 2018) by analytic arguments showing that the affine curvature of the curve should be constant. 
 In Section 5 we present a different, geometric proof, 
 analogous to Poritsky's arguments from \cite[section  7]{poritsky} 
 (which were given for Birkhoff billiards and string construction),  which works on all the surfaces of constant curvature simultaneously.}. 
\end{theorem}
 
\subsection{Coincidence of Poritsky and Lazutkin lengths}
Everywhere in the subsection $\Sigma$ is a two-dimensional surface equipped with a  $C^3$-smooth Riemannian metric. 

\begin{definition} Let $\gamma\subset\Sigma$ be a 
$C^2$-smooth curve, let $s$ be its natural length parameter. Let $\kappa(s)$ denote its geodesic curvature. 
 Fix a point in $\gamma$, let $s_0$ denote 
the corresponding length parameter value.  The parameter 
\begin{equation}t_L:=\int_{s_0}^s\kappa^{\frac23}(\zeta)d\zeta\label{lazpar}\end{equation}
is called the {\it Lazutkin parameter.} See 
\cite[formula (1.3)]{laz}.
\end{definition}
 
\begin{theorem} \label{ttk} Let $\gamma\subset\Sigma$ be a germ of 
$C^3$-smooth curve with positive geodesic 
curvature $\kappa$ and  string Poritsky property. 
Then its Poritsky string length parameter $t$ coincides 
with the Lazutkin parameter (\ref{lazpar}) up to additive 
and multiplicative constants. That is, up to constant factor one has
\begin{equation}\frac{dt}{ds}=\kappa^{\frac23}(s).\label{dtk}\end{equation}
\end{theorem}

A direct proof of Theorem \ref{ttk} will be presented in Section 6. It is 
based on the following theorem on asymptotics of the function $L(A,B)$ 
and its corollaries on string diffeomorphisms, also proved 
in the same section.

\begin{theorem} \label{lasyl} Let $\gamma\subset\Sigma$ be a  
$C^3$-smooth curve with positive geodesic curvature. 
For every $A\in\gamma$ let $s_A$ denote the corresponding 
natural length parameter value. Let $L(A,B)$ denote the quantity defined in (\ref{lab}). 
One has 
\begin{equation}L(A,B)=\frac{\kappa^2(A)}{12}|s_A-s_B|^3 (1+o(1)), 
\label{labas}\end{equation}
uniformly, as  $s_A-s_B\to0$ so that $A$ and $B$ remain in a compact 
subarc in $\gamma$. Asymptotic (\ref{labas}) is also uniform in 
the metric  running through a closed 
bounded subset in the space of $C^3$-smooth 
Riemannian metrics. 
\end{theorem}

 \begin{corollary} \label{ttk2} Let  $\gamma\subset\Sigma$ be a germ of 
 $C^3$-smooth curve  with positive geodesic curvature. For every small 
 $p>0$ let $\mct_p:\gamma\to\gamma$ denote the corresponding 
 string diffeomorphism (induced by reflection of geodesics tangent to $\gamma$ 
 from the string curve $\Gamma_p$ and acting on the tangency 
 points). For every points $A$ and $Q$ lying in a compact subarc 
 $\wh\gamma\Subset\gamma$ one has  
 \begin{equation}\kappa^{\frac23}(A)\la(A,\mct_p(A))\simeq\kappa^{\frac23}
 (Q)\la(Q,\mct_p(Q)), \text{ as } p\to0,\label{ka32la}
 \end{equation}
 uniformly in  $A,Q\in\wh\gamma$.  
 \end{corollary}
 \begin{corollary} \label{cttk2} In the conditions of Corollary \ref{ttk2} one has 
   \begin{equation}\kappa^{\frac23}(A)\la(A,\mct_p(A))\simeq\kappa^{\frac23}
 (\mct_p^m(A))\la(\mct_p^m(A),\mct_p^{m+1}(A)), \text{ as } p\to0,\label{ka32m}
 \end{equation}
uniformly in $A\in\wh\gamma$ and those $m\in\nn$ for which 
$\mct_p^m(A)\in\wh\gamma$. 
\end{corollary}
A symplectic generalization of Theorem \ref{ttk} to families of the so-called 
weakly billiard-like maps of string type will be presented in Section 7.

\subsection{Unique determination by 4-jet}
The next theorem is a Riemannian generalization of the classical 
fact stating that each planar conic is uniquely determined by 
its 4-jet at some its point.
\begin{theorem} \label{uniq4} Let $\Sigma$ be a surface equipped with a 
$C^6$-smooth Riemannian metric. A $C^5$-smooth germ of curve with string Poritsky property 
is uniquely determined by its 4-jet.
\end{theorem}

Theorem \ref{uniq4} will be proved in Section 8.

\begin{remark} In the case, when $\Sigma$ is the Euclidean plane, the statement of Theorem \ref{uniq4} 
follows from Poritsky's result \cite[section  7]{poritsky} (see statement (iv) of Example \ref{expor}). Similarly, in the case, when $\Sigma$ is a surface of constant curvature, 
the statement of Theorem \ref{uniq4} follows from Theorem \ref{t-bir}.
\end{remark}

\section{Background material from Riemannian geometry}
We consider curves $\gamma$ with positive geodesic curvature 
on an oriented  surface $\Sigma$ equipped with a Riemannian metric. 
In Subsection 2.1 we recall the notion of normal coordinates. 
We state and prove  equivalence of different definitions of  
geodesic curvature. 
One of these definitions deals with  geodesics tangent to 
$\gamma$ at close points $A$ and $B$ and the asymptotics of 
angle between them at their intersection point $C$. In the same subsection we prove 
existence of two geodesics tangent to $\gamma$ through every point $C$ close to 
$\gamma$ and lying on 
the concave side from $\gamma$; the corresponding tangency points will be denoted by 
$A=A(C)$ and $B=B(C)$. We also prove some technical statements on derivative  of azimuth of a vector tangent to a geodesic 
(Proposition \ref{pro1a}). In Subsection 2.2 we  prove 
 formulas for the derivatives 
$\frac{dA}{dC}$, $\frac{dB}{dC}$, which will be used in 
the proofs of Theorems \ref{thsm}, \ref{t-bir}, \ref{ttk}. In Subsection 2.3 we consider a pair of geodesics issued from the same point $A$ and their points $G(s)$, $Z(s)$ 
lying at a given distance $s$  to $A$. We  
give an asymptotic formula for difference of azimuths of their 
tangent vectors at  $G(s)$ and $Z(s)$, as $s\to0$. 
We will use this formula in the proof of Theorem \ref{uniq4}. 

\subsection{Normal coordinates and equivalent definitions of 
geodesic curvature}
Let $\Sigma$ be a two-dimensional surface equipped with a 
$C^3$-smooth Riemannian metric $g$.  Let $O\in\Sigma$. Let $\gamma$ be a $C^2$-smooth germ of curve  
at $O$ parametrized by its natural length parameter. Recall that  its geodesic curvature $\kappa=\kappa(O)$ equals the norm of 
the covariant derivative $\nabla_{\dot\gamma}\dot\gamma$. In the Euclidean case it coincides with  the inverse of the osculating circle radius. 

Consider the exponential 
chart $\exp:v\mapsto\exp(v)$ parametrizing a neighborhood of the point $O$ by a neighborhood of zero in the tangent plane $T_O\Sigma$. We introduce orthogonal  
linear coordinates $(x,y)$, on $T_O\Sigma$, which together with the exponential 
chart, induce {\it normal coordinates} centered at $O$, 
also denoted by $(x,y)$, on a neighborhood of the point $O$. 
It is well-known that in normal coordinates the metric has the same 1-jet 
at $O$, as 
the standard Euclidean metric (we then say that its {\it 1-jet is trivial} at $O$.) 
Its Christoffel symbols vanish at $O$. 

\begin{remark} \label{sigsmooth} Let the surface $\Sigma$ and the metric be $C^{k+1}$-smooth. Then normal coordinates are 
$C^{k}$-smooth. This follows from  theorem on dependence of 
solution of differential equation on initial condition (applied to the equation of geodesics) 
and $C^{k}$-smoothness of the Christoffel symbols. Thus,  each $C^{k}$-smooth curve is represented by a 
$C^k$-smooth curve in normal coordinates.
\end{remark}

\begin{proposition} \label{kappa=} For every curve $\gamma$  as above 
its geodesic curvature $\kappa(O)$ equals its Euclidean geodesic curvature $\kappa_e(O)$ 
in normal coordinates centered at $O$. 
If the normal coordinates $(x,y)$ are chosen so that the $x$-axis is tangent to $\gamma$, then 
$\gamma$ is the graph of a germ of function: 
\begin{equation} \gamma=\{ y=f(x)\}, \ \ \ f(x)=\pm\frac{\kappa(O)}2x^2+o(x^2), \text{ as } x\to0.\label{x122}\end{equation}
\end{proposition}

\begin{proof}  Proposition \ref{kappa=} 
 follows from definition and vanishing of the Christoffel symbols at $O$ in normal coordinates. 
 \end{proof}

\def\az{\operatorname{az}}
\def\dist{\operatorname{dist}}

\begin{proposition} \label{cabex} Let the germ $(\gamma,O)\subset\Sigma$ be the same as at the beginning of 
the subsection, and let $\gamma$ have positive geodesic curvature. Let $\mcu\subset\Sigma$ 
be the domain adjacent to $\gamma$ from the concave side: $\gamma$ is its concave boundary. 
Let $\widehat\gamma\subset\gamma$ be a compact subset: an arc with boundary.
For every $C\in\mcu$ close enough to $\widehat\gamma$ there exist exactly two geodesics 
through $C$ tangent to $\gamma$. In what follows we denote their tangency points with $\gamma$ by 
$A=A(C)$ and $B=B(C)$ so that $AC$ is the right geodesic through $C$ tangent  to $\gamma$. 
\end{proposition}
\begin{proof} The statement of the proposition is obvious in the Euclidean case. The non-Euclidean 
case is reduced to the Euclidean case by considering a point $C\in\mcu$ close to $\widehat\gamma$ 
and normal coordinates $(x_C,y_C)$ centered at $C$ so that their family depends smoothly on $C$. 
In these coordinates the curves $\gamma=\gamma_C$ depend smoothly on $C$ and 
are  strictly convex in the Euclidean sense, by Proposition \ref{kappa=}. 
The geodesics through $C$ are lines. This together with the statement of Proposition \ref{cabex} 
in the Euclidean case implies its statement in the non-Euclidean case. 
\end{proof}

Let us consider that 
$\Sigma$ is a Riemannian disk centered at $O$, the curve $\gamma$ splits $\Sigma$ into two open parts, and 
$\gamma$ has positive geodesic 
curvature. For every point $A\in\gamma$ the 
geodesic tangent to $\gamma$ at $A$ will be denoted by $G_A$. 

\begin{proposition} \label{pdiv} Taking the disk $\Sigma$ small enough, one can 
achieve that  for every $A\in\gamma$  the curve $\gamma$ lies in 
the closure of one and the same component of the complement 
$\Sigma\setminus G_A$, $\gamma\cap G_A=\{ A\}$. 
\end{proposition}
Proposition \ref{pdiv} follows its Euclidean version and  Proposition \ref{kappa=}. 

\begin{proposition} \label{pr24}  For every two points $A,B\in\gamma$ close enough to $O$ the geodesics $G_A$ and $G_B$ intersect at a 
unique point $C=C_{AB}\in\mcu$ close to $O$.
\end{proposition}

\begin{proof} Let $H$ denote the geodesic through $B$ orthogonal 
to $T_B\gamma$. It intersects the geodesic $G_A$ at  
some  point $P(A,B)\in\mcu$. The geodesic 
$G_B$ separates $P(A,B)$ from 
the punctured curve $\gamma\setminus\{ B\}$, by construction 
and Proposition \ref{pdiv}. Therefore, $G_B$ intersects the 
interval $(A,P(A,B))$ of the geodesic $G_A$. This proves the 
proposition.
\end{proof}

\begin{proposition} \label{gangle}
 For every $A,B\in\gamma$ close enough to $O$ let $C=C_{AB}$ denote 
the  point of intersection $G_A\cap G_B$. Let $\alpha(A,B)$ 
denote the acute angle between the  geodesics $G_A$ and $G_B$ at $C$, and let  $\la(A,B)$ denote the length of 
the arc $AB$ of the curve $\gamma$. The  geodesic curvature $\kappa(O)$ 
of the curve $\gamma$ at $O$ can be found  from any of the two 
following limits:
\begin{equation}\kappa(O)=\lim_{A,B\to O}\frac{\alpha(A,B)}{\la(A,B)};\label{ka}
\end{equation}
\begin{equation}\kappa(O)=\lim_{A,B\to O}2\frac{\dist(B,G_A)}{\la(A,B)^2}.\label{x22}
\end{equation}
\end{proposition}
 
 \begin{proof} In the Euclidean case formulas (\ref{ka}) and (\ref{x22}) are classical. Their non-Euclidean 
 versions follow by applying the Euclidean versions in normal coordinates centered respectively at $C$ 
 and at the point in $G_A$ closest to $B$,  as in the proof of Proposition \ref{cabex} (using 
 smoothness of family of representations of 
 the curve $\gamma$ in  normal coordinates with variable center).
 \end{proof}

For every point $A\in\Sigma$ lying in a chart $(x,y)$, e.g., a normal chart  centered at $O$, and 
 every tangent vector $v\in T_A\Sigma$ set  
$$\az(v):=\text{ the azimuth of the vector } v: \text{ its Euclidean 
angle with the } x-\text{axis},$$
i.e., the angle in the Euclidean metric in the coordinates $(x,y)$. 
The azimuth of an oriented one-dimensional subspace in $T_A\Sigma$ is defined analogously. 

\begin{proposition} \label{pro1a} Let $A\in\Sigma$ be a point close to $O$ and $\alpha(s)$  be a geodesic
through $A$ parametrized by the natural length parameter $s$,  
$\alpha(0)=A$.  

1) Let $\kappa_e(s)$ denote the Euclidean curvature of the geodesic $\alpha$ as a planar 
curve in normal chart $(x,y)$ centered at $O$. For every $\var>0$ small enough 
\begin{equation}\kappa_e(s)=O(\dist(\alpha,O)), \text{ as } A\to O, 
\text{ uniformly on } \{|s|\leq\var\},\label{kappas}
\end{equation}
$$\dist(\alpha,O):=\text{ the distance of the geodesic } \alpha \text{ to the point } O.$$
2) Set $v(s)=\dot\alpha(s)$. One has 
\begin{equation}\frac{d\az(v(s))}{ds}=O(\dist(\alpha,O))
=O(\angle(v(0),AO)\dist(A,O)) \text{ as } A\to O,
\label{azim}\end{equation}
uniformly on the set $\{|s|\leq\var\}$. The latter angle in (\ref{azim}) is the Riemannian angle between 
the vector $v(0)$ and the Euclidean line $AO$. 
\end{proposition}

\begin{proof}  In the coordinates $(x,y)$ the geodesics are solutions of the second order ordinary differential equation saying that 
$\ddot\alpha$ equals a quadratic form in $\dot\alpha$ with coefficients equal to appropriate Christoffel symbols 
of the metric $g$ (which vanish at $O$), and $|\dot\alpha|=1$ in the metric $g$. 
  The derivative in (\ref{azim}) is expressed in terms of the Christoffel symbols. 
This derivative taken along a geodesic $\alpha$ through $O$ vanishes identically 
on $\alpha$, since each geodesic through $O$ is a straight line  in normal coordinates. 
Therefore if we move the geodesic through $O$ out of $O$ by a small distance 
$\delta$, then  the derivative in (\ref{azim}) will change by an amount of order $\delta$: 
 the Christoffel symbols are $C^1$-smooth, since the metric is $C^3$-smooth 
 (hence, $C^2$-smooth in normal coordinates). 
This  implies the first  equality in (\ref{azim}). 
The second equality follows from the fact that  the 
geodesics through $A$ issued in the direction 
of the vectors  $\vec{AO}$ and $v(0)$ are 
respectively the  line $AO$ and  
$\alpha$, hence, $\dist(\alpha,O)=O(\angle(v(0),AO)\dist(A,O))$. 
This proves (\ref{azim}). 

Let $s_e$ denote the Euclidean natural parameter of the curve $\alpha$, with respect to the standard Euclidean 
metric in the chart $(x,y)$. 
Recall that $\kappa_e(s)=\frac{d\az(v(s))}{ds_e}$. For $\var>0$ small enough and $A$  close enough to $O$ 
the ratio   $\frac{ds_e}{ds}$ is uniformly bounded on $\{ |s|\leq\var\}$.   
This together with (\ref{azim}) implies (\ref{kappas}). The proposition is proved.
\end{proof} 

 \subsection{Angular derivative of exponential mapping 
 and the derivatives $\frac{dA}{dC}$, $\frac{dB}{dC}$}

In the proof of main results we will use an explicit formula for 
 the derivatives of the functions $A(C)$ and $B(C)$ from Proposition \ref{cabex}. To state it, let us 
introduce the following auxiliary functions. For every $x\in\Sigma$ set 
$$\psi(x,r):=\frac1{2\pi}\text{ (the length of circle of radius } r 
\text{ centered at } x).$$
Consider the 
polar  coordinates $(r,\phi)$ on the Euclidean plane 
$T_x\Sigma$. For every unit vector $v\in T_x\Sigma$, 
$|v|=1$ (identified with the corresponding angle coordinate $\phi$) and every 
$r>0$ let $\Psi(x,v,r)$ denote  $\frac1r$ times 
the module of derivative in $\phi$ of the exponential mapping at the point $rv$: 
\begin{equation}\Psi(x,v,r):=r^{-1}\left|\frac{\partial\exp}{\partial\phi}(rv)\right|.\label{depsi}\end{equation}

\begin{proposition}  (see \cite{bennett} in the hyperbolic case). Let 
$\Sigma$ be a complete simply connected Riemannian 
surface of constant curvature. Then 
\begin{equation}r\Psi(x,v,r)=\psi(x,r)=\psi(r)=\begin{cases}  r, & \text{ if } \Sigma \text{ is Euclidean plane,}\\
 \sin r, & \text{ if } \Sigma \text{ is unit sphere,}\\
 \sinh r, & \text{ if } \Sigma \text{ is hyperbolic plane.}\end{cases}\label{psix0}
\end{equation}
\end{proposition}
 
\begin{proof} The left equality in (\ref{psix0}) and independence 
of $x$ and $v$ follow from homogeneity. Let us prove the right inequality: formula for the function $\psi(r)$. 
In the planar case this formula is obvious.

a) Spherical case. Without loss of generality let us place the 
center $O$ of the 
circle under question to the north pole $(0,0,1)$ in the Euclidean coordinates 
$(x_1,x_2,x_3)$ on the ambient space. Since each geodesic is a big circle 
of length $2\pi$ and due to symmetry, without loss of generality we consider that 
$0<r\leq\frac{\pi}2$. Then the disk in $\Sigma$ centered at $O$ of radius $r$ is 
1-to-1 projected to the disk of radius $\sin r$ in the coordinate $(x_1,x_2)$-plane, and 
the  length of its boundary obviously equals the Euclidean length of the boundary 
of its projection, that is, $2\pi\sin r$. This proves statement a). 

b) Case of hyperbolic plane. We consider the hyperbolic plane in the model of 
unit disk equipped with the metric $\frac{2|dz|}{1-|z|^2}$ in the complex coordinate $z$. 
For every $R>0$, $R<1$ the Euclidean circle $\{ |z|=R\}$  of radius $R$ is a hyperbolic 
circle of radius 
$$r=\int_0^R\frac{2ds}{1-s^2}=\log\left|\frac{1+R}{1-R}\right|.$$
The hyperbolic length of the same circle equals $L=\frac{4\pi R}{1-R^2}$.
Substituting the former formula to the latter one yields 
$$R=\frac{e^r-1}{e^r+1}, \  L=2\pi\sh r$$
and finishes the proof of the proposition.
\end{proof} 

\begin{proposition} \label{pdersq}  Let $\gamma\subset\Sigma$ be a germ of $C^2$-smooth curve. 
 Let $s$ be the length parameter on $\gamma$ orienting it positively 
as a boundary of a convex domain. Let $\mcu\subset\Sigma$ be the 
concave domain adjacent to $\gamma$, see Proposition \ref{cabex}. 
For every $C\in\mcu$ let $A(C)$, $B(C)\in\gamma$ be 
the corresponding points from Proposition \ref{cabex}, and let 
$s_A=s_A(C)$, $s_B=s_B(C)$ denote the corresponding  
length parameter values as functions of $C$. Set 
$$L_A:=|CA(C)|, \ L_B:=|CB(C)|.$$
For every $Q=A,B$ let $w_Q\in T_Q\gamma$ be the  unit tangent 
vector of the geodesic $CQ(C)$ directed to $C$, and let 
$\zeta_Q\in T_C\Sigma$ denote the  unit tangent vector 
of the same geodesic at $C$ directed to $Q(C)$.  For every $v\in T_C\Sigma$ and every $Q=A,B$ one has 
\begin{equation}\frac{ds_Q}{dv}=
\frac{v\times \zeta_Q}{\kappa(Q)L_Q\Psi(Q,w_Q,L_Q)}; \ 
v\times\zeta_Q:=|v|\sin\angle(v,\zeta_Q),
\label{dersq}\end{equation}
where  $\angle(v,\zeta_Q)$ is the oriented angle between 
the vectors $v$ and $\zeta_Q$: it is positive, if the latter vectors form 
an orienting basis of the space $T_C\Sigma$. 
\end{proposition}

\begin{proof}  Let us prove (\ref{dersq}) for $Q=A$; the proof 
for $B$ is analogous. 
As $A=A(C)$ moves by $\var$ along the curve $\gamma$ 
to the point $A_\var$ with the natural parameter $s_A+\var$, 
 the geodesic $G_A$ tangent to $\gamma$ at $A$ 
 is deformed to the geodesic $G_{A_\var}$ intersecting $G_A$ 
 at a point converging 
 to $A$, as $\var\to0$. Let $\alpha(\var)$ denote their intersection angle at the latter point. One has 
 \begin{equation}\alpha(\var)\simeq\kappa(A)\var.
 \label{alfep}\end{equation}
Both above statements follow from  (\ref{ka}) and definition. One 
also has  
 \begin{equation}\dist(C,G_{A_\var})\simeq 
 \alpha(\var)L_A\Psi(A,w_A,L_A)\simeq\var\kappa(A)  
 L_A\Psi(A,w_A,L_A),\label{diagc}\end{equation} 
 by the definition of the function $\Psi$ and (\ref{alfep}). 
 
 Without loss of generality we consider that $v$ 
is a unit vector. Let us draw a curve $c$ through $C$ 
tangent to $v$ and oriented by $v$. Let $\tau$ denote 
its  natural parameter defined by this orientation. Let 
$C_{\var}$ denote the point of intersection of 
the geodesic $G_{A_\var}$ with $c$, see Fig. 3. Consider $\tau=\tau(C_\var)$ as a function of $\var$: 
$\tau=\tau(\var)$.  
 One has
\begin{figure}[ht]
  \begin{center}
   \epsfig{file=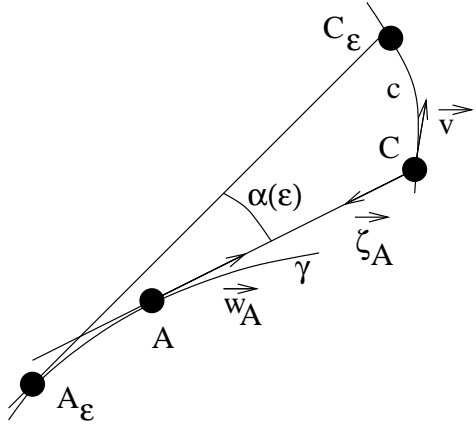}
    \caption{The tangent geodesics to $\gamma$ at points $A$ and $A_\var$ $\var$-close along the curve $\gamma$. The angle between them is $\alpha(\var)\simeq\kappa(A)\var$. One has
 $\dist(C,A_\var C_\var)\simeq \alpha(\var)L_A\Psi(A,w_A,L_A)$, $L_A=|AC|$. The length of the arc $CC_\var$ is asymptotic to $\dist(C,A_\var C_\var)\slash\sin\angle(v,\zeta_A)$.}
    \label{fig:01}
  \end{center}
\end{figure}
 
 $$\frac{ds_Q}{dv}=\left(\frac{d\tau}{d\var}(0)\right)^{-1}, \ \tau(C_{\var})-\tau(C)\simeq\frac{\dist(C,G_{A_\var})}{\sin\angle(v,\zeta_A)}
\simeq\var\frac{d\tau}{d\var}(0)=\var\left(\frac{ds_Q}{dv}\right)^{-1},$$
as $\var\to0$, which follows from  definition. Substituting (\ref{diagc}) to the latter formula yields (\ref{dersq}).
\end{proof}
\subsection{Geodesics passing through the same base point; azimuths of tangent vectors at equidistant points}

\begin{proposition} Let $\Sigma$ be a surface with a $C^3$-smooth Riemannian metric. 
Let $G_t(s),Z_t(s)\subset\Sigma$ be two families of 
geodesics parametrized by the natural length  $s$ and depending on a parameter $t\in[0,1]$. Let they be issued from the same point $A_t=G_t(0)=Z_t(0)$. Let $A_t$ lie in a given 
 compact subset (the same for all $t$) in a local chart $(x,y)$   
 (not necessarily a normal chart).  Set
 $$\phi_t=\az(\dot G_t(0))-\az(\dot Z_t(0)).$$ 
 One has
\begin{equation}\az(\dot G_t(s))-\az(\dot Z_t(s))\simeq \phi_t,  \text{ as } s\to 0, \text{ uniformly in } t\in[0,1].\label{azas}\end{equation}
\end{proposition}
\begin{proof} 
A geodesic, say, $G(s)$ is a solution of a  
second order vector differential equation with a 
given initial condition: a point $A\in\Sigma$ and the azimuth $\az(v(0))$ of a 
unit vector $v(0)\in T_A\Sigma$. Here we set $v(s)=\dot G(s)$. 
It depends smoothly on the initial condition. The derivative of the solution  in the 
initial conditions is a linear operator ($3\times 3$-matrix) function in $s$ that is a solution of the corresponding linear equation in variations. The  right-hand sides of the equation for geodesics and 
the  corresponding equation in variations are respectively 
$C^2$- and $C^1$-smooth. 
Let us now   
fix the initial point $A$  and consider the  derivative of the azimuth 
$\az(v(s))$ in the initial azimuth $\az(v(0))$ for fixed $s$. If $s=0$, 
then the 
latter derivative equals 1, since the  initial condition in 
the equation in variations is the  identity matrix. Therefore, 
in the general case the derivative of the azimuth $\az(v(s))$ in  $\az(v(0))$ equals 
$1+u(s)$, where $u(s)$ is a  $C^1$-smooth function with 
$u(0)=0$. This together with the above discussion and Lagrange Increment Theorem for the derivative in the initial azimuth implies  (\ref{azas}). 
\end{proof} 
 \subsection{Geodesic-curvilinear triangles   in normal coordinates}
 
 Everywhere below in the present subsection 
 $\Sigma$ is a two-dimensional surface equipped with a 
$C^3$-smooth Riemannian metric $g$, and $O\in\Sigma$.

 \begin{proposition} \label{rtri} Let $A_uB_uC_u$ be a 
family of geodesic right triangles lying in a compact subset in $\Sigma$ with right angle $B_u$. Set 
$$c=c_u=|A_uB_u|, \ b=b_u=|A_uC_u|, \ a=a_u=|B_uC_u|, \ \alpha=\alpha_u=\angle B_uA_uC_u.$$
Let   $b_u, \alpha_u\to0$, as $u\to u_0$. Then 
\begin{equation}b\simeq c, \ b-c\simeq\frac{a^2}{2c}\simeq\frac12c\alpha^2\simeq\frac12a\alpha, \ 
\angle B_uC_uA_u=\frac{\pi}2-\alpha+o(\alpha).\label{abc}\end{equation}
\end{proposition}
\begin{proof} 
Consider  normal coordinates $(x_u,y_u)$ centered at $A_u$ (depending 
smoothly on the base point $A_u$). The coordinates 
$$(X_u,Y_u):=\left(\frac{x_u}{c_u},\frac{y_u}{c_u}\right)$$ 
are normal coordinates centered at $A_u$ for the Riemannian metric rescaled by division by $c_u$. 
For the rescaled metric one has $|A_uB_u|=1$. In the rescaled normal coordinates 
$(X_u,Y_u)$ the  rescaled metric has trivial 1-jet at 0 
and tends to the Euclidean metric, as $u\to u_0$: its nonlinear part  tends to zero, as $u\to u_0$, uniformly 
on the Euclidean disk of radius 2 in the coordinates $(X_u,Y_u)$. One has obviously 
$|A_uB_u|\simeq|A_uC_u|$ in the rescaled metric, since $\alpha_u\to0$.  Rescaling back, we get the first asymptotic formula 
in (\ref{abc}). 

Let $S_u$ denote the  circle  
of radius $|A_uB_u|$ centered at $A_u$, and let $D_u$ denote its point lying on the 
geodesic $A_uC_u$: $|A_uB_u|=|A_uD_u|$; the arc $B_uD_u$ of the circle $S_u$ is its intersection with the 
geodesic angle $B_uA_uC_u$. In the rescaled coordinates $(X_u,Y_u)$ 
the circle $S_u$ tends  to the Euclidean unit circle. Thus, its geodesic curvature in the rescaled metric tends to 1. 
The geodesic segment $B_uC_u$ is tangent  to $S_u$ at the point $B_u$, and 
$\angle B_uC_uA_u\to\frac{\pi}2$. The two latter statements together with Proposition \ref{kappa=} 
(applied to $O=B$ and $\gamma=S_u$) imply that 
in the rescaled metric one has $|B_uC_u|\simeq\alpha$,
$$|D_uC_u|=|A_uC_u|-|A_uB_u|\simeq\frac{|B_uC_u|^2}2\simeq\frac12\alpha^2\simeq\frac12|B_uC_u|\alpha.$$
Rescaling back to the initial metric, we get the second, third and fourth 
formulas in (\ref{abc}). The fifth 
formula   follows from Gauss--Bonnet Formula, which 
implies  that the sum of angles in the triangle $A_uB_uC_u$ 
differs from $\pi$ by a quantity $O(Area)=O(|A_uB_u||B_uC_u|)=
O(\alpha|A_uB_u|^2)=o(\alpha)$. 
\end{proof}

\begin{proposition} Consider a family of $C^3$-smooth arcs $\gamma_u=A_uB_u$ of curves in $\Sigma$ 
(lying in a compact set) with uniformly bounded geodesic curvature such that $|A_uB_u|\to0$, 
as $u\to0$. Let $\la(A_u,B_u)$ denote their lengths. Let $\alpha_u$ denote 
the angle at $A_u$ between the arc $\gamma_u$ and the geodesic 
segment $A_uB_u$. One has 
\begin{equation}\la(A_u,B_u)=|A_uB_u|+O(|A_uB_u|^3), \ \alpha_u=O(|A_uB_u|).\label{labu}\end{equation}
\end{proposition}
\begin{proof} The proposition obviously holds  in Euclidean metric. It remains valid in the normal 
coordinates centered at $A_u$ with the geodesic $A_uB_u$ being the $x$-axis. Indeed, the 
length of the arc $\gamma_u$ in the Euclidean metric in the normal chart differs from its Riemannian length by 
a quantity $O(|A_uB_u|^3)$, since the difference of the metrics at a point $P\in\gamma_u$   is 
$O(|PA_u|^2)=O(|A_uB_u|^2)$ and the curvature of the arcs $\gamma_u$ is bounded.
\end{proof} 

\begin{proposition} \label{ptric} Consider a family of curvilinear triangles $T_u:=A_uB_uC_u$ in $\Sigma$ 
where the side $A_uB_u$ is geodesic and the sides $A_uC_u$, $B_uC_u$ are arcs of 
$C^3$-smooth curves  with uniformly bounded geodesic curvature. Let the side $A_uC_u$ be 
tangent to the side $A_uB_u$ at $A_u$. Set
$$\var:=|A_uB_u|, \ \theta:=\frac\pi2-\angle A_uB_uC_u.$$
Let the triangles under question lie in a compact subset in $\Sigma$, 
and let $\var$ and $\theta$ tend to zero, as $u\to0$. Then 
\begin{equation}\la(A_u,C_u)-|A_uB_u|=O(\var^3)+O(\var^2\theta).\label{tricurve}\end{equation}
\end{proposition}
\begin{proof} One has $|B_uC_u|=O(\var^2)$, which follows from 
construction and the definition of the angle $\theta$. Therefore, $|A_uC_u|\simeq \var$ and the 
distance of the point $C_u$ to the geodesic $A_uB_u$ is $O(\var^2)$.  Let $D_u$ denote the point 
closest to $C_u$ in the latter geodesic: the points $A_u$, $C_u$ and $D_u$ form a right triangle $\Delta_u$ 
with right angle at $D_u$. One has  
\begin{equation}|C_uD_u|=O(\var^2), \ \la(A_u,C_u)-|A_uC_u|=O(\var^3),\label{oe3}\end{equation}
 by definition and (\ref{labu}),   
 \begin{equation}  |A_uC_u|-|A_uD_u|=O(\frac{|C_uD_u|^2}{|A_uB_u|})
=O(\var^3), \label{oe32}\end{equation}
by (\ref{oe3}) and (\ref{abc}) applied to $\Delta_u$. Let us show that 
 \begin{equation}|A_uD_u|-|A_uB_u|=|B_uD_u|=O(\var^2\theta)+O(\var^3).\label{ova3}\end{equation}
In  the  right triangle $\wh\Delta_u$ with vertices $B_u$, $C_u$, $D_u$ 
one has $\angle D_uB_uC_u=\frac\pi2-\theta+O(\var^2)$. Indeed, the latter angle is the sum (difference) of the two following angles at $B_u$: 
the angle $\frac{\pi}2-\theta$ of the triangle $T_u$;  
the angle between the geodesic 
$B_uC_u$ and the curved side $B_uC_u$ in $T_u$, 
which is $O(|B_uC_u|)=O(\var^2)$, by (\ref{labu}). 
This implies the above formula 
for the angle $\angle D_uB_uC_u$, which 
in its turn implies that in the triangle $\wh\Delta_u$ one has  
 $\angle B_uC_uD_u=O(\theta)+O(\var^2)$ (the last formula in (\ref{abc})). 
 The latter formula together with (\ref{oe3}) imply  (\ref{ova3}). 
Adding formulas  (\ref{oe3}), (\ref{oe32}),  (\ref{ova3}) 
yields (\ref{tricurve}).
\end{proof}

\section{Smoothness of string foliation. Proof of Theorem 
\ref{thsm}}

\subsection{Finite smoothness lemmas}

Everywhere below in the present section 
we are dealing with a function  
$f(x,y)$ of two variables 
$(x,y)$: the variable $y$ is scalar, and 
the variable $x$ may be a vector variable. The function 
$f$ is supposed to be defined on the product 
$$Z=\overline U\times V$$ 
of closure of a domain 
$U$ in $x$-variable and  an interval $V=(-\var,\var)\subset\rr$ in 
$y$-variable.

The following two basic smoothness lemmas will be used 
in the proof of smoothness of the line field $\Lambda$. 

\begin{lemma} \label{lsm1} Lef a function $f$ as above 
be $C^k$-smooth on $Z$, $k\geq2$, and let 
\begin{equation} f(x,y)=a(x)y^2(1+o(1)), \text{ as } 
y\to 0, \ \text{ uniformly in } x\in\overline U; \ \ a>0 \text{ on } 
\overline U.
\label{quadas}\end{equation}
Then the function $g(x,y):=\sign(y)\sqrt{f(x,y)}$ is 
$C^{k-1}$-smooth on $Z$. 
\end{lemma}

\begin{lemma} \label{lsm2} 
Let a function $f(x,y)$ as at the beginning 
of the section be $C^k$-smooth on $Z$ and even in $y$: 
$f(x,y)=f(x,-y)$. Then $g(x,z):=f(x,\sqrt z)$ is 
$C^{[\frac k2]}$-smooth on $Z$, and its restriction to $Z\setminus\{ y=0\}$ is $C^k$-smooth. 
\end{lemma}

In the proof of the lemmas for simplicity 
without loss of generality 
we consider that the variable $x$ is one-dimensional; 
in higher-dimensional case the proof is the same. 
 We use the following 
definition and a more precise version of the asymptotic 
Taylor formula for finitely-smooth functions. 
\begin{definition} Let $l,m\in\zz_{\geq0}$. 
 We say that 
$$f(x,y)= o_l(y^m), \text{ as } y\to0,$$
if for every $j,s\in\zz_{\geq0}$, $j\leq l$, $s\leq m$ 
 the derivative $\frac{\partial^{j+s}f}{\partial^jx\partial^sy}$ 
 exists and is continuous on $\overline U\times V$ and 
  one has 
\begin{equation} 
\frac{\partial^{j+s}f}{\partial^jx\partial^sy}(x,y)=o(y^{m-s}), 
\text{ as } y\to 0, \text{ uniformly in } x\in\overline U.
\label{defofol}\end{equation}
\end{definition}

\begin{proposition} Let $f(x,y)$ be as at the beginning 
of the section, and let $f$ be $C^k$-smooth on $Z$. 
Then for every $l,m\in\zz_{\geq0}$ with  
$l+m\leq k$ one has 
\begin{equation}f(x,y)=f(x,0)+\sum_{j=1}^ma_j(x)y^j+R_{m}(x,y), \ 
a_j(x)=\frac1{j!}\frac{\partial ^j f}{\partial y^j}(x,0)\in C^{l}(\overline U), \label{asform}\end{equation}
$$R_{m}(x,y)=o_l(y^m), \text{ as } y\to 0, \text{ uniformly in } x\in\overline U.$$
\end{proposition}
\begin{proof} The first formula in (\ref{asform}) holds with 
\begin{equation} R_{m}(x,y)=\int_{0\leq y_m\leq\dots\leq y_1\leq y}(\frac{\partial^m}{\partial y^m}f(y_{m})-\frac{\partial^m}{\partial y^m}f(0))dy_mdy_{m-1}\dots
dy_1,\label{rmex}\end{equation}
by the classical asymptotic Taylor formula 
with error term in integral presentation. For example, 
$$f(x,y)-f(x,0)=\int_0^yf'_y(x,\eta)d\eta=yf'_y(x,0)+
\int_0^y(f'_y(x,\eta)-f'_y(x,0))d\eta,$$
etc. Now it remains to notice that the expression 
(\ref{rmex}) is $o_l(y^m)$, whenever $f\in C^k$ and 
$k\geq l+m$. The proposition is proved.
\end{proof}

\begin{proposition} One has 
\begin{equation} y^{-s}o_l(y^m)=o_l(y^{m-s}) \text{ for 
every } \ m,s\in\zz_{\geq0}, \ m\geq s.\label{prol1}
\end{equation}
\end{proposition}
The proposition follows from definition.

\begin{proof} {\bf of Lemma \ref{lsm1}.} The function 
$g(x,y):=\sign(y)\sqrt{f(x,y)}$ is well-defined, by (\ref{quadas}). 
It is obviously $C^k$-smooth outside the hyperplane 
$\{ y=0\}$. Fix arbitrary 
$l,m\in\zz_{\geq0}$ such that $l+m\leq k-1$. 
Let us prove continuity of the 
derivative $\frac{\partial^{l+m}g}{\partial x^l\partial y^m}$ 
on $Z$. 

Case $m=0$; then $k\geq l+1$. The above derivative is a linear combination of 
expressions 
\begin{equation}
\sign(y) f^{\frac12-s}(x,y)\prod_{j=1}^s\frac{\partial^{n_j}f(x,y)}{\partial x^{n_j}}, \ s\in\nn, \ n_j\geq1,\  \sum_{j=1}^sn_j=l.
\label{f12sn}\end{equation}
The  partial derivatives in (\ref{f12sn}) are $C^1$-smooth, since $f$ is $C^k$-smooth and 
$n_j\leq l\leq k-1$. One has 
 \begin{equation}\sign(y)f^{\frac12-s}(x,y)\simeq a^{\frac12-s}(x)y^{1-2s},\label{1-2s}\end{equation}
 by definition. If $s=1$, then $y^{1-2s}=y^{-1}$, the expression 
 (\ref{f12sn}) contains only one  derivative, and this derivative 
 is asymptotic to $y$ times a continuous function in $x$, as $y\to0$, by smoothness and since 
 $f(x,0)\equiv0$. Therefore,  the expression 
 (\ref{f12sn}) is continuous. If $s\geq2$, then $n_j\leq l-1\leq k-2$. Hence, each derivative in (\ref{f12sn}) 
 is $C^2$-smooth, has vanishing first derivative in $y$ at $y=0$ 
  and is asymptotic to $y^2$ times a continuous function in $x$, and (\ref{f12sn}) is 
 again continuous, by (\ref{1-2s}). 
 
 Case $m=1$ is treated analogously with the following change: one of the derivatives in (\ref{f12sn}) 
 will contain one differentiation in $y$ and will be asymptotic to $y$ times a continuous function in $x$. 
 
Case $m\geq2$. Then $k\geq m+l+1\geq l+3$. One has 
\begin{equation}g(x,y)=a^{\frac12}(x)y\sqrt{w(x,y)},\label{3.8}\end{equation}
\begin{equation}w(x,y)=
1+\sum_{j=3}^{m+1}a^{-\frac12}(x)a_j(x)y^{j-2}+\frac{o_l(y^{m+1})}{y^2}, \ \ a, a_j\in C^l(\overline U),\label{a12x}\end{equation} 
 by (\ref{asform}) applied to the function $f(x,y)$ and $m$ replaced by $m+1$.
The derivative $\frac{\partial^{l+m-1}g}{\partial x^{l}\partial y^{m-1}}$ exists and continuous for small $y$, 
by (\ref{a12x}) and since $\frac{o_l(y^{m+1})}{y^2}=o_l(y^{m-1})$, see (\ref{prol1}). Now it remains to 
prove the same statement for the derivative $h:=\frac{\partial^{l+m}g}{\partial x^{l}\partial y^{m}}$. 
Those terms in its expression that include the derivatives of the function $\frac{o_l(y^{m+1})}{y^2}=o_l(y^{m-1})$ 
with differentiation in $y$ of orders less than $m$ 
  are  well-defined and continuous, as above. Each  term in $h$ that contains a derivative 
  $\frac{\partial^{j+m}}{\partial x^j\partial y^m}(\frac{o_l(y^{m+1})}{y^2})$ contains only one such derivative, and it comes with the factor $y$ from (\ref{3.8}). 
 On the other hand, the latter derivative is $\frac{o_{l-j}(y)}{y^2}=o(\frac1y)$, by (\ref{prol1}). Thus, its 
 product with the above factor $y$ is a continuous function, as are the other factors in 
 the term under question. Continuity of the derivative $h$ is proved. Lemma \ref{lsm1} is proved.
 \end{proof}

 \begin{proof} {\bf of Lemma \ref{lsm2}.} Fix $l,m\in\zz_{\geq0}$ such that $l+m\leq [\frac{k}2]$. 
 Then $l+2m\leq k$, and one has 
 $$f(x,y)=\sum_{j=0}^ma_j(x)y^{2j}+o_l(y^{2m}),$$
 where the functions $a_j(x)$ are $C^l$-smooth, 
 by (\ref{asform}) and evenness. Set $z=y^2$. The derivative 
 $\frac{\partial^{l+m}}{\partial x^l\partial z^m}$ 
 of the sum in the latter right-hand side is obviously continuous, since the sum is a polynomial in $z$ with 
 coefficients  being $C^l$-smooth functions in $x$. 
 Let us prove continuity of the derivative of the remainder 
 $o_l(y^{2m})$. One has 
 $$\frac{\partial}{\partial z}=\frac1{2y}\frac{\partial}{\partial y}.$$
 Therefore, the above $(l+m)$-th partial derivative of 
 the remainder $o_l(y^{2m})$ is $o(y^{\kappa})$, 
 $\kappa=2m-2m=0$, see (\ref{prol1}). Thus, it is $o(1)$. 
 This proves continuity and Lemma \ref{lsm2}.
 \end{proof} 
 
 \subsection{Proof of Theorem \ref{thsm}}
 
 The fact that the exterior bisector line field $\La$ is tangent to the string construction curves 
 is well-known and proved as follows. Consider the value $L(A,B)=|AC|+|CB|-\la(A,B)$ as a function of 
 $C$: here $A=A(C)$ and $B=B(C)$ are the same, as in Proposition \ref{cabex}. Its derivative along 
  the string construction curve $\Gamma_p$ through $C$ 
 should be zero. Let $v\in T_C\Sigma$ be a unit vector. 
 Let $\alpha$ and $\beta$ be respectively the oriented angles between the vector $v$ and 
 the  vectors $\zeta_A$ and $\zeta_B$ in $T_C\Sigma$ directing the geodesics $G_A$, $G_B$ from $C$ 
 to $A$ and $B$ respectively. The derivative of the above function $L(A(C),B(C))$ along the vector $v$ 
 is equal to $-(\cos\alpha+\cos\beta)$. Therefore, it vanishes if and only if the line generated by $v$ is the 
 exterior bisector $\La(C)$ of the angle $\angle ACB$. Therefore, the level sets of the function $L(A(C),B(C))$, 
 i.e., the string construction curves  are integral curves of the line field $\La$.
   
 It suffices to prove only statement 1) of Theorem \ref{thsm}: 
 $C^{k-1}$-smoothness on $\mcu$ and $C^{r(k)}$-smoothness 
 on $\overline\mcu$ of the line field $\La$. Statement 2) on $C^{r(k)+1}$-regularity of its integral curves (the string construction curves) and continuity 
 then follows from  the  next general fact: {\it for every $C^r$-smooth 
 line field  the $(r+1)$-jets  of its integral 
 curves at base points $A$ are expressed 
 analytically in terms of $r$-jets of the line field, and 
 hence, depend continuously on $A$.}
 
Fix a $C^k$-smooth coordinate system 
 $(s,z)$ on $\Sigma$ centered at the base point $O$ 
 of the curve $\gamma$ such that $\gamma$ is the 
 $s$-axis, 
 $s|_{\gamma}$ is the natural length parameter of
  the curve $\gamma$ and 
$\mcu=\{ z>0\}$. For every $\sigma\in\rr$ small 
  enough let $G(\sigma)$ denote the geodesic 
  tangent to $\gamma$ at the point with length parameter 
  value $\sigma$. For every $\sigma,s\in\rr$ small enough 
  let $A(\sigma,s)$ denote the point of intersection 
  of the geodesic $G(\sigma)$ with the line parallel to 
  the $z$-axis and having abscissa $s$.  The mapping 
  $(\sigma,s)\mapsto A(\sigma,s)$ is  
  $C^{k-1}$-smooth, since so is the family of geodesics 
  $G(\sigma)$ (by $C^k$-smoothness of the metric) and by transversality. Set 
  $$z(\sigma,s):=z(A(\sigma,s)).$$
  
  \begin{proposition} \label{propk-2}
  The function 
  $$y(\sigma,s):=\sign(\sigma-s)\sqrt{z(\sigma,s)}$$ 
  is $C^{k-2}$-smooth on a neighborhood of zero in 
  $\rr^2$ and $C^{k-1}$-smooth outside the diagonal 
  $\{\sigma=s\}$. The mapping 
  \begin{equation}F:(\sigma,s)\mapsto(s,y(\sigma,s))
  \label{deff}\end{equation}
  is a  $C^{k-2}$-smooth diffeomorphism of a neighborhood of the origin onto a neighborhood of the origin, and it 
  is $C^{k-1}$-smooth outside the diagonal. It sends the diagonal to the axis $\{y=0\}$.
  \end{proposition}
  
  \begin{proof} For every point $Q\in\Sigma$ lying in a smooth chart $(s,z)$ let   
  $u(Q)$ denote the orthogonal projection of the vector $\frac{\partial}{\partial z}\in T_Q\Sigma$ 
  to  the line   $(\rr\frac{\partial}{\partial s})^\perp$. Set 
 $\mu(Q):=||w||^{-1}$. Recall that  $\kappa(s)>0$. One has 
  \begin{equation}z(\sigma,s)=\frac12\mu(s,0)\kappa(s)(s-\sigma)^2+o((s-\sigma)^2), \text{ as } \sigma\to s, 
  \label{kssig}\end{equation} 
  uniformly in small $s$. This follows from formula (\ref{x22}) applied to normal 
  coordinates centered at $Q=(s,0)$ with one coordinate axis being tangent to 
  $\dot\gamma(Q)=\frac{\partial}{\partial s}$. This together with 
  $C^{k-1}$-smoothness of the function $z$ and 
  Lemma \ref{lsm1} implies the statements of the 
  proposition. 
  \end{proof}
  
Let us now return to the proof of statement 1) of 
Theorem \ref{thsm}. 
Consider the mapping inverse to the mapping $F$ from 
(\ref{deff}):
$$F^{-1}:(s,y)\mapsto(\sigma,s).$$
The function $\sigma=\sigma(s,y)$ is $C^{k-2}$-smooth, by 
Proposition \ref{propk-2}, and it is $C^{k-1}$-smooth outside 
the axis $\{ y=0\}$. 
Recall that the geodesic $G(\sigma(s,y))$ passes through 
the point $A=(s,z)=(s,y^2)\in\mcu$. For every $s$ and $y$ 
let $v=v(s,y)\in T_{A}\Sigma$ denote the unit tangent 
vector of the geodesic 
$G(\sigma(s,y))$  that 
 orients it in the same way, as the orienting 
 tangent vector of the curve $\gamma$ at $\sigma(s,y)$. 
 The vector function $v(s,y)$ is $C^{k-2}$-smooth 
 in $(s,y)$. For a given point $A=(s,z)$, set 
 $y:=\sqrt z$, the unit vectors $v(s,y)$ and $v(s,-y)$ both 
 lie in $T_A\Sigma$. They are tangent to the two 
 geodesics through $A$ that are tangent to $\gamma$, 
 by construction.   The sum $w(s,y)=v(s,y)+v(s,-y)$ 
 is a vector 
 generating the line $\Lambda(A)$ of the line field 
 $\Lambda$, which follows from definition. The vector function $w(s,y)$ is even 
 in $y$, 
 $C^{k-2}$-smooth in both variables, and $|w|=
 2|v|=2$, whenever $y=0$. Thus, $w$ is 
 $C^{[\frac k2]-1}$-smooth in $(s,z)$ and $C^{k-1}$-smooth outside the curve $\gamma=\{ z=0\}$, by 
 Proposition \ref{propk-2} and Lemma \ref{lsm2}. Finally,  $w$ induces a vector field generating $\La$ that is 
 $C^{[\frac k2]-1}$-smooth on $\overline \mcu$ and 
 $C^{k-1}$-smooth on $\mcu$. Theorem \ref{thsm} is proved.

\section{Billiards on surfaces of constant curvature. Proofs of Proposition 
\ref{p-bir} and Theorem \ref{t-bir}}

In Subsection 4.1 we prove Proposition \ref{p-bir}. The proof of Theorem \ref{t-bir}, 
which  follows 
its  proof given in \cite[section  7]{poritsky} in the Euclidean case,  
takes the rest of the section. In Subsection 4.2 we prove the following coboundary property of curves with string Poritsky property:  
for every $A,B\in\gamma$, set $C=C_{AB}$, the ratio of lengths of the geodesic segments 
$AC$ and $BC$ equals the ratio of values at $A$ and $B$ of some function 
on $\gamma$. In Subsection 4.3 we deduce Theorem \ref{t-bir} from the 
coboundary property by 
arguments of  elementary planimetry  using Ceva's Theorem. 

\subsection{Proof of Proposition \ref{p-bir}}

We re-state and prove Proposition \ref{p-bir} in a more general Riemannian context. To do this, let us recall the 
following definition.

\begin{definition} \label{devol} \cite[p. 345]{amiran} (implicitly considered in \cite{poritsky}) 
Let $\Sigma$ be a surface equipped with a Riemannian metric, 
 $\gamma\subset\Sigma$ be a (germ of) curve with positive geodesic curvature. Let $\Gamma_p$ denote 
 the family of curves obtained from it by string construction. We say that $\gamma$ has 
 {\it evolution} (or {\it Graves}) {\it property,} if for every $p_1<p_2$ the curve $\Gamma_{p_1}$ is a caustic for the curve $\Gamma_{p_2}$. 
 \end{definition}
 
 \begin{example} \label{excon} It is well-known that each conic on a surface $\Sigma$ of constant curvature has evolution property, and the corresponding curves $\Gamma_p$ given by string construction are confocal conics. In the Euclidean case 
 this follows from the classical fact saying that the caustics of a billiard in a conic are confocal conics (Proclus--Poncelet Theorem). Analogous statements hold in non-zero constant curvature and in higher dimensions, 
 see \cite[theorem 3]{veselov2}.  
\end{example}
\begin{proposition} \label{pev} Let $\Sigma$ be a 
surface equipped with a $C^4$-smooth Riemannian metric. Let $\gamma\subset\Sigma$ be a $C^4$-smooth 
germ of curve with positive geodesic curvature. Let $\gamma$ have evolution property. Then it has string Poritsky property\footnote{Very recently it was shown in a joint paper of the author with Sergei Tabachnikov and Ivan Izmestiev \cite{four} that for a $C^{\infty}$-smooth curve $\gamma$ 
 {\it evolution property is} {\bf equivalent} {\it to Poritsky property}. 
And that it is also  equivalent to the  
statement that the foliation by the curves $\Gamma_p$ and its orthogonal foliation 
form a {\it Liouville net} on the concave side $\mcu$  from 
the curve $\gamma$.}. For every $p,q>0$ the reflections from the corresponding curves 
$\Gamma_p$ and $\Gamma_q$ commute as transformations acting on the space of oriented geodesics intersecting both of them, disjoint from the curve $\gamma$  and lying on the concave side $\mcu$ 
 from the 
curve $\gamma$. 
\end{proposition}

\begin{remark} In the Euclidean case the first part of Proposition \ref{pev} with a proof is  contained in 
\cite{poritsky, amiran}. Commutativity then follows by arguments 
from \cite[chapter 3]{tab}. The proof of the first part of 
Proposition \ref{pev}
given below is analogous to arguments from 
\cite{poritsky}, \cite[ch.3]{tab}. 
The analogue of evolution property for outer billiards was introduced and 
studied by E. Amiran \cite{amiran2}. 
\end{remark}

\begin{proof} {\bf of Proposition \ref{pev}.} Let  $O$ denote the base point of the germ $\gamma$. 
The string  curves $\Gamma_p$ form a  foliation tangent to 
a line field $\La$ that is $C^3$-smooth on 
the concave side $\mcu$ and $C^1$-smooth on $\overline{\mcu}$ (Theorem  \ref{thsm}).

Consider the billiard reflections from the curves $\Gamma_p$ acting on the manifold of oriented geodesics. They 
preserve a canonical symplectic area form $\omega$ on the latter space, one and the same for all the reflections. See
 \cite[section 3]{poritsky}, \cite[chapter 3]{tab} in the planar case; in the general case the 
 form $\omega$ is given by Melrose construction, see  
\cite[section 1.5]{tab95}, \cite{melrose1, melrose2, ar2, ar3}. 
For every $p$ let $\Gamma_p^*$ denote the family or  geodesics tangent to $\Gamma_p$ and 
oriented as $\Gamma_p$. 
For every $p_1<p_2$ the curve $\Gamma_{p_1}^*$ is invariant under 
the reflection $T_{p_2}$ from the curve $\Gamma_{p_2}$ (evolution property). The curves $\Gamma_p^*$ form a germ of foliation $F$ 
in the space of oriented geodesics; 
its base point  represents the geodesic tangent to $\gamma$ at $O$. 
The foliated domain, which consists of points representing 
the geodesics tangent to $\Gamma_p$, $p>0$, will be denoted 
by $\mcu^*$. The curve $\Gamma_0^*$ lies in its boundary and is a leaf of the foliation $F$.   The 
 foliation $F$ is $C^3$-smooth on $\mcu^*$ and $C^1$-smooth on its closure, as is the line field $\La$. In more detail, consider the mapping 
 of the set $\overline\mcu$ to the space of geodesics that sends each point 
 $Q\in\overline\mcu$ to the geodesic tangent to $\La(Q)$. This is a 
 mapping of the same regularity, as $\La$. The foliation $F$ is the 
 image of the  foliation by string curves $\Gamma_p$, and hence, also 
 has the above regularity.  Thus, $F$ is the 
 foliation by level curves of a function $\psi$ of the same regularity and without critical points. 
  The function $\psi$ is $T_p$-invariant for all $p$, by construction. Therefore,  its 
 Hamiltonian vector field $H_{\psi}$ is also invariant and tangent to the curves $\Gamma_q^*$. 
 Hence, for every $q<p$ the reflection 
 $T_p:\Gamma_q^*\to\Gamma_q^*$ acts by translation in the time coordinate 
 $t_q$ of the field $H_\psi$ on $\Gamma_q^*$, and this also holds for $q=0$. 
 The time coordinate $t_0$ on $\Gamma_0^*$ induces a parameter, also denoted 
 by $t_0$, on the curve $\Gamma_0=\gamma$. Therefore, $\gamma$   
 has Poritsky property with Poritsky--Lazutkin parameter $t_0$, by the above discussion. 
 Any two reflections 
 $T_p$ and $T_q$ commute while acting on the union of the curves $\Gamma_r^*$ with 
 $r\leq\min\{ p,q\}$,  since the latter curves are invariant and $T_p$, $T_q$  
 act as translations there. Proposition \ref{pev} is proved.
 \end{proof}
 
 Proposition \ref{p-bir} follows from Proposition \ref{pev} and Example \ref{excon}.

\subsection{Preparatory coboundary property of length ratio} 
Let $\Sigma$ be an oriented surface of constant 
curvature $K\in\{0,\pm1\}$: either Euclidean plane, or unit sphere in 
$\rr^3$, or hyperbolic plane. Let $O\in\Sigma$, and let $\gamma\subset\Sigma$ 
be a regular germ of curve through $O$ with positive geodesic curvature. We consider that 
$\gamma$ is oriented clockwise with respect to the orientation of the surface $\Sigma$.  
For every point $X\in\gamma$ by $G_X$ we denote the geodesic 
tangent to $\gamma$ at $X$.  
Let $A,B\in\gamma$ be two distinct points close to $O$ such that the curve $\gamma$ 
be oriented from $B$ to $A$. Let  $C=C_{AB}$ denote the unique intersection point of 
the geodesics $G_A$ and $G_B$  that is close to $O$. (Then $CA$ is the right geodesic 
tangent to $\gamma$ through $C$.) 
Set 
$$L_A:=|CA|; \ L_B:=|CB|;$$
here $|CX|$ is the length of the geodesic arc $CX$. 
Recall that we denote 
\begin{equation}\psi(x)=\begin{cases}  x, & \text{ if } \Sigma \text{ is Euclidean plane,}\\
 \sin x, & \text{ if } \Sigma \text{ is unit sphere,}\\
 \sinh x, & \text{ if } \Sigma \text{ is hyperbolic plane.}\end{cases}\label{psix}
\end{equation}

\begin{proposition} \label{coboundary} Let $\Sigma$ be as above, $\gamma\subset\Sigma$ 
be a germ of $C^2$-smooth curve at a point $O\in\Sigma$ with string Poritsky property. 
There exists a positive smooth function $u(X)$, $X\in\gamma$, such that for every 
$A,B\in\gamma$ close enough to $O$ one has 
\begin{equation} \frac{\psi(L_A)}{\psi(L_B)}=\frac{u(A)}{u(B)}.\label{psiab}
\end{equation}
\end{proposition}

\begin{proof} 
For every $p>0$ small enough and every $C\in\Gamma_p$ close enough to $O$ there are two geodesics issued from the 
point $C$ that are tangent to $\gamma$ (Proposition \ref{cabex}). 
The corresponding tangency points $A=A(C)$ and $B=B(C)$ in $\gamma$ 
depend smoothly on the point $C\in\Gamma_p$.  
Let $s_p$ denote the natural length parameter 
of the curve $\Gamma_p$. We set $s=s_0$: the natural length parameter of the curve 
$\gamma$. We write  $C=C(s_p)$, and  consider the natural 
parameters $s_A(s_p)$, $s_B(s_p)$ of the points $A(C)$ and 
$B(C)$ as functions of $s_p$. Let $\alpha(C)$ denote the oriented 
angle between a vector $v\in T_C\Gamma_p$ 
 orienting the curve 
$\Gamma_p$ and a vector $\zeta_A\in T_CG_A$ 
directing the geodesic $G_A$ from $C$ to $A$. 
 It is equal  (but with opposite sign) to the oriented angle between 
the  vector $-v$ and a  vector $\zeta_B\in T_CG_B$ directing the geodesic 
$G_B$ from $C$ to $B$, since the 
tangent line to $\Gamma_p$ at $C$ is the exterior bisector 
of the angle between the geodesics $G_A$ and $G_B$ 
(Theorem \ref{thsm}). One has 
\begin{equation}
\frac{ds_A}{ds_p}=\frac{\sin\alpha(C)}{\kappa(A(C))\psi(|AC|)}, \ 
\frac{ds_B}{ds_p}=\frac{\sin\alpha(C)}{\kappa(B(C))\psi(|BC|)},
\label{dsabp}\end{equation}
by (\ref{dersq}),  (\ref{psix0}) and the above angle equality. 

Let now $t$ be the Poritsky parameter of the curve $\gamma$.  
Let  $t_A(s_p)$ and $t_B(s_p)$ denote its values at  
the points $A(C)$ and $B(C)$ respectively as functions of $s_p$. 
 Their difference is 
constant, by Poritsky property. Therefore, 
$$\frac{dt_A}{ds_p}=\frac{dt}{ds}(A)\frac{ds_A}{ds_p}=
\frac{dt_B}{ds_p}=\frac{dt}{ds}(B)\frac{ds_B}{ds_p}.$$
Substituting (\ref{dsabp}) to the latter formula and cancelling 
out $\sin\alpha(C)$ yields (\ref{psiab}) with 
$$u=\frac{\nu}{\kappa}, \ \nu:=\frac{dt}{ds}.$$
\end{proof}

\subsection{Conics and Ceva's Theorem on surfaces of constant curvature. 
Proof of Theorem \ref{t-bir}} 

\begin{definition} \label{dti} Let $\Sigma$ be a surface with Riemannian metric. 
We say that a germ of curve $\gamma\subset\Sigma$ at a point $O$ with non-zero geodesic curvature 
has {\it tangent incidence property,} if the following statement holds. Let $A',B',C'\in\gamma$ 
be arbitrary three distinct points close enough to $O$. 
Let $a$, $b$, $c$ denote the geodesics tangent to 
$\gamma$ at $A'$, $B'$, $C'$ respectively. Let $A$, $B$, $C$ denote the points of intersection $b\cap c$, $c\cap a$, $a\cap b$ respectively. Then the geodesics 
$AA'$, $BB'$, $CC'$ intersect at one point. See \cite[p.462, fig.5]{poritsky} and Fig. 4 
below. 
\end{definition}
\begin{figure}[ht]
  \begin{center}
   \epsfig{file=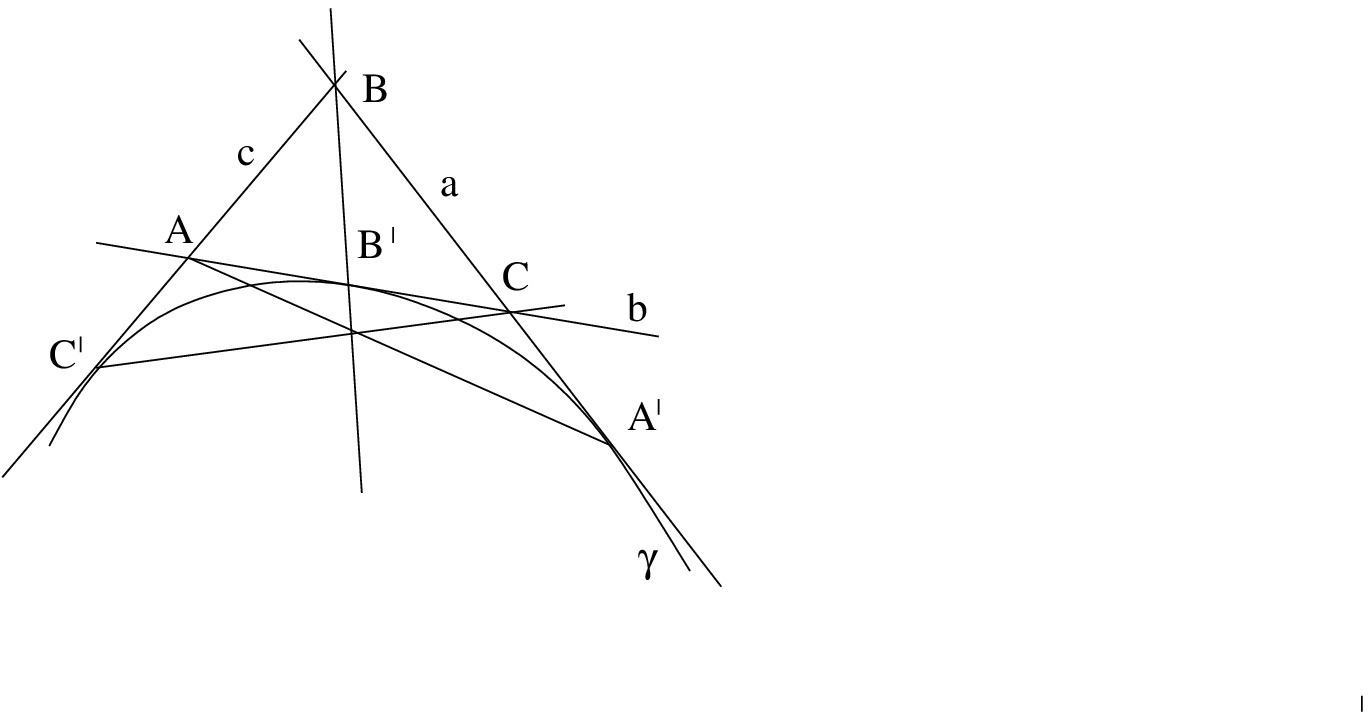}
    \caption{A curve $\gamma$ with tangent incidence property}
    \label{fig:0}
  \end{center}
\end{figure}

\begin{proposition} \label{propinc} Every germ of $C^2$-smooth curve  
with string Poritsky property on a surface 
of constant curvature has tangent incidence property.
\end{proposition}

As it is shown below,  Proposition \ref{propinc} follows from Proposition \ref{coboundary} 
and the next theorem. 

\begin{theorem} \label{ceva} \cite[pp. 3201--3203]{masal} 
(Ceva's Theorem on surfaces of constant curvature.) 
Let $\Sigma$ be a simply connected complete 
surface of constant curvature. Let $\psi(x)$ be the corresponding function 
in (\ref{psix}): the length of circle of radius $x$ divided by $2\pi$. 
Let $A,B,C\in\Sigma$ be three 
distinct points. Let $A'$, $B'$, $C'$ be respectively some points on 
the {\bf sides} $BC$, $CA$, $AB$ of the geodesic triangle $ABC$. 
Then the geodesics 
$AA'$, $BB'$, $CC'$ intersect at one point, if and only if 
\begin{equation}\frac{\psi(|AB'|)}{\psi(|B'C|)}\frac{\psi(|CA'|)}{\psi(|A'B|)}
\frac{\psi(|BC'|)}{\psi(|C'A|)}=1.\label{fce}\end{equation}
\end{theorem}
{\bf Addendum to Theorem \ref{ceva}.} {\it Let now in the conditions of Theorem 
\ref{ceva} $A'$, $B'$, $C'$ be points on the} {\bf geodesics} {\it $BC$, $CA$, $AB$ 
respectively so that some two 
of them, say $A'$, $C'$ do not lie on the corresponding} {\bf sides} {\it and the 
remaining third point $B'$ lies on the corresponding side $AC$, see Fig. 4. 

1) In the Euclidean and 
spherical cases the geodesics $AA'$, $BB'$, $CC'$ intersect at the same point, 
if and only if (\ref{fce}) holds. 

2) In the hyperbolic case (when $\Sigma$ is of negative curvature) the geodesics 
$AA'$, $BB'$, $CC'$ intersect at the same point, 
if and only if some two of them intersect and (\ref{fce}) holds.  

3) Consider  the 
standard model of the hyperbolic plane $\Sigma$  in the Minkovski space $\rr^3$, see Subsection 1.1. Consider the 2-subspaces defining 
the geodesics $AA'$, $BB'$, $CC'$, and let us denote the corresponding 
projective lines (i.e., their tautological projections to $\rp^2$) by $\mca$, $\mcb$, $\mcc$ 
respectively. The projective lines $\mca$, $\mcb$, $\mcc$ intersect at one point 
(which may  be not the projection of a point in $\Sigma$), if and only if (\ref{fce}) holds.}

\begin{proof} Statements 1) and 2) of the addendum follow from Theorem \ref{ceva} by analytic extension, when 
some two points $A'$ and $C'$ go out of the corressponding sides $BC$, $BA$ while remaining on the same (complexified) geodesics $BC$, $BA$. Statement 3) is proved analogously. 
\end{proof}

\begin{proof} {\bf of Proposition \ref{propinc}.} Let $O$ be the base point of the germ $\gamma$, and let $A'$, $B'$, $C'$  be 
its three subsequent points close enough to $O$. Let $a$, $b$, $c$ be 
respectively the geodesics tangent to $\gamma$ at them. Then each pair of the latter 
geodesics intersect at one point close to $O$. Let $A$, $B$, $C$  be the points of 
intersections $b\cap c$, $c\cap a$, $a\cap b$ respectively. The point $B'$ lies on the 
geodesic arc $AC\subset b$. 
This follows from the assumption that the point $B'$ lies between $A'$ and $C'$ on the curve $\gamma$ and the inequality $\kappa\neq0$. In a similar way we get that the points 
$A'$ and $C'$ lie on the corresponding geodesics $a$ and $c$ but outside the sides 
$BC$ and $AB$ of the geodesic triangle $ABC$ so that $A$ lies between $C'$ and $B$, and $C$ lies between $A'$ 
and $B$.  The geodesics $BB'$ and $AA'$ intersect, by the two latter arrangement 
statements. Let $u:\gamma\to\rr$ be the function from Proposition \ref{coboundary}. One has 
$\frac{\psi(|BA'|)}{\psi(|BC'|)}=\frac{u(A')}{u(C')}$, by (\ref{psiab}), 
and similar equalities hold with $B$ replaced by  $A$ and $C$. Multiplying  the 
three latter equalities we get (\ref{fce}), since  the right-hand side cancels out.  Hence the geodesics $AA'$, $BB'$ and $CC'$ 
intersect at one point, by statements 1), 2) of the addendum to Theorem \ref{ceva}. 
Proposition \ref{propinc} is proved.
\end{proof}

\begin{theorem} \label{incon} Each conic on a surface of constant curvature has 
 tangent incidence property. Vice versa, each $C^2$-smooth curve on  a surface of constant 
 curvature that has  tangent incidence property is a conic.
 \end{theorem} 
 
 \begin{proof} The first, easy statement of the theorem follows from 
 Propositions \ref{p-bir} and \ref{propinc}.  The proof of its second statement 
 repeats the arguments from \cite[p.462]{poritsky}, which are given in the 
 Euclidean case but remain valid in the other cases of constant curvature without 
 change. Let us repeat them briefly in full generality for completeness of presentation. 
 Let $\gamma$ be a germ of curve with tangent incidence property on a surface 
 $\Sigma$ of constant curvature. Let $A'$, $B'$, $C'$ denote three distinct 
 subsequent points of the curve $\gamma$, and let $a$, $b$, $c$ be respectively 
 the geodesics tangent to $\gamma$ at these points. Let $A$, $B$, $C$ denote 
 respectively the points of intersections $b\cap c$, $c\cap a$, $a\cap b$. 
 Fix the points $A'$ and $C'$. Consider the pencil $\mcc$ of conics through $A'$ and $C'$ 
 that are tangent to $T_{A'}\gamma$ and $T_{C'}\gamma$. Then each point of the surface 
 $\Sigma$ lies in a unique conic in $\mcc$ 
 (including two degenerate conics: the double geodesic 
 $A'C'$; the union of the geodesics $G_{A'}$ and $G_{C'}$). Let $\phi\in\mcc$ denote the conic 
  passing through the point $B'$. 
 
{\bf Claim.} {\it The tangent line $l=T_{B'}\phi$ coincides with $T_{B'}\gamma$.}

\begin{proof} Let $L$ denote the geodesic through $B'$ tangent to 
$l$. Let $C_1$ and $A_1$ denote respectively the points of intersections 
$L\cap a$ and $L\cap c$. Both curves $\gamma$ and $\phi$ have tangent incidence property. Therefore, the three geodesics $AA'$, $BB'$, $CC'$ intersect at the same point 
denoted $X$, and the three geodesics $A'A_1$, $BB'$, $C'C_1$ intersect at the same point 
$Y$; both $X$ and $Y$ lie on the geodesic $BB'$. We claim that this is impossible, if 
$l\neq T_{B'}\gamma$  (or equivalently, if $L\neq b$). Indeed, let to the contrary, 
$L\neq b$. Let  us turn the geodesic $b$ continuously towards $L$ in the family of 
geodesics $b_t$ through $B'$, $t\in[0,1]$: $b_0=b$, $b_1=L$, the azimuth of 
the line $T_{B'}b_t$ turns monotonously (clockwise or counterclockwise), as $t$ 
increases. Let $A_t$, $C_t$ denote respectively the points of the intersections 
$b_t\cap c$ and $b_t\cap a$: $A_0=A$, $C_0=C$. Let $X_t$ denote the point of the intersection of the geodesics $A'A_t$ and $C'C_t$: $X_0=X$, $X_1=Y$. 
At the initial position, when $t=0$, the point $X_t$ lies on the fixed geodesic $BB'$. 
As $t$ increases from 0 to 1, the points $A$ and $C$ remain fixed, while 
the points $C_t$ and $A_t$  move monotonously, so that as $C_t$ moves 
towards (out from) $B$ along the geodesic $a$, the point $A_t$ moves out from 
(towards) $B$ along the geodesic $c$, see Fig. 5. 
In the first case, when $C_t$ moves towards 
$B$ and $A_t$ moves out from $B$, the point $X_t$ moves out of the geodesic $BB'$, 
to the half-plane bounded by $BB'$ that contains $A$, and its distance to $BB'$ 
increases. Hence, $Y=X_1$ does not lie on $BB'$. The second case is 
treated analogously. The contradiction thus obtained proves the claim.
\begin{figure}[ht]
  \begin{center}
   \epsfig{file=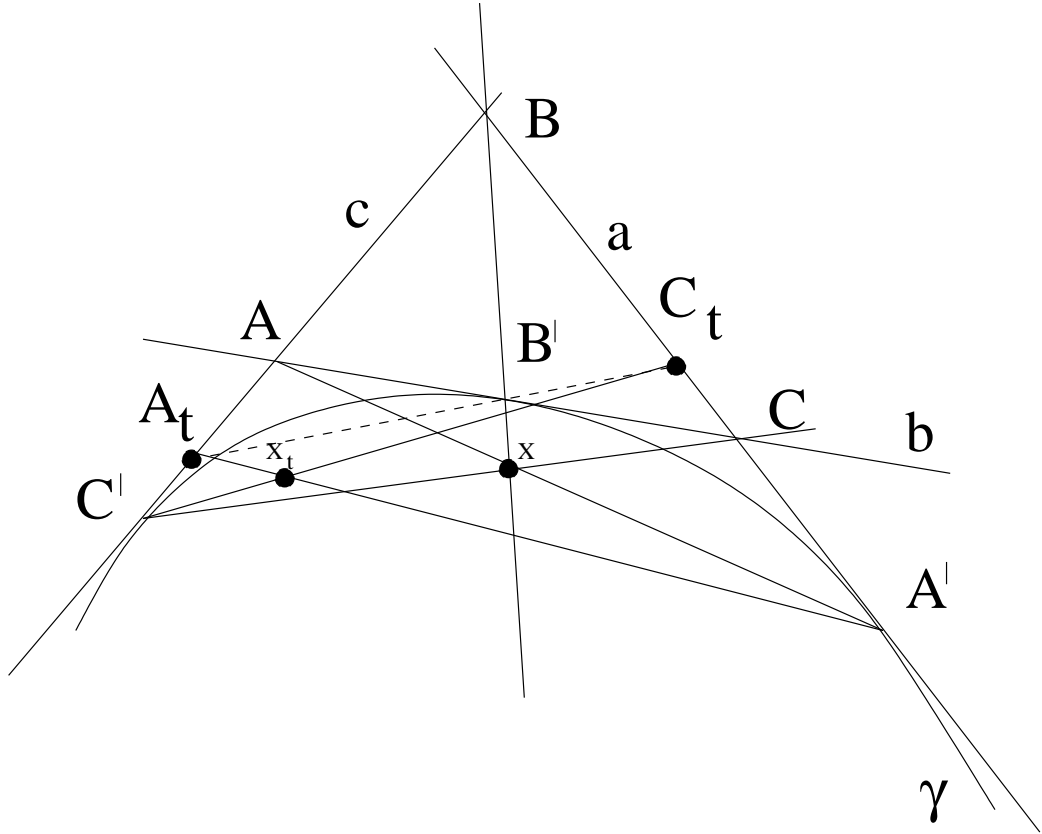}
    \caption{The intersection point $X_t$ moves away from the geodesic $BB'$.}
    \label{fig:0}
  \end{center}
\end{figure}
\end{proof} 

For every point $Q\in\Sigma$ such that the conic $\phi_Q\in\mcc$ passing through $Q$ 
is regular, set $l_Q:=T_Q\phi_Q$. The lines $l_Q$ form an analytic line field outside 
the union of three geodesics: $G_A'$, $G_C'$,  $A'C'$. Its 
 phase curves are the conics from the pencil $\mcc$. The curve $\gamma$ 
is also tangent to the latter line field, by the above claim. Hence, $\gamma$ is a conic. 
This proves Theorem  \ref{incon}.
 \end{proof}
 
 \begin{proof} {\bf of Theorem \ref{t-bir}.} Let $\gamma$ be a germ of $C^2$-smooth 
 curve with string Poritsky property on a  surface of constant curvature. Then it has 
 tangent incidence property, by Proposition \ref{propinc}. Therefore, it is a conic, 
 by Theorem \ref{incon}. Theorem \ref{t-bir} is proved.
 \end{proof}

\section{Case of outer billiards: proof of Theorem \ref{t-outer}}
Everywhere below in the present section $\Sigma$ is a simply connected 
complete Riemannian surface of constant curvature, and  
$\gamma\subset\Sigma$ is a germ of $C^2$-smooth curve at a point $O\in\Sigma$ 
with non-zero geodesic curvature.

\begin{proposition} \label{angco} Let $\Sigma$, $O$, $\gamma$ be 
as above, and let $\gamma$ have area Poritsky property. Then  there exists a continuous function 
$u:\gamma\to\rr_+$ such that for every 
$A,B\in\gamma$ close enough to $O$ the following statement holds. Let $\alpha$, 
$\beta$ denote the angles between the chord $AB$ and the curve $\gamma$ at the 
points $A$ and $B$ respectively. Then 
\begin{equation} \frac{\sin\alpha}{\sin\beta}=\frac{u(A)}{u(B)}.\label{psial}
\end{equation}
Let $t$, $s$ denote respectively the area Poritsky and length parameters of the curve $\gamma$. The above statement holds for 
the function 
 $$u:=t'_s=\frac{dt}{ds}.$$ 
\end{proposition}
\begin{proof} Recall that for every $C,D\in\gamma$ by $\lambda(C,D)$ we denote 
the length of the arc $CD$ of the curve $\gamma$.  Fix $A$ and $B$ as above. Set $A(0)=A$, $B(0)=B$. 
For every small $\tau>0$ let  
$A(\tau)$ denote the point of the curve $\gamma$ 
such that $\la(A(\tau),A(0))=\tau$ and the curve $\gamma$ is oriented from $A(0)$ to 
$A(\tau)$. Let $B(\tau)$ denote the family of points of the curve $\gamma$ 
such that the area of the domain bounded by the chord $A(\tau)B(\tau)$ and 
the arc $A(\tau)B(\tau)$ of the curve $\gamma$ remains constant, independent on $\tau$. 
For every $\tau$ small enough the chord $A(\tau)B(\tau)$ intersects the chord $A(0)B(0)$ at 
a point $X(\tau)$ tending to the middle of the chord $A(0)B(0)$, see Fig. 6. This follows from 
constance of area and homogeneity (constance of curvature) of the surface $\Sigma$. 
 \begin{figure}[ht]
  \begin{center}
   \epsfig{file=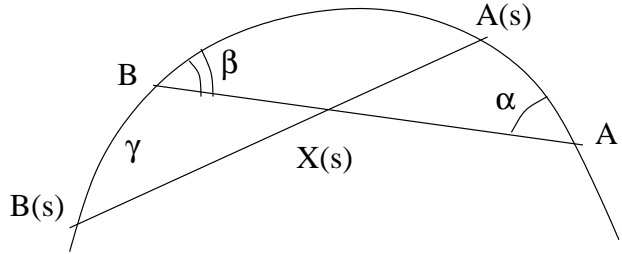}
    \caption{Curve $\gamma$ with area Poritsky property. The chords $AB$,  
    $A(\tau)B(\tau)$.}
    \label{fig:8}
  \end{center}
\end{figure} 
 One has 
$$t(A(\tau))-t(A(0))=t(B(\tau))-t(B(0)) \text{ for every } \tau \text{ small enough,}$$
by area Poritsky property. The above left- and right-hand sides are asymptotic to $u(A)\la(A(0),A(\tau))$ 
and $u(B)\la(B(0),B(\tau))$ respectively, as $\tau\to0$, with $u=\frac{dt}{ds}$. Therefore, 
\begin{equation} \frac{\la(B(0),B(\tau))}{\la(A(0),A(\tau))}\to\frac{u(A)}{u(B)}, \text{ as } s\to0.\label{larat}\end{equation}
The length $\la(A(0),A(\tau))$ is asymptotic to $\frac1{\sin\alpha}$ times  $\dist(A(\tau),B(0)A(0))$: the distance  
of the point $A(\tau)$ to the geodesic $X(\tau)A(0)=B(0)A(0)$. Similarly, $\la(B(0),B(\tau))\simeq\frac1{\sin\beta}\dist(B(\tau),B(0)A(0))$, as $\tau\to0$. The above distances of the points 
$A(\tau)$ and $B(\tau)$ to the geodesic $A(0)B(0)$ are asymptotic to each other, since the intersection point 
$X(\tau)$ of the chords $A(\tau)B(\tau)$ and $A(0)B(0)$ tends to the middle of the chord $A(0)B(0)$ and by homogeneity. This implies 
that the left-hand side in (\ref{larat}) tends to the ratio $\frac{\sin\alpha}{\sin\beta}$, 
as $\tau\to0$. This together with (\ref{larat}) proves (\ref{psial}).
\end{proof} 
\begin{proposition} \label{areat} 
Let $\Sigma$, $O$ and $\gamma$ be as at the beginning of the 
section. Let there exist a function $u$ on $\gamma$ that satisfies (\ref{psial}) for 
every $A,B\in\gamma$ close to $O$. Then $\gamma$ has tangent incidence property, 
see Definition \ref{dti}. 
\end{proposition}
\begin{proof} Let $A'$, $B'$, $C'$ be three subsequent points of the curve $\gamma$.  
Let $a$, $b$, $c$ denote respectively the geodesics tangent to $\gamma$ at these points.  
Let $A$, $B$, $C$ denote respectively the points of intersections $b\cap c$, $c\cap a$, 
$a\cap b$ (all the points $A'$, $B'$, $C'$, and hence $A$, $B$, $C$ are close enough to 
the base point $O$), as at Fig. 4. Let $\psi$ be the same, as in (\ref{psix}). One has 
\begin{equation}\frac{\sin\angle CA'B'}{\sin\angle CB'A'}=\frac{\psi(|CB'|)}{\psi(|CA'|)}=
\frac{u(A')}{u(B')}, 
\label{angles}\end{equation}
by (\ref{psial}) and Sine Theorem 
on the Euclidean plane and its analogues for unit sphere and hyperbolic plane applied 
to the geodesic triangle $CA'B'$, see \cite[p.215]{klein}, \cite[theorem 10.4.1]{sos}. 
Similar equalities hold for other pairs of points $(B',C')$, 
$(C',A')$. Multiplying all of them yields relation (\ref{fce}):
the ratios of values of the function $u$ at $A'$, $B'$, $C'$ cancel out. This together 
with Theorem \ref{ceva} and its addendum implies that $\gamma$ has tangent indicence 
property and proves Proposition \ref{areat}. 
\end{proof}

\begin{proof} {\bf of Theorem \ref{t-outer}.} Let $\gamma$ be a curve with area Poritsky property on a surface of constant curvature. Then it has tangent incidence property, 
by Propositions \ref{angco} and \ref{areat}. Hence, it is a conic, by Theorem \ref{incon}. 
Theorem \ref{t-outer} is proved. 
\end{proof}

\section{The function $L(A,B)$ and  the Poritsky--Lazutkin parameter.  
Proofs of Theorems \ref{lasyl}, \ref{ttk} and  Corollaries \ref{ttk2}, 
\ref{cttk2}}

\begin{proof} {\bf of Theorem \ref{lasyl}.} Let $g$ denote the metric. 
 Let $C=C_{AB}$ denote the point of intersection of the geodesics 
$G_A$ and $G_B$ tangent to $\gamma$ at the points $A$ and $B$ 
respectively. We will work in normal coordinates $(x,y)$ centered at $C$ 
and the corresponding polar coordinates $(r,\phi)$. 
\medskip

{\bf Claim 1.} {\it The length $s_A-s_B$ 
of the arc $AB$ of the curve $\gamma$ differs 
from its Euclidean length in the coordinates $(x,y)$ by a quantity 
$o((s_A-s_B)^3)$. The same statement also holds for the quantity 
$L(A,B)$.}

\begin{proof} 
It is known that 
the metric $g$ is $O(r^2)$-close to the Euclidean metric, and 
the polar coordinates are $g$-orthogonal. In the polar coordinates  $g$ has the same 
radial part $dr^2$, as the Euclidean metric $dr^2+r^2d\phi^2$, 
and their angular parts differ by a quantity   
$\Delta=O(r^2)r^2d\phi^2=O(r^4d\phi^2)$. Let us write the $g$-length of the arc 
$AB$ as the integral over its Euclidean length parameter of the $g$-norm  
of the Euclidean-unit tangent vector field to $\gamma$. The contribution 
of the above  difference $\Delta$ to the latter integral is bounded 
from above by the integral $I$ of a quantity $O(r^2\alpha)$, where $\alpha$ 
is the angle of a tangent vector $\dot\gamma(Q)$ with the 
radial line $CQ$. Set $\delta:=|s_A-s_B|$. 
The arc $AB$ lies in a 
$O(\delta)$-neighborhood of the point $C$. 
It is clear that the distance of the arc $AB$ to 
the origin $C$ is of order $O(\delta^2)$. Note that those points 
in the arc $AB$ where $\alpha$ is bounded away from zero are on distance 
$O(\delta^2)$ from the origin $C$. Therefore, $\alpha=o(1)$, as 
$\delta\to0$, uniformly on the complement of the arc $AB$ to the 
disk $D_{\delta^{\frac32}}$ 
of radius $\delta^{\frac32}$ centered at $C$. Hence, 
the above integral of $O(r^2\alpha)$ over the complement 
to the disk $D_{\delta^{\frac32}}$ is $o(\delta^3)$. The integral inside the disk 
is also $o(\delta^3)$, since its intersection with $\gamma$ has length of 
order $O(\delta^{\frac32})$, while the subintegral expression is 
$O(\delta^2)$. Finally, the upper bound $I$ for the contribution 
of the non-Euclidean 
angular part $\Delta$ is $o(\delta^3)$. This implies 
the statement of the claim for the $g$-length $\la(A,B)$, and hence, 
for the expression $L(A,B)$: the  $g$-lengths of the segments $AC$, 
$BC$ coincide with their Euclidean length by the definition of normal 
coordinates.
\end{proof}

{\bf Claim 2.} {\it Let  $\gamma\subset\rr^2$ be a curve with positive geodesic 
curvature. 
For every point $A\in\gamma$ consider the osculating circle $S_A$ 
at $A$ of the curve $\gamma$. For every $B\in\gamma$ close to $A$ 
let us consider the point $B'\in S_A$ closest to $B$ ($BB'\perp S_A$) 
and the corresponding expressions $\la(A,B')$, $L(A,B')=L_{S_A}(A,B')$ written 
for the circle $S_A$. One has}
$$\la(A,B')-\la(A,B)=o((s_A-s_B)^3), \ \ L(A,B')-L(A,B)=o((s_A-s_B)^3).$$

\begin{proof} Recall that we denote $\delta=|s_A-s_B|$. 
The lengths of the arcs $AB\subset\gamma$ 
and $AB'\subset S_A$ differ by a quantity $o(\delta^3)$. 
Indeed, the projection of the arc $AB$ 
to the arc $AB'$ along the radii of the circle $S_A$ 
has norm of derivative of order  $1+o(\delta^2)$. This is implied by the 
two following statements: 1) the distance between the 
source and the image is of order $o(\delta^2)$ (the circle is osculating); 
2) the slopes of the corresponding tangent lines differ by a quantity 
$o(\delta)$. The asymptotics $1+o(\delta^2)$ for the norm of 
projection implies that $\la(A,B)-\la(A,B')=o(\delta^3)$. 
Let us now show that the straightline parts of the expressions 
$L(A,B)$ and $L(A,B')$ also differ by a quantity $o(\delta^3)$. 
The tangent lines $T_B\gamma$ and 
$T_{B'}S_A$ pass through $o(\delta^2)$-close points 
$B$ and $B'$, and their slopes differ by a quantity $o(\delta)$, see  
the above statements 1) and  2). 
Note that  $BB'\perp T_{B'}S_A$. This implies that the distance between 
their points $C$ and $C'$ of intersection 
with the line $T_A\gamma$ is $o(\delta)$. Let $H$ denote 
the point of intersection of the line $T_{B'}S_A$ and its orthogonal 
line passing through $C$. The 
difference of the straighline parts of the expressions $L(A,B)$ and 
$L(A,B')$ is equal to $(|BC|-|B'H|)\pm(|CC'|-|C'H|)$. 
The second bracket is the difference of a cathet and a hypothenuse, 
both of order $o(\delta)$, in a right triangle with angle $O(\delta)$ 
between them. Hence, the latter difference is $o(\delta^3)$, since 
the cosine of the angle is $1+O(\delta^2)$. The first bracket 
is equal to the similar difference in another right triangle,  
with cathet $B'H$  and 
 hypothenuse being the segment $BC$ shifted by the 
vector $\vec{BB'}$; both are of order $O(\delta)$, and 
the  angle between them is $o(\delta)$. 
Hence, the first bracket is $o(\delta^3)$ (the cosine being now 
$1+o(\delta^2)$). Finally, the difference of the straightline parts 
of the expressions $L(A,B)$ and $L(A,B')$ is $o(\delta^3)$. 
The claim is proved.
\end{proof}

Claims 1 and 2 reduce  Theorem \ref{lasyl} to the case, when the metric 
is Euclidean and $\gamma$ is  a circle in $\rr^2$. Let $R$ denote 
its radius. Let $AB$ be its arc cut by a sector of small angle $\phi$.   
Then 
$$L(A,B)=R(2\tan(\frac{\phi}2)-\phi)\simeq\frac R{12}\phi^3=
\frac{\kappa^2}{12}|s_A-s_B|^3, \ \kappa=R^{-1}.$$
Uniformity of the latter asymptotics, as stated in Theorem \ref{lasyl}, 
follows by the above arguments from uniformity of intermediate asymptotics used in the proof. This proves Theorem \ref{lasyl}.
\end{proof}

\begin{proof} {\bf of Corollary \ref{ttk2}.} Let $C\in\Gamma_p$. 
Let $A=A(C)$, $B=B(C)\in\gamma$  denote the points such that 
the geodesics $AC$ and $BC$ are tangent to $\gamma$ at $A$ and $B$ 
respectively. We order them so that $B(C)=\mct_p(A(C))$. One has $L(A(C),B(C))=p$ for all $C\in\Gamma_p$, by the 
definition of string curve $\Gamma_p$. On the other hand, 
$L(A(C),B(C))\simeq\kappa^2(A(C))|s(A(C))-s(B(C))|^3$, as 
$C$ tends to a compact subarc $\wh\gamma\Subset\gamma$, 
by Theorem \ref{lasyl}. Therefore, all the quantities 
$\kappa^{\frac23}(A(C))|s(A(C))-s(B(C))|$ are uniformly 
asymptotically equivalent. Substituting 
$B(C)=\mct_p(A(C))$, we get  (\ref{ka32la}). Corollary 
\ref{ttk2} is proved.
\end{proof}

Corollary \ref{cttk2} follows immediately from Corollary \ref{ttk2}.

\begin{proof} {\bf of Theorem \ref{ttk}.} Let the curve $\gamma$ have 
string Poritsky property. Let $t$ denote its Poritsky parameter. 
Set $f:=\frac{dt}{ds}$.  For the proof of Theorem \ref{ttk} 
it suffices to show that $f\equiv\kappa^{\frac23}$ up to constant factor. 
Or equivalently, that for every $A,Q\in\gamma$ one has 
\begin{equation}
\frac{f(Q)}{f(A)}=\frac{\kappa^{\frac23}(Q)}{\kappa^{\frac23}(A)}.
\label{gqua}\end{equation}
Fix a small $p>0$. Set $B:=\mct_p(A)$, $R=\mct_p(Q)$. One has 
\begin{equation} t(B)-t(A)=t(R)-t(Q),\label{porrq}\end{equation}
by Poritsky property. On the other hand, the latter left- and right-hand sides 
are asymptotically equivalent respectively to $f(A)\la(A,B)$ and 
$f(Q)\la(Q,R)$. But 
$$\kappa^{\frac23}(A)\la(A,B)\simeq\kappa^{\frac23}(Q)\la(Q,R), 
\text{ as } p\to0,$$
by Corollary \ref{ttk2}. Substituting the two latter asymptotics to 
(\ref{porrq})   yields (\ref{gqua}). Theorem \ref{ttk} is proved.
\end{proof}

\section{Symplectic generalization of Theorem \ref{ttk}}

In Subsection 7.1 we give a background material 
on symplectic properties of the billiard ball reflection map. 
In Subsection 7.2 we introduce weakly billiard-like maps. We consider 
the so-called string type families of weakly  billiard-like maps, which 
generalize the family of billiard reflections from string construction curves 
 defined by a curve with string Poritsky property. We state 
Theorem \ref{tsympor}, which is a 
symplectic generalization  of Theorem \ref{ttk} ($C^6$-smooth case) to the string type 
billiard-like map families.  
Theorem \ref{tsympor} will be proved in Subsection 7.4. 
For its proof  we introduce an analogue of Lazutkin coordinates, 
the so-called modified Lazutkin coordinates, for weakly billiard-like maps (Subsection 7.3). In the same subsection 
we prove Theorem \ref{thlame}, which gives an asymptotic normal form for a 
weakly billiard-like map  in Lazutkin coordinates. We also prove  Lemma \ref{plog} 
on asymptotics of orbits in Lazutkin coordinates. 
 
 In Subsection 7.5 we show how to retrieve Theorem 
\ref{ttk} for $C^6$-smooth curves from Theorem \ref{tsympor}.

\def\mcj{\mathcal J}

The idea to extend Theorem \ref{ttk} to 
 a more general symplectic context was 
suggested by Sergei Tabachnikov.

\subsection{Symplectic properties of billiard ball map}
\def\mct{\mathcal T}

Here we recall a background material on symplecticity of billiard ball map. It can be found 
in \cite{ar2, ar3, mm, melrose1, melrose2, tab95}. 

Let $\Sigma$ be a surface with Riemannian metric. Let $\Pi:T\Sigma\to\Sigma$ denote the tautological projection. 
 Let us recall that  the {\it tautological 1-form} 
$\alpha$ on $T\Sigma$ (also called the {\it Liouville form}) 
is defined as follows: for every $(Q,P)\in T\Sigma$ with $Q\in\Sigma$ and $P\in T_Q\Sigma$ for every $v\in T_{(Q,P)}(T\Sigma)$ set 
\begin{equation}\alpha(v):=<P,\Pi_*v>.\label{alp}\end{equation} 
The differential 
$$\omega=d\alpha$$
of the 1-form $\alpha$ is the {\it canonical symplectic form} on $T\Sigma$. 

Let $O\in\Sigma$, and let $\gamma\subset\Sigma$ be a germ of regular oriented curve at $O$. 
Let us parametrize it by its natural length parameter $s$. The corresponding function $s\circ\Pi$ on $T\gamma$ 
will be also denoted by $s$.  For every $Q\in\gamma$ and $P\in T_Q\gamma$  set 
$$\dot\gamma(Q)=\frac{d\gamma}{ds}(Q):= \text{ the directing unit tangent vector to } 
\gamma \text{ at } Q,$$ 
\begin{equation}\sigma(Q,P):=<P,\dot\gamma(Q)>, \ \ y(Q,P):=1-\sigma(Q,P).
\label{y(q,p)}\end{equation}
The restriction to $T\gamma$ of the form $\omega$ is a symplectic form, which will be denoted by the same 
symbol $\omega$. 
\begin{proposition} \label{psy} (see \cite[formula (3.1)]{mm} in the Euclidean case). 
The coordinates $(s,y)$ on $T\gamma$ are symplectic: 
$\omega=ds\wedge dy$ on $T\gamma$.
\end{proposition}
\begin{proof}
The proposition follows from the definition of the symplectic structure $\omega=d\alpha$, 
$\alpha$ is the same, as in (\ref{alp}): in local coordinates $(s,\sigma)$ one has 
$\alpha=\sigma ds$, thus, $\omega=d\sigma\wedge ds=ds\wedge dy$. 
\end{proof}

Let $V$ denote 
the Hamiltonian vector field on $T\Sigma$ with the Hamiltonian $||P||^2$: the field $V$ 
generates the geodesic flow. Consider the unit circle bundle over $\Sigma$: 
$$S=\mct_1\Sigma:=\{||P||^2=1\}\subset T\Sigma.$$
 It is known that for every point $x\in S$  the kernel of the restriction $\omega|_{T_xS}$ is 
the one-dimensional linear subspace spanned by the vector $V(x)$ of the field $V$. 
Each cross-section  $W\subset S$ to the field $V$ is identified with the (local) space 
of geodesics.  The symplectic structure 
$\omega$ induces a well-defined symplectic structure on $W$ called the 
{\it symplectic reduction.} 

\begin{remark} \label{holin} The symplectic reduction is invariant under  holonomy along orbits 
of the geodesic flow. Namely,  for every arc $AB$ of trajectory of the geodesic flow 
with endpoints $A$ and $B$ for every  two germs of 
cross-sections $W_1$ and $W_2$ through  $A$ and $B$ respectively the holonomy mapping 
$W_1\to W_2$, $A\mapsto B$ 
along the  arc $AB$ is a symplectomorphism. 
\end{remark}

Consider the local hypersurface 
$$\Gamma=\Pi^{-1}(\gamma)\cap S=(\mct_1\Sigma)|_{\gamma}\subset S.$$
At those points $(Q,P)\in\Gamma$, for which the vector $P$ is transverse to $\gamma$ 
the hypersurface $\Gamma$ is locally a cross-section for the restriction to $S$ of the geodesic flow. Thus, near the latter points 
the hypersurface $\Gamma$ carries a canonical 
symplectic structure given by the symplectic reduction. Set 
$$\mco_{\pm}:=(O,\pm\dot\gamma(O))\in\Gamma.$$
For every 
$(Q,P)\in \Gamma$ close enough to $\mco_{\pm}$ with $Q$ lying 
on the convex side from $\gamma$ the geodesic issued from the point 
$Q$ in the direction $P$ (and oriented by $P$) 
 intersects the curve $\gamma$ at two points $Q$ and $Q'$ 
 (which coincide if $P$ is tangent to $\gamma$). Let $P'$ denote the orienting unit tangent vector of the latter 
 geodesic at $Q'$. This 
 defines the germ at $\mco_{\pm}$ of involution
 \begin{equation}\beta:(\Gamma,\mco_{\pm})\to(\Gamma,\mco_{\pm}), \ \ \ \beta(Q,P)=(Q',P'), \ \ \beta^2=Id,\label{defb}
\end{equation}
 which will be called the {\it billiard ball geodesic correspondence.}

Consider the following open subset in $T\gamma$: the unit ball bundle
$$\mct_{\leq1}\gamma:=\{ (Q,P)\in T\gamma \ | \ ||P||^2\leq1\}.$$ 
Let $\pi:(T\Sigma)|_{\gamma}\to T\gamma$ denote the mapping acting by orthogonal projections 
$$\pi:T_Q\Sigma\to T_Q\gamma, \ Q\in\gamma.$$ 
It induces the following projection also denoted by $\pi$: 
\begin{equation}\pi:\Gamma\to \mct_{\leq1}\gamma.\label{pigt}\end{equation}

Let $\mcv$ denote a convex domain with boundary containing $\gamma$. 
Every point $(Q,P)\in\mct_{\leq1}\gamma$ has two $\pi$-preimages $(Q,w_{\pm})$  in $\Gamma$: the 
vector $w_+$ ($w_-$)  is directed inside (respectively, outside) the domain $\mcv$. 
The vectors $w_{\pm}$ coincide, 
if and only if $||P||=1$, and in this case they  lie in $T_Q\gamma$. 
 Thus, the mapping 
$\pi:\Gamma\to\mct_{\leq1}\gamma$ 
has two continuous inverse branches. Let $\mu_+:=\pi^{-1}:\mct_{\leq1}\gamma\to\Gamma$ 
denote the inverse branch sending $P$ to $w_+$, cf. \cite[section 2]{mm}.  The above mappings define 
the germ of mapping 
\begin{equation}\delta_+:=\pi\circ\beta\circ\mu_+:(\mct_{\leq1}\gamma,\mco_{\pm})\to(\mct_{\leq1}\gamma,\mco_{\pm}).
\label{d+}\end{equation}
Recall that $\Gamma$ carries a canonical symplectic structure given by the 
above-mentioned symplectic reduction (as a cross-section), and $T\gamma$ carries 
the standard symplectic structure: the restriction to $T\gamma$ 
of the form $\omega=ds\wedge dy$. 
\begin{theorem} \label{tsym} \cite[subsection 1.5]{tab95}, \cite{melrose1, melrose2, 
ar2, ar3}
The mappings $\beta$, $\pi$, and hence,  $\delta_+$ given by (\ref{defb})--(\ref{d+}) respectively are symplectic. 
\end{theorem} 

\begin{proof} Symplecticity of the mapping $\beta$ follows from the definition of symplectic reduction and its holonomy invariance 
(Remark \ref{holin}). Symplecticity of the projection $\pi$ follows from definition 
and the fact that the $\pi$-pullback of the tautological 
1-form $\alpha$ on $T\gamma$ is the restriction to $\Gamma$ of the tautological 
1-form on $T\Sigma$. This implies symplecticity of $\mu_+=\pi^{-1}$,  
 and hence, $\delta_+$. 
\end{proof}

Let $I:\Gamma\to\Gamma$ denote the reflection involution 
$$I:(Q,P)\mapsto (Q,P^*),$$
$$Q\in\gamma, \  P^*:= \text{ the vector symmetric to } P \text{ with respect to the line } T_Q\gamma.$$
\begin{proposition} \label{psmi} The involution $I$ preserves the 
tautological 1-form $\alpha$, and hence, is symplectic. 
 The involutions $I$ and $\beta$ are $C^r$-smooth germs of mappings $(\Gamma,\mco_{\pm})\to
(\Gamma,\mco_{\pm})$, if the  Riemannian metric  and the curve 
 $\gamma$ are $C^{r+1}$-smooth. 
The mapping $\delta_+$ is conjugated to their product 
\begin{equation}\wt\delta_+:=I\circ\beta=\mu_+\circ\delta_+\circ\mu_+^{-1}.\label{prinv0}\end{equation}
\end{proposition}
The proposition follows immediately from definitions.

 The billiard transformation $T$ of reflection from the curve $\gamma$ 
 acts on the space of oriented geodesics that intersect $\gamma$ and are close enough to the 
 geodesic tangent to $\gamma$ at $O$. Each of them intersects $\gamma$ at two points. To each oriented geodesic $G$ 
 we put into correspondence a point $(Q,P)\in\Gamma=(\mct_1\Sigma)|_{\gamma}$, where $Q$ is its first intersection point with 
 $\gamma$ (in the sense of the orientation of the geodesic $G$) and $P$ is the orienting unit vector tangent to $G$ at $Q$. 
 This is a locally bijective correspondence. 
 
 \begin{proposition} \label{bma} Let the  metric and the curve $\gamma$ be  
 $C^3$-smooth. The billiard mapping $T$ written as a mapping $\Gamma\to\Gamma$ via the above correspondence 
 coincides with $\wt\delta_+$. 
 Consider the coordinates $(s,\phi)$ on  $\Gamma$: 
 $s=s(Q)$ is the natural length parameter of a point $Q\in\gamma$; $\phi=\phi(Q,P)$ is the oriented angle of the vector $\dot\gamma(Q)$ 
 with a vector $P\in T_Q\Sigma$. 
 In the coordinates $(s,\phi)$ the mappings $I$, $\beta$ and $T=\wt\delta_+$ are $C^2$-smooth and take the form
 \begin{equation}I(s,\phi)=(s,-\phi), \ \beta(s,\phi)=(s+2\kappa^{-1}(s)\phi+O(\phi^2), -\phi+O(\phi^2)).\label{beti}\end{equation}
 \begin{equation} \wt\delta_+(s,\phi)=(s+2\kappa^{-1}(s)\phi+O(\phi^2), \phi+O(\phi^2)).\label{bilike}\end{equation}
 The asymptotics are uniform in $s$, as $\phi\to0$. In the coordinates 
 \begin{equation} (s,y),  \ \ \ y=1-\cos\phi,\label{ycos}\end{equation}
  see (\ref{y(q,p)}), the billiard mapping $T$ coincides with $\delta_+$ and takes the form
 \begin{equation}\delta_+(s,y)=(s+2\sqrt 2\kappa^{-1}(s)\sqrt y+O(y), y+O(y^{\frac32})).\label{deltsy}\end{equation}
\end{proposition}
\begin{proof} All the statements of the proposition except for the formulas follow from definition. Formula (\ref{beti}) follows 
from the definitions of the mappings $I$ and $\beta$:  a geodesic issued from a point $Q\in\gamma$ at 
a small angle $\phi$ with the tangent vector $\dot\gamma(Q)$ intersects 
$\gamma$ at a point $Q'$ such that $\lambda(Q,Q')=2\kappa^{-1}(Q)\phi+O(\phi^2)$. The latter formula 
follows from its Euclidean analogue (applied to the curve $\gamma$ 
represented in normal coordinates centered at $Q$), Proposition \ref{kappa=} 
 and smoothness. 
Formulas (\ref{beti}) and (\ref{prinv0}) imply (\ref{bilike}), which in its turn implies (\ref{deltsy}), since $y=\frac{\phi^2}2+O(\phi^4)$. 
\end{proof}

\subsection{Families of billiard-like maps with invariant curves. A symplectic version of Theorem \ref{ttk}}

In this and the next subsections we study the following class of area-preserving mappings generalizing the billiard mappings 
represented in the coordinates $(s,y)$, see (\ref{deltsy}). 
\begin{definition} \label{tdw} A {\it  weakly billiard-like  map} is 
a germ of mapping preserving the standard area form 
$dx\wedge dy$, 
$$F:(\rr\times\rr_{\geq0},(0,0))\to(\rr\times\rr_{\geq0},(0,0)),$$
\begin{equation}F=(f_1,f_2):(x,y)\mapsto(x+w(x)\sqrt y+O(y),y+O(y^{\frac32})),  \ w(x)>0,\label{fbm}\end{equation}
for which the $x$-axis is a line of fixed points and such that 
the variable change 
$$(x,y)\mapsto (x,\phi), \ y=\phi^2$$ 
conjugates $F$ with a $C^2$-smooth germ $\wt F(x,\phi)$. The above 
asymptotics are uniform in $x$, as $y\to0$.  If, in addition to the above 
assumptions, the latter mapping $\wt F$ 
is a product of two involutions:
$$\wt F=I\circ\beta, \ I(x,\phi)=(x,-\phi),$$
\begin{equation}\beta(x,\phi)=(x+w(x)\phi+O(\phi^2), -\phi+O(\phi^2)), \ \beta^2=Id,
\label{prinv}\end{equation}
then $F$ will be called {\it a (strongly) billiard-like  map.}
\end{definition}

\begin{example} \label{exdel} 
The mapping $\delta_+$ from (\ref{deltsy}) 
is strongly billiard-like in the coordinates 
$(s,y)$ with $w(s)=2\sqrt 2\kappa^{-1}(s)$, see  (\ref{prinv0}), (\ref{bilike}) 
and (\ref{deltsy}). 
\end{example}
The next definition generalizes the notion of curve 
with Poritsky property to weakly billiard-like maps. 

\begin{definition} A family $F_{\var}(x,y)$ of 
weakly billiard-like maps (\ref{fbm}) 
 depending on a parameter 
$\var\in[0,\var_0]$ is of {\it string type,} if the derivatives 
up to order 2 of the corresponding mappings 
$\wt F_\var(x,\phi)$ are continuous in $(x,\phi,\var)$ on a product 
$\{|x|\leq\delta_0\}\times[0,\phi_0]\times[0,\var_0]$ and  
there exist a $\delta\in(0,\delta_0]$ and a family $\gamma_\var$ of 
$F_\var$-invariant  graphs of continuous functions $h_\var:[-\delta,\delta]
\to\rr_{\geq0}$,
\begin{equation}\gamma_\var=\{ y=h_\var(x)\},\label{gae}
\end{equation}
 such that  $\gamma_\var$ converge to 
the $x$-axis: $h_{\var}(x)\to0$ uniformly 
on  $[-\delta,\delta]$.
\end{definition}

\begin{example} Let $\gamma\subset\Sigma$ be 
a germ of curve with positive geodesic curvature such that 
the corresponding string construction curves $\Gamma_p$ are $C^3$-smooth  
and their 3-jets depend continuously on the base points. 
(For example, this holds automatically in the case, when 
the curve $\gamma$ and the metric are $C^6$-smooth, see Theorem 
\ref{thsm}.) Then the family of  billiard 
reflection maps from the curves $\Gamma_p$ is a 
string type family. The invariant curves $\gamma_p$ from (\ref{gae}) are 
identified with one and the same curve in the space of oriented geodesics: 
the family of geodesics tangent to the curve $\gamma$ and oriented by its 
tangent vectors $\dot\gamma$. See Subsection 7.5 for more details.
\end{example}

The next theorem extends Theorem \ref{ttk} on 
coincidence of Poritsky and Lazutkin parameters to 
the string type families of weakly billiard maps.

\begin{theorem} \label{tsympor} Let $F_\var(x,y)$ be 
a string type family of weakly billiard maps. Let for every $\var$ 
small enough 
there  exist a continuous 
strictly increasing parameter $t_\var$ on $\gamma_\var$ 
in which $F|_{\gamma_\var}$ is a translation 
by $\var$-dependent constant,
\begin{equation}t_\var\circ F|_{\gamma_\var}=
t_\var+c(\var),\label{tefe}\end{equation}
such that the parameter $t_\var=t_\var(x)$ 
considered as a function of $x$  converges to a strictly increasing function 
$t_0(x)$ uniformly on $[-\delta,\delta]$, as $\var\to0$. Then 
\begin{equation} t_0=a X+b, \ 
X:=\int_0^xw^{-\frac23}(z)dz, \ 
a,b\equiv const.\label{t0por}\end{equation}
Here $w=w_0(x)$ is the function from (\ref{fbm}) 
corresponding to the mapping $F_\var$ with $\var=0$. 
\end{theorem}

Theorem \ref{tsympor} is proved in Subsection 7.4. 

\subsection{Modified Lazutkin coordinates $(X,Y)$}

In the proof of Theorem \ref{tsympor} we use the following 
well-known theorem. 

\begin{theorem} \label{thlame} Let $F$ be a 
weakly billiard-like map $F$, and let $w(x)$ be the 
corresponding function in (\ref{fbm}). The transformation 
\begin{equation}
\mcl:(x,y)\mapsto(X,Y), \ \begin{cases} X(x)=\int_0^x
w^{-\frac23}(z)dz\\
Y(x,y):=w^{\frac23}(x)y\end{cases}\label{lazz}
\end{equation}
is symplectic and 
conjugates $F$ to  a mapping 
\begin{equation}F_L:(X,Y)\mapsto(X+\sqrt Y+O(Y), 
 Y+o(Y^{\frac32})), \ \text{ as } Y\to0.\label{flxy}\end{equation}
 The latter asymptotics are uniform in $X$. 
 The coordinates $(X,Y)$ will be called the {\bf modified Lazutkin 
 coordinates}. 
 \end{theorem}
 A version of Theorem \ref{thlame}  is implicitly contained in \cite{laz, mm}. For 
 completeness of presentation, we 
 present its proof below. In its proof we use the following 
 proposition. 
 
 \begin{proposition} The $y$-component of 
 a weakly billiard-like map (\ref{fbm}) admits the following 
 more precise formula: 
 \begin{equation} f_2(x,y)=y-\frac23w'(x)y^{\frac32}+
 o(y^{\frac32}).\label{f2y}\end{equation}
 \end{proposition}
 \begin{proof}
Indeed, consider the asymptotic Taylor expansion of the $C^2$-smooth lifting $\wt f_2(s,\phi)$ of the 
mapping $f_2$:
$$\wt f_2(x,\phi)=\phi+c(x)\phi^2+o(\phi^2); \ 
\phi=\sqrt y, \ \wt f_2=\sqrt f_2.$$
Hence, 
\begin{equation}f_2(x,y)=\wt f_2^2(x,\phi)= 
y(1+c(x)\sqrt y+o(\sqrt y))^2=y+
2c(x) y^{\frac32}+o(y^{\frac32}),\label{f22}\end{equation}
$$\frac{\partial f_2}{\partial y}(x,y)=\frac1{\sqrt y}\wt f_2
\frac{\partial \wt f_2}{\partial\phi}(s,\phi)=1+3c(x)\sqrt y+
o(\sqrt y).$$
This together with analogous calculations of the other 
partial derivatives, 
$$\frac{\partial f_1}{\partial x}=1+w'(x)\sqrt y+O(y),$$
$$\frac{\partial f_1}{\partial y}=O(y^{-\frac12}), 
\frac{\partial f_2}{\partial x}=2\wt f_2(x,\phi)\frac{\partial\wt f_2}{\partial x}(x,\phi)=o(\phi^2)=o(y),$$
shows that the determinant of the Jacobian matrix 
of the mapping $F(x,y)$ equals $1+(w'(x)+3c(x))\sqrt y+
o(\sqrt y)$. But it should be equal to 1, by symplecticity. 
Therefore, $c(x)=-\frac13w'(x)$. This together with 
(\ref{f22}) proves the proposition.
\end{proof}
 
 \begin{proof} {\bf of Theorem \ref{thlame}.} 
 Symplecticity of the transformation 
 $\mcl$ follows from definition. Let us show that 
 it conjugates $F$ to a mapping $F_L$ as in (\ref{flxy}). 
 One has 
 $$X\circ F(x,y)=X+\int_x^{x+w(x)\sqrt y+O(y)}
 w^{-\frac23}(z)dz$$
 \begin{equation}=X+w(x)w^{-\frac23}(x)\sqrt y +O(y)=
 X+\sqrt Y+O(Y),\label{xcircf}\end{equation}

$$Y\circ F(x,y)=w^{\frac23}(f_1(x,y))f_2(x,y)$$
$$=
w^{\frac23}(x+w(x)\sqrt y+o(\sqrt y))y(1-
\frac23w'(x)y^{\frac12}+o(y^{\frac12})).$$
Substituting the expressions $y=Yw^{-\frac23}(x)$ and 
$$w^{\frac23}(x+w(x)\sqrt y+o(\sqrt y))=
w^{\frac23}(x)+\frac23w^{-\frac13}(x)w'(x)w(x)\sqrt y+o(\sqrt y)$$
$$=w^{\frac23}(x)(1+\frac23w'(x)\sqrt y+o(\sqrt y))$$
to the above formula yields
$$Y\circ F(x,y)=Y(1+\frac23w'(x)y^{\frac12}+o(y^{\frac12}))
(1-\frac23w'(x)y^{\frac12}+o(y^{\frac12}))=Y+o(Y^{\frac32}).
$$
This together with (\ref{xcircf}) proves 
Theorem \ref{thlame}.
 \end{proof}
 
 \begin{lemma} \label{plog} Let $F(x,y)$ be a weakly billiard-like map, 
 and let $(X,Y)$ be the corresponding modified Lazutkin 
 coordinates; we write $F$ in the coordinates $(X,Y)$. 
 Let the mapping be defined 
 and the asymptotics 
 from Theorem \ref{thlame} hold for $(X,Y)\in V_{\Delta,\eta}:=[-\Delta,\Delta]\times[0, \eta]$, $0<\eta<1$.   For every $\delta>0$, $\delta<\Delta$, 
 and every $q_0\in V_{\delta,\eta}$ set 
 \begin{equation} 
 q_j:=F^j(q_0), \ m=m(q_0):=\max\{ j\in\nn \ | \ 
 q_i\in V_{\delta,\eta} \text{ for every } i\leq j\}.\label{defmq0}\end{equation}
There exist  positive non-decreasing functions 
  $\alpha=\alpha(Y)$, $\beta=\beta(Y)$ in small $Y>0$, $\alpha(Y),\beta(Y)\to0$, as $Y\to0$, 
   such that for every $\wt\eta\in(0,\eta)$ small enough for every $q_0\in V_{\delta,\wt\eta}$ the following statements hold:
   
   1)  One has 
    \begin{equation} |\ln\left(\frac{Y(q_j)}{Y(q_0)}\right)|<\alpha(Y(q_0)) \text{ for every  
 } j\leq m(q_0).\label{logy}\end{equation}
 
 2) The number $m(q_0)$ is exactly 
 the biggest number $m$ such that $X(q_j)\leq\delta$ 
 for every $j\leq m$. 
 
 3) For every $j=1,\dots,m(q_0)$ one has 
 $$e^{-\beta(Y(q_0))}Y^{\frac12}(q_0)
 \leq X(q_j)-X(q_{j-1})\leq e^{\beta(Y(q_0))}Y^{\frac12}(q_0),$$ 
 \begin{equation} je^{-\beta(Y(q_0))}Y^{\frac12}(q_0)
 \leq X(q_j)-X(q_0)\leq je^{\beta(Y(q_0))}Y^{\frac12}(q_0).
 \label{ineqxq}\end{equation}
 \end{lemma}
 {\bf Addendum to Lemma \ref{plog}.} 
 {\it Let $F_\var$ be a family of weakly billiard-like maps. 
 Let   the derivatives in $(x,\phi)$ up to order 
 2 of the corresponding mappings $\wt F_\var(x,\phi)$ 
be  continuous on  
 $[-\sigma,\sigma]\times[0,\phi_0]\times[0,\var_0]$. Let 
 $(X_\var,Y_\var)$ be the corresponding family 
 of Lazutkin coordinates (\ref{lazz}). The coordinates $(X_\var,Y_\var)$ depend continuously 
 on $(x,y,\var)$ with first derivatives in $(x,y)$. There exist 
 $\delta\in(0,\sigma]$, $\eta,\var_1>0$, $\wt\eta\in(0,\eta)$, such that 
 asymptotics (\ref{flxy}) are uniform in $X\in[-\delta,\delta]$ and $\var\in[0,\var_1]$ and such that 
 
 - the statements of Lemma \ref{plog} hold for all the mappings $F_\var$, $\var\in[0,\var_1]$,  
 the same domain $V_{\delta,\eta}$ and the same $\wt\eta$; 

- each of the  functions $\alpha(Y)$, $\beta(Y)$ from Lemma \ref{plog} corresponding to 
 individual mappings $F_\var$    can be chosen independent on $\var\in[0,\var_1]$.} 
 
 \begin{remark} It is possible that Lemma \ref{plog} is known to specialists, 
 but the author did not found its statements in literatute. 
\end{remark}
 
 \begin{proof} Consider the coordinate representation 
  $F=(f_1,f_2)$ of the mapping $F$ in the Lazutkin 
  coordinates $(X,Y)$. Recall that $f_2(X,Y)=Y+o(Y^{\frac32})$. 
 This means that there exists a function $\mu(z)>0$ 
 in small $z>0$, $\mu(z)\to0$, as $z\to0$, 
 such that for every $q_0\in V_{\delta,\eta}$ one has 
\begin{equation}|Y\circ F(q_0)-Y(q_0)|\leq\mu(Y(q_0))Y^{\frac32}(q_0).
\label{ineqy}\end{equation}
We take the function $\mu(Y)$ non-decreasing, 
replacing it by  $\sup_{0\leq z\leq Y}\mu(z)$. 
Inequality (\ref{ineqy}) applied to $q_j$ for every $j\leq m(q_0)$ implies that 
 the $Y$-coordinates $Y(q_j)$ are majorated by the values at $t=j$ of  the 
solution of the differential equation 
\begin{equation}\dot{\mathcal Y}=\mu_{j,\max}
\mathcal Y^{\frac32}, \ \mu_{j,\max}:=\mu(\max_{s\leq j-1}Y(q_s))\leq\mu(\eta), 
\label{diffyeq}\end{equation}
with the initial condition $\mcy(0)=Y(q_0)$ (before it escapes to infinity). 
The solution $\mcy(t)$ 
of equation (\ref{diffyeq}) with the initial 
condition $\mcy(0)=\mcy_0$ and its escape time $t_{\infty}$ to infinity are found from the formulas 
\begin{equation}t=\frac2{\mu_{j,\max}}(\mcy_0^{-\frac12}-\mcy^{-\frac12}(t)), \ \ t_{\infty}(\mcy_0)=\frac2{\mu_{j,\max}}\mcy_0^{-\frac12}.
\label{tinf}\end{equation}
This together with the above discussion   yields 
$$\frac{Y(q_j)}{Y(q_0)}<
(1-\frac j2\mu_{j,\max}Y^{\frac12}(q_0))^{-2} \ \text{ for  } j=0,\dots,
\min\{ m(q_0), t_{\infty}(Y(q_0))\}.$$
If $Y(q_j)<Y(q_0)$ and $j\leq \min\{m(q_0), t_{\infty}(Y(q_j)\}$, then applying the above 
argument in the inverse time we get 
$$\frac{Y(q_0)}{Y(q_j)}<
(1-\frac j2\mu_{j,\max}Y^{\frac12}(q_j))^{-2}.$$
But if $Y(q_j)<Y(q_0)$, then $t_{\infty}(Y(q_j))> t_{\infty}(Y(q_0))$, by (\ref{tinf}), 
and the above  right-hand side is less than the same expression with $q_j$ replaced by $q_0$.
Hence, the above inequality with thus modified right-hand side 
automatically holds if $j\leq\min\{ m(q_0), t_{\infty}(Y(q_0))\}$, and finally, 
\begin{equation}\max_{\pm}\left(\frac{Y(q_j)}{Y(q_0)}\right)^{\pm1}<
(1-\frac j2\mu_{j,\max}Y^{\frac12}(q_0))^{-2} \ \text{ for  } j\leq
\min\{ m(q_0), t_{\infty}(Y(q_0))\}.
 \label{ineqy12}\end{equation} 
For every $q_0\in V_{\delta,\eta}$ set  
$$m_4(q_0):=\max\{ j\leq m(q_0) \ | \ 
\frac14Y(q_0)\leq Y(q_i)\leq4Y(q_0) \text{ for every } i\leq j\}.$$

{\bf Claim 1.} {\it One has either $m_4(q_0)=m(q_0)$, or $m_4(q_0)\geq t_{\infty}(Y(q_0))-1$, or} 
\begin{equation} m_4(q_0)\geq k_4(Y(q_0)):=(\mu(4Y(q_0))
Y^{\frac12}(q_0))^{-1}-1.\label{m4m2}\end{equation}

\begin{proof}  One always has $m_4(q_0)\leq m(q_0)$, 
by definition. Suppose that $m_4(q_0)<\min\{ m(q_0), t_{\infty}(Y(q_0))-1\}$.  
For every $j\leq m_4(q_0)+1$ one has 
$\mu_{j,\max}\leq \mu(4Y(q_0))$. Hence, for $j=m_4(q_0)+1$ one has 
$$4<\max_{\pm}\left(\frac{Y(q_{j})}{Y(q_0)}\right)^{\pm1}<(1-\frac{j}2\mu(4Y(q_0))
Y^{\frac12}(q_0))^{-2},$$
by definition and (\ref{ineqy12}). The latter inequality implies  (\ref{m4m2}). 
\end{proof}

{\bf Claim 2.} {\it  For every $\wt\eta\in(0,\frac\eta4)$ small enough  
for every $q_0\in V_{\delta,\wt\eta}$
 one has}   
 \begin{equation} m(q_0)=m_4(q_0)<t_{\infty}(Y(q_0))-1, \ Y\circ F^{m(q_0)+1}(q_0)<\eta.\label{mq0y}\end{equation}

\begin{proof} 
One has $f_1(X,Y)=X+\sqrt Y(1+o(1))$, as $Y\to0$. 
This means that there exists a non-decreasing 
 function $\phi(\wt\eta)>0$, 
$\phi(\wt\eta)\to0$, as $\wt\eta\to0$,  
such that 
\begin{equation}|f_1(X,Y)-X-\sqrt Y|<\phi(\wt\eta)\sqrt Y 
\text{ whenever } |X|\leq\delta, \ Y\leq\wt\eta.
\label{ineqf1}\end{equation}
Take a $\wt\eta>0$ such that  
$\wt\eta<\frac{\eta}4$ and $\phi(\wt\eta)<\frac14$. Then for every $q_0\in V_{\delta,\wt\eta}$ and every $j\leq m_4(q_0)$ 
one has 
$$\frac34Y^{\frac12}(q_j)<X(q_{j+1})-X(q_j)<\frac54
Y^{\frac12}(q_j),$$
by (\ref{ineqf1}), and $\frac14Y(q_0)\leq Y(q_j)\leq4Y(q_0)$, by definition. Therefore,  
\begin{equation}\frac38Y^{\frac12}(q_0)<X(q_{j+1})-X(q_j)<\frac52
Y^{\frac12}(q_0),\label{38j}\end{equation}
\begin{equation}j\frac38Y^{\frac12}(q_0)<X(q_j)-X(q_0)<
j\frac54Y^{\frac12}(q_0),
\label{38}\end{equation}
adding inequalities (\ref{38j}). Hence, 
\begin{equation} 
m_4(q_0)\leq \frac83(\delta-X(q_0))Y^{-\frac12}(q_0)
\leq l(Y(q_0)):=6\delta Y^{-\frac12}(q_0).\label{inm0}
\end{equation}
On the other hand, 
$l(Y(q_0))<\min\{ k_4(Y(q_0)), t_{\infty}(Y(q_0))-1\}$, whenever $Y(q_0)$ is small 
enough, see (\ref{tinf}) and (\ref{m4m2}), since $\mu(Y)\to0$ as $Y\to0$; we consider that this holds 
whenever $Y(q_0)\leq\wt\eta$, choosing $\wt\eta$ small enough. 
Hence, $m_4(q_0)=m(q_0)\leq t_{\infty}(Y(q_0))-1$, by  (\ref{inm0}) and Claim 1. Therefore,  $Y\circ F^{m(q_0)}(q_0)\leq 4Y(q_0)\leq 4\wt\eta$. Thus, if $\wt\eta$ is chosen small enough, then 
$Y\circ F^{m(q_0)+1}(q_0)<\eta$ whenever $q_0\in V_{\delta,\wt\eta}$, 
by (\ref{ineqy}) applied to $F^{m(q_0)}(q_0)$ instead of $q_0$. 
This proves (\ref{mq0y}). 
\end{proof}

Statement 2) of Lemma \ref{plog} follows immediately from Claim 2. Indeed, 
$F^{m(q_0)+1}(q_0)\notin V_{\delta,\eta}$, by definition. This can happen 
only when $X\circ F^{m(q_0)+1}(q_0)>\delta$, by the last inequality in (\ref{mq0y}).

Let us prove Statement 1). 
Whenever $Y(q_0)\leq\wt\eta$, 
one has $Y(q_j)\leq 4Y(q_0)$ for every 
 $j=1,\dots, m(q_0)$ (Claim 2), and  $\mu(Y)$ is non-decreasing. Hence, 
 $$\mu(\max_{j\leq m(q_0)}Y(q_j))
 \leq\mu(4Y(q_0)).$$
 This together with (\ref{ineqy12}) implies that for 
 every $j\leq m(q_0)$ one has 
\begin{equation}\max_{\pm}\left(\frac{Y(q_j)}{Y(q_0)}\right)^{\pm1}<
(1-\frac j2\mu(4Y(q_0))Y^{\frac12}(q_0))^{-2},
\label{innew}\end{equation}
$$j\leq m(q_0)=m_4(q_0)\leq6\delta Y^{-\frac12}(q_0),$$
see (\ref{inm0}). Substituting the latter inequality to 
(\ref{innew}) yields 
$$\max_{\pm}\left(\frac{Y(q_j)}{Y(q_0)}\right)^{\pm1}<
(1-3\delta\mu(4Y(q_0)))^{-2}.$$
This proves Statement 1) of Lemma  
\ref{plog} with 
$$\alpha(Y)=-2\ln(1-3\delta\mu(4Y)).$$
Inequality $Y(q_j)\leq 4Y(q_0)$  together with Statement 1) and 
(\ref{ineqf1}) imply that 
$$e^{-\alpha(Y(q_0))}Y(q_0)\leq Y(q_j)\leq e^{\alpha(Y(q_0))}Y(q_0) 
\text{ for } j\leq m(q_0),$$ 
$$(1-\phi(4Y(q_0)))Y^{\frac12}(q_j)\leq X(q_{j+1})-X(q_j)\leq (1+\phi(4Y(q_0)))Y^{\frac12}(q_j).$$ 
Substituting the former inequality to the latter one and adding the inequalities 
thus obtained for $j=0,\dots,m(q_0)-1$ yields Statement 3) 
of Lemma \ref{plog} with 
  $$\beta(Y)=-\ln(1-\phi(4Y))+\frac12\alpha(Y)$$
and finishes the proof of the lemma.
 \end{proof}

 \begin{proof} {\bf of the addendum to Lemma \ref{plog}.} Continuity and uniformity of asymptotics follows 
 from the formula for the modified Lazutkin coordinates. The possibility to choose constants and functions 
 $\alpha(Y)$, $\beta(Y)$ independent of $\var$ follows from uniformity of 
 asymptotics and the  proof of Lemma \ref{plog}.
 \end{proof}
 
 \subsection{Proof of Theorem \ref{tsympor}} 
 Everywhere below we write the mappings $F_\var$ 
 in the modified Lazutkin coordinates $(X_\var,Y_\var)$.  
 We consider that $F_\var$ are well-defined on 
 one and the same set $V_{\delta,\eta}=[-\delta,\delta]
 \times[0,\eta]$ for all $\var\in[0,\var_0]$ 
 in the coordinates $(X_\var,Y_\var)$. Thus, we 
 identify the Lazutkin coordinates for all $\var$ 
 and denote them by $(X,Y)$. 
 To show that  the limit parameter $t_0$ 
is  equal to the Lazutkin coordinate $X$ up to 
multiplicative and additive constants, we have to show 
that for every four distinct points in the $X$-axis 
with  $X$-coordinates $\mcx_j$, 
$$-\delta<\mcx_1<\mcx_2<\mcx_3<\mcx_4<\delta,$$
 the ratios of lengths of the segments 
$$I_1:=[\mcx_1,\mcx_2], \ \ I_3:=[\mcx_3,\mcx_4]$$ in 
the parameters $X$ and $t_0$ are equal: 
\begin{equation}\frac{t_0(\mcx_2)-t_0(\mcx_1)}{t_0(\mcx_4)-t_0(\mcx_3)}=\frac{\mcx_2-
\mcx_1}{\mcx_4-
\mcx_3}.\label{ineqte}\end{equation}

Take  a $\var>0$  small enough, and consider 
the corresponding $F$-invariant curve $\gamma_\var$. 
It can be represented as the graph 
$\{ Y=H_\var(X)\}$ of a continuous function. The parameter $t_\var$ on $\gamma_\var$ 
in which $F_\var$ is a translation induces a parameter 
on the $X$-axis via projection; the induced parameter 
will be also denoted by $t_\var$.  
Consider the following  orbit in the curve 
$\gamma_\var$:  
$$q_{0,\var}=(\mcx_1,H_\var(\mcx_1)), \ q_{j,\var}:=F^j_\var(q_{0}).$$
Let $m(q_{0,\var})$ denote the  number 
in (\ref{defmq0}) defined for the mapping $F_\var$. Set 
$$\nu(\var):=H_\var(\mcx_1); \ \nu(\var)\to0, \text{ as } \var\to0.$$

Recall that the sequence $X(q_{j,\var})$ is well-defined for $j\leq m(q_{0,\var})$, strictly 
increasing with steps $\nu^{\frac12}(\var)(1+o(1))$, as $\var\to0$, and for $j=m(q_{0,\var})$ 
one has $X(q_{j+1,\var})\geq\delta$. The 
above asymptotics of steps is uniform in $j\leq m(q_{0,\var})$. The three latter statements follow from 
 Lemma \ref{plog} and its addendum. 
For every $i=1, 2,3,4$ let $j_i=j_{i,\var}<m(q_{0,\var})$ denote the maximal number $j$ for which $X(q_{j,\var})\leq\mcx_i$. 
By definition, $i_1=0$.
For every $\var$ small enough one has 
$\mcx_{i}-X(q_{j_i,\var})<2\nu^{\frac12}(\var)$ for all 
$i=2,3,4$, by the above asymptotics. The sequence 
$t_\var(X(q_{j,\var}))$ is an arithmetic progression, 
since $F_\var|_{\gamma_\var}$ acts as a translation 
in the parameter $t_\var$. Its step tends to zero, 
as $\var\to0$, 
since $t_\var$ limits to a strictly increasing continuous parameter $t_0$ 
and the 
$X$-lengths of  steps tend to zero uniformly. This implies 
that the ratio of the $t_\var$-lengths of the segments 
$I_1$ and $I_3$ has the same finite asymptotics, 
as the ratio 
$$R_{1,3}(\var):=\frac{j_{2,\var}-
j_{1,\var}}{j_{4,\var}-j_{3,\var}}.$$ 
But the ratio of their $X$-lengths 
has also the same asymptotics, as $R_{1,3}(\var)$, 
since all the steps of 
the sequence $X(q_{j,\var})$ are asymptotically 
equivalent to one and the same quantity $\nu^{\frac12}(\var)$. 
This proves (\ref{ineqte}) and Theorem \ref{tsympor}.

\subsection{Deduction of Theorem \ref{ttk} for $\gamma\in C^6$ from 
Theorem \ref{tsympor}} 

Let $\gamma\subset\Sigma$ be a germ of $C^6$-smooth 
curve with Poritsky property. Let $\Gamma_\var$ be the corresponding family of string curves. 
Let $\wt F_\var:=\wt\delta_{+,\var}$ be the billiard ball maps (\ref{prinv0}) 
defined by reflections from the curves $\Gamma_\var$; see also (\ref{bilike}). 
We write them in coordinates $(s_\var,\phi_\var)$ associated to $\Gamma_\var$ on the space of 
oriented geodesics, see Proposition \ref{bma}.  The curves $\Gamma_\var$ form a foliation tangent to 
a $C^2$-smooth line field on the closure of the 
concave domain adjacent to $\gamma$. Their 
the 3-jets depend continuously on 
points. Both statements follow from Theorem \ref{thsm}. This  implies that 
the mappings $\wt F_\var=\wt\delta_{+,\var}(s_\var,\phi_\var)$   have derivatives of order up to 2 
that are continuous in $(s_\var,\phi_\var,\var)$. Therefore,  the corresponding maps 
$F_\var:=\delta_{+,\var}=\delta_\var(s_\var,y_\var)$, $y_\var=1-\cos\phi_\var$, see (\ref{ycos}), (\ref{d+}), (\ref{deltsy}), are (strongly) billiard-like. 
 
 The  maps $F_\var$ have invariant curves $\gamma_\var$, which are 
identified with the family of  geodesics tangent to the curve 
$\gamma$ and oriented as $\gamma$. In the coordinates $(s_\var,y_\var)$ the curves 
$\gamma_\var$ are graphs of continuous functions converging to 
zero uniformly, by construction. 

Let now $\gamma$ have string Poritsky property. 
Then the Poritsky parameter $t$ induces a parameter 
denoted by $t_\var$ on each invariant curve 
$\gamma_\var$: by definition, 
the value of the parameter $t_\var$ at a geodesic tangent to $\gamma$ 
is the value of the Poritsky parameter $t$ at the tangency point. 
The map $\delta_\var:\gamma_\var\to\gamma_\var$ act
 by translations in the parameters $t_\var$. The 
 parameters $t_\var$ obviously converge uniformly to 
 the Poritsky parameter $t=t_0$ of the curve 
 $\gamma=\Gamma_0$, as $\var\to0$. Therefore, 
 the billiard ball maps $F_\var$ satisfy the conditions 
 of Theorem \ref{tsympor} with $w=2\sqrt 2\kappa^{-1}$, 
 see Example \ref{exdel}. This together with Theorem 
 \ref{tsympor} implies that $t_0=at_L+b$, 
 $a,b\equiv const$, and proves Theorem \ref{ttk} 
 in the case, when the metric and the curve $\gamma$ are $C^6$-smooth.

\section{Osculating curves with string Poritsky property. Proof of Theorem \ref{uniq4}}
Here we prove Theorem \ref{uniq4}, which states that  a germ   of curve with string Poritsky property 
is uniquely determined by its  4-jet.

\subsection{Cartan distribution, a generalized version of Theorem 
\ref{uniq4} and plan of the section} 

Everywhere below for a curve (function) $\gamma$ by $j^r_p\gamma$ we denote 
its $r$-jet at the point $p$. Set 
$$\mcf^r:=\text{ the space of } r \text{-jets of functions of one variable } x\in\rr.$$ 

Let $\Sigma$ be a $C^m$-smooth two-dimensional manifold. For every $r\in\zz_{\geq0}$, $r\leq m$, set 
$$\mcj^r=\mcj^r(\Sigma):= \text{ the space of } r \text{-jets of regular curves in } \Sigma.$$ 
In more detail, a {\it germ of  regular curve} is 
 the graph of a germ of function $\{ y=h(x)\}$ 
in appropriate local chart 
$(x,y)$. 
Locally a neighborhood  in $\mcj^r$ of the jet of a given $C^r$-germ of regular curve 
is thus identified with a neighborhood of a jet in $\mcf^r$. One has $\dim\mcf^r=\dim\mcj^r=r+2$. There are local coordinates $(x,b_0,\dots,b_r)$ on $\mcf^r$ defined by the condition 
 that for every jet $j^r_ph\in\mcf^r$ one has 
 \begin{equation}
 x(j^r_ph)=p, \ b_0(j^r_ph)=h(p), \ b_1(j^r_ph)=h'(p),\dots,b_r(j^r_ph)=h^{(r)}(p).
 \label{cbj}\end{equation}
Recall that the {\it $r$-jet extension} of a function (curve) is the curve in the 
jet space $\mcf^r$ (respectively, $\mcj^r$) consisting of its $r$-jets at all points.

\begin{definition} (see an equivalent definition in \cite[pp.122--123]{olver}). Consider the space $\mcf^r$ equipped with the above coordinates $(x,b_0,\dots,b_r)$. 
The {\it Cartan (or contact) distribution} $\mcd_r$ on $\mcf^r$ is the field of two-dimensional 
subspaces in its tangent spaces defined by 
the system of Pfaffian equations
\begin{equation}db_0=b_1dx, \ db_1=b_2dx, \ \dots, \ db_{r-1}=b_rdx.\label{candis}
\end{equation}
For every $C^m$-smooth surface $\Sigma$ and every $r\leq m$ the 
{\it Cartan (or contact) distribution (plane field) on } $\mcj_r$, which is also denoted by $\mcd_r$, 
is defined by (\ref{candis}) locally on its domains identified with open subsets in 
$\mcf^r$; the distributions (\ref{candis}) defined on intersecting 
domains $V_i$, $V_j$ with respect to different charts $(x_i,y_i)$ and $(x_j,y_j)$ 
coincide and yield a global plane field on $\mcj_r$. 
\end{definition}
\begin{remark} Recall that the $r$-jet extension of 
each function (curve) is tangent to the Cartan distribution.
\end{remark}

\def\mcp{\mathcal P}

The main result of the present section is the following theorem, which immediately implies Theorem \ref{uniq4}. 
Proofs of both theorems will be given in Subsection 8.7. 

\begin{theorem} \label{osc} Let $\Sigma$ be a  
two-dimensional surface with a $C^6$-smooth Riemannian metric. 
There exists  a 
$C^1$-smooth line field $\mcp$ on $\mcj^4=\mcj^4(\Sigma)$ lying in the 
Cartan plane field $\mcd_4$ such that the 4-jet extension of every 
$C^5$-smooth curve on $\Sigma$ with string Poritsky property (if any) is a phase 
curve of the field $\mcp$. 
\end{theorem}

Let $\gamma$ be a germ of curve with string Poritsky property at a point $O\in\Sigma$. The Poritsky--Lazutkin parameter 
$t$  on $\gamma$ is given  by already known formula (\ref{dtk}). We normalize it by additive and multiplicative constants 
  so that $t(O)=0$ and $\frac{dt}{ds}(O)=\kappa(O)$, see (\ref{ltt}). We identify 
points of the curve $\gamma$ with the corresponding values of the parameter $t$. Consider 
the function $L(A,B)$ defined in (\ref{lab}). Let $t(A)=a$, $t(B)=a+\tau$. Poritsky property implies that 
the function $L(a,a+\tau)=L(0,\tau)$ is independent on $a$. In particular, the function 
\begin{equation}\La(t):=L(0,t)-L(-t,0)\label{lat}\end{equation} 
vanishes.  For the proof of Theorem \ref{osc} we show (in the Main Lemma stated in Subsection 8.2) that for every odd $n>3$ 
 the "differential equation" $\La^{(n+1)}(0)=0$ is equivalent to an equation saying that the coordinate $b_n=b_{n-1}'$ 
 of the $n$-jet of the curve $\gamma$ is equal to a function of  the other coordinates $(x,b_0,\dots,b_{n-1})$. 
 For $n=5$ this  yields an ordinary 
 differential equation on $\mcj^4$ satisfied by 
 the 4-jet extension of the curve $\gamma$. It will be represented by a line field contained in $\mcd_4$.  
 
The proof of the Main Lemma takes the most of the section. For its proof we study (in Subsection 8.3) two germs of  curves $\gamma$ 
and $\gamma_{n,b}$ at a point $O$ having contact of order $n\geq3$. More precisely, they are graphs of functions 
$y=h(x)$ and $y=h_{n,b}(x)$ such that $h_{n,b}(x)-h(x)=bx^n+o(x^n)$.  We show that the corresponding 
functions $\La(t)$ and $\La_{n,b}(t)$ differ by  $c_nbt^{n+1}+o(t^{n+1})$, with $c_n$ being a known  constant 
depending on the second jet of the curve $\gamma$; $c_n\neq0$ for odd $n>3$. To this end, we consider a local normal chart $(x,y)$ centered at $O$ with $x$-axis being tangent to $\gamma$ at $O$. We 
compare different quantities related to both curves, all of them being considered as functions of $x$: 
the natural parameters, the curvature etc.   In Subsection 8.4 we show that the asymptotic 
Taylor coefficients of order $(n+1)$  
of the functions $L(0,t)$ and $\La(t)$ depend only on the $n$-jet of the 
metric at $O$. We show in Subsection 8.5 that the 
above Taylor coefficients are analytic functions of the $n$-jets of metric 
and the curve (using results of Subsections 8.3 and 8.4). In Subsection 8.6 we show that the degree $n+1$ coefficient of the function $\La(t)$ is a linear non-homogeneous 
function in $b_{n}=b_n(\gamma)$ with coefficients depending on $b_j$, $j<n$; the coefficient at $b_n$ being 
equal to $\sigma_n=n!c_n$ (using results of Subsection 8.3). This will prove the Main Lemma. 

 \subsection{Differential 
equations in jet spaces  and the Main Lemma}

Let $s$ denote the natural orienting  
length parameter of the curve $\gamma$, $s(O)=0$. Let $\kappa$ be 
its geodesic curvature considered as a function $\kappa(s)$, and let $\kappa>0$.  
We already know that if the curve $\gamma$ has  string Poritsky property, then its Poritsky--Lazutkin parameter $t$ is 
expressed as a function of a point $Q\in\gamma$ in terms of the  parameter $s$ via formula (\ref{lazparl}), 
up to constant factor and additive constant, which can be chosen arbitrarily. We normalize it as follows:
\begin{equation}t(Q):=\kappa^{\frac13}(0)\int_0^{s(Q)}\kappa^{\frac23}(s)ds\label{ltt}\end{equation}
We can define the parameter $t$ given by (\ref{ltt}) on any curve $\gamma$, not necessarily having Poritsky property. We identify the points of the curve $\gamma$ with the corresponding values of the parameter $t$; 
thus, $t(O)=0$. 

\begin{remark} \label{rtinv} 
The parameter $t$ on a curve $\gamma$ given by (\ref{ltt}) is invariant under rescaling of the 
metric by constant factor. This follows from the fact that if the norm induced by the metric is multiplied by 
a constant factor $C$, then the Levi-Civita connexion 
remains unchanged, the unit tangent vectors $\dot\gamma$ are divided by $C$, and the geodesic curvature 
$||\nabla_{\dot\gamma}\dot\gamma||$ of the curve $\gamma$ considered as a function of a point in 
$\gamma$ is divided by $C$.
\end{remark}

Let $G=G(0)$ denote the geodesic tangent to $\gamma$ at its base point $O$. 
We will work in normal coordinates $(x,y)$ centered at $O$, in which 
$G$ coincides with the $x$-axis. For every $t$ let $G(t)$ denote the geodesic 
tangent to $\gamma$ at the point $t$, and let $C(t)$ denote the point of the intersection 
$G\cap G(t)$. 

Let $L(A,B)$ the function of $A,B\in\gamma$ defined in (\ref{lab}). 
We consider $L(A,B)$  as a function of the corresponding parameters 
$t(A)$ and $t(B)$, thus,  
\begin{equation}L(0,t)=L(O,\gamma(t))=|OC(t)|+|C(t)\gamma(t)|-\lambda(0,t),
\label{lot}\end{equation}
where  $\lambda(0,t)=\lambda(O,\gamma(t))$ 
is the length of the arc $O\gamma(t)$ of the curve $\gamma$. 

\def\mci{\mathcal I}

The main part of the proof of Theorem \ref{osc} is the following lemma. 

\begin{lemma} \label{lmjet} {\bf (The Main Lemma).} Let $n\in\nn$, $n\geq3$. Let $\Sigma$ be a  surface equipped with 
a $C^{n+1}$-smooth Riemannian metric. Let $E\in\Sigma$, and let $(x,y)$ be  coordinates 
on $\Sigma$ centered at $E$ and parametrizing some its neighborhood 
$V=V(E)\subset\Sigma$. Let  $\mcj^n_y(V)$ 
denote  the space of $n$-jets of curves in $V$ that are graphs of $C^{n}$-smooth functions $y=y(x)$;  
thus, it is naturally identified with an open subset   $\mcf^n_y(V)\subset\mcf^n$. Let 
$(x,b_0,\dots,b_n)$ denote the corresponding coordinates on 
$\mcf^n_y(V)\simeq \mcj^n_y(V)$ 
given by (\ref{cbj}). Set 
$$J_2:=(x,b_0,b_1,b_2).$$
There exist  $C^{1}$-smooth functions $\sigma_n(J_2)$ and $P_n(J_2; b_3,\dots,b_n)$,  
\begin{equation}\sigma_n\neq0 \text{ for  odd } n>3; \ \sigma_n\equiv0 \text{ for } 
n=3 \text{ and for every even } n>3,\label{sigmaneq}\end{equation} 
such that every jet  $J_n=(x,b_0,\dots,b_n)\in\mcj^n_y(V)$ extending $J_2$ 
satisfies the following statement. Let $\gamma$ be a $C^n$-smooth 
germ of curve representing  the jet $J_n$, and let $t$ be the parameter on $\gamma$ defined by (\ref{ltt}). Let $L(0,t)$ be the same, as in (\ref{lot}). The 
corresponding function  $\Lambda(t)$  from (\ref{lat}) admits an  asymptotic Taylor formula of degree $n+1$ at $0$ of the following type: 
\begin{equation}\Lambda(t)=\sum_{k=3}^{n+1}\wh\Lambda_kt^k+o(t^{n+1}),\label{tayfor}\end{equation}
\begin{equation} \wh\Lambda_{n+1}=\sigma_n(J_2)b_n-P_n(J_2;b_3,\dots,b_{n-1}).\label{lan+1}\end{equation}
 \end{lemma}
 \begin{definition} A {\it pure} $n$-jet of  curve $\gamma$  in $\rr^2$ is a class of $n$-jets of curves modulo 
translations. It is identified with the collection of Taylor coefficients   of the germ of function $h(x)$ defining $\gamma=\{ y=h(x)\}$ at monomials of  
degrees from 1 to $n$. 
A {\it pure} $n$-jet of metric on a planar domain 
is a class of $n$-jets of metrics modulo translations. It is identified with the 
collection of Taylor coefficients  of the metric tensor at monomials of degrees from 0 to $n$. 
\end{definition}

{\bf Addendum to Lemma \ref{lmjet}.} 
{\it The function $\sigma_n$ depends analytically on the pure 
1-jet of the metric and the pure 2-jet of the curve. The function 
$P_n$ depends analytically 
on the pure $n$-jet of the metric and the pure $(n-1)$-jet of the curve. 
The function $\sigma_n$ is defined by the following formula. 
Set $u=u(J_2):=(1,b_1)\in T_{(x,b_0)}\Sigma$. Let  
$w\in T_{(x,b_0)}\Sigma$ denote  
the image of the vector $\frac{\partial}{\partial y}\in T_{(x,b_0)}\Sigma$ 
under the orthogonal projection to the line $\rr u^{\perp}$.   Let $\kappa=\kappa(J_2)$ denote the geodesic curvature 
of a curve $\gamma$ representing the jet $J_2$ (it depends on the pure 
1-jet of the curve and  the pure 1-jet of the metric). Then for every odd $n\geq3$ } 
\begin{equation}\sigma_n(J_2)=
\frac{(n-2)(n-3)}{6(n+1)!}||w||(||u||\kappa(J_2))^{-n}.
\label{sigmanform}\end{equation}

Lemma \ref{lmjet} and its addendum will be proved in Subsection 8.6.

\subsection{Comparison of functions $L(0,t)$ and $\La(t)$ for osculating curves}

Let $n\geq3$. Let $\Sigma$ be a surface equipped with a Riemannian metric, $O\in\Sigma$. 
Let us consider normal coordinates $(x,y)$ centered at $O$. We consider that the metric 
under question is $C^2$-smooth in the normal coordinates. 
Let $b\in\rr$, and let $\gamma,\gamma_{n,b}\subset\Sigma$ be two germs of $C^n$-smooth curves 
at $O$ with the same $(n-1)$-jet that are tangent to the $x$-axis at $O$,
$$\gamma=\{ y=h(x)\}, \ \gamma_{n,b}=\{ y=h_{n,b}(x)\}, \ h_{n,b}(x)=f(x)+bx^n+o(x^n).$$
Here $o(x^n)$ is a function tending to zero together with its derivatives up to order $n$, as 
$x\to0$. The geodesic curvatures of the curves $\gamma$ 
and $\gamma_{n,b}$ at $O$  are equal to the same number $\kappa(O)=h''(0)=h''_{n,b}(0)$, by (\ref{x122}). 
Without loss of generality we consider that $\kappa(O)=1$. One can achieve this by rescaling the 
norm of the metric by constant factor $\kappa(O)$, see Remark \ref{rtinv}.

The main result of the present subsection is the following lemma.
\begin{lemma} \label{lcomp} In the above conditions let $t$ be the parameter 
on $\gamma$ given by (\ref{ltt}).  Let $L(0,t)$, $L_{n,b}(0,t)$ and 
$\La(t)$, $\La_{n,b}(t)$ be the functions from (\ref{lat}) 
defined for the curves $\gamma$ and $\gamma_{n,b}$ respectively. One has 
\begin{equation} L_{n,b}(0,t)-L(0,t)= \frac{(n-2)(n-3)}{12(n+1)}bt^{n+1}+o(t^{n+1}), \text{ as } t\to0,\label{lnb}
\end{equation}
\begin{equation}\La_{n,b}(t)-\La(t)=\begin{cases}\frac{(n-2)(n-3)}{6(n+1)}bt^{n+1}
  +o(t^{n+1}), \text{ if } n \text{ is odd,}\\
  o(t^{n+1}), \text { if } n \text{ is even.}\end{cases}\label{dlat}\end{equation}
\end{lemma}

For the proof of Lemma \ref{lcomp} we first 
compare the natural parameters $s(x)$, $s_{n,b}(x)$ centered at $O$ and the parameters $t(x)$, $t_{n,b}(x)$ 
given by (\ref{ltt}) for the curves $\gamma$ 
and $\gamma_{n,b}$ as functions of $x$. We also compare the corresponding inverse functions 
$x=x(t)$ and $x=x_{n,b}(t)$ as functions of $t$, see Proposition \ref{sktx} below. Afterwards 
we prove formula (\ref{lnb}) using the above-mentioned comparison results and the 
results of Section 2. Then we deduce (\ref{dlat}).

\begin{proposition} \label{sktx} As $x\to0$ (or equivalently, $t\to0$), one has 
\begin{equation}t(x)\simeq t_{n,b}(x)\simeq x,  \ \  x(t)\simeq t\simeq h'(x(t)),\label{xte}\end{equation}
\begin{equation}s_{n,b}(x)-s(x)=\frac n{n+1}bx^{n+1}+o(x^{n+1}),\label{snb}\end{equation}
\begin{equation}\kappa_{n,b}(x)-\kappa(x)=n(n-1)bx^{n-2}+o(x^{n-2}),\label{knb}\end{equation}
\begin{equation} t_{n,b}(x)-t(x)=\frac{2n}3bx^{n-1}+o(x^{n-1}),\label{tnb}\end{equation}
 \begin{equation} x_{n,b}(t)-x(t)=-\frac{2n}3bt^{n-1}+o(t^{n-1}).\label{xbnt}\end{equation}
 Here $o(x^k)$, $o(t^k)$ are functions that tend to zero together with their derivatives up to order $k$, 
 as $x\to0$ ($t\to0$). 
\end{proposition}
\begin{proof} Formulas (\ref{xte}) 
follow from (\ref{ltt}), since $\kappa(O)=1$. In the parametrizations $\gamma=\gamma(x)$, $\gamma_{n,b}=\gamma_{n,b}(x)$ one has 
\begin{equation}s(x)=\int_{0}^x||\dot\gamma(u)||du, \ s_{n,b}(x)=\int_{0}^x||\dot\gamma_{n,b}(u)||du. 
\label{snnb}\end{equation}
We claim that 
\begin{equation}||\dot\gamma_{n,b}(x)||-||\dot\gamma(x)||=nbx^n+o(x^{n}).\label{asdif}\end{equation}
Indeed, let us identify the tangent spaces $T_{(x,y)}\Sigma$ at different points  $(x,y)$ by translations. 
One has $\dot\gamma(x),\dot\gamma_{n,b}(x)=(1,x+o(x))$, 
\begin{equation} v(x):=\dot\gamma_{n,b}(x)-\dot\gamma(x)=
(0,nbx^{n-1}+o(x^{n-1})):\label{defvas}\end{equation}
$h'(x)\simeq x$, since $h''(0)=\kappa(O)=1$, by assumption. The metric has trivial 1-jet at the base point $O$. 
Therefore, the difference of metric tensors at the $O(x^n)$-close points 
$\gamma(x)$, $\gamma_{n,b}(x)$, which are $O(x)$-close to $O$, is $O(x^{n+1})$.  
Hence, it suffices to prove (\ref{asdif}) for the  vector $\dot\gamma_{n,b}(x)$ 
  being translated to the point $\gamma(x)$. The Euclidean angle $\alpha$ between 
the vectors $v(x)$ and $\dot\gamma(x)$ is $\frac{\pi}2-x+o(x)$, by (\ref{defvas}). Therefore, the angle 
between them in the metric of the tangent plane $T_{\gamma(x)}\Sigma$ has the same 
asymptotics. Hence, 
$$||\dot\gamma_{n,b}(x)||^2=||v+\dot\gamma(x)||^2=
||\dot\gamma(x)||^2+2nbx^n+o(x^n),$$
by Cosine Theorem. The latter formula together with the obvious formula $||\dot\gamma(x)||=1+O(x)$ imply (\ref{asdif}), which 
together with (\ref{snnb}) implies (\ref{snb}). 

Let us prove (\ref{knb}).  
The Christoffel symbols at the $O(x^n)$-close points $\gamma(x)$ 
and $\gamma_{n,b}(x)$ are $O(x^n)$-close, as in the above discussion. Therefore, the difference 
$\kappa_{n,b}(x)-\kappa(x)$ is equal up to $O(x^n)$ to the same difference, where each curvature is calculated in the 
metric (Christoffel symbols) of the point $\gamma(x)$. The difference of the Christoffel  parts of 
the curvatures, which are quadratic in the vectors $\frac1{||\dot\gamma(x)||}\dot\gamma(x)$, 
$\frac1{||\dot\gamma_{n,b}(x)||}\dot\gamma_{n,b}(x)$, is $O(||v(x)||)=O(x^{n-1})$, by (\ref{asdif}).  
The difference of their second derivative terms  is equal to 
$h_{n,b}''(x)-h''(x)+O(x^n)=n(n-1)bx^{n-2}+
o(x^{n-2})$, by definition and (\ref{asdif}). This together with the above discussion implies (\ref{knb}). 

Let us prove (\ref{tnb}). One has 
$$t_{n,b}(x)-t(x)=\int_0^x(\kappa_{n,b}^{\frac23}(u)||\dot\gamma_{n,b}(u)||-\kappa^{\frac23}(u)||\dot\gamma(u)||)
du$$
$$\simeq\int_0^x(\kappa_{n,b}^{\frac23}(u)-\kappa^{\frac23}(u))||\dot\gamma(u)||du+O(x^n),$$
by definition and (\ref{asdif}). The latter right-hand side is asymptotic to 
$\frac23\int_0^xn(n-1)bu^{n-2}du=\frac{2n}3bx^{n-1}$, by (\ref{knb}). This proves (\ref{tnb}). 

Formula (\ref{xbnt}) follows from (\ref{tnb}). Proposition \ref{sktx} is proved.
\end{proof}

In the proof of formula (\ref{lnb}) we use the following notations:
$$P=P(t):=\gamma(t), \ Q=Q(t):=(x_{n,b}(t),h(x_{n,b}(t)))\in\gamma, \ 
A=A(t):=\gamma_{n,b}(t),$$
$$G(t):= \text{ the geodesic tangent to } \gamma \text{ at } P, \ G(0)=\text{ the } x-\text{axis,}$$
$$C=C(t):=G(t)\cap G(0), \ V=V(t):=\{ x=x_{n,b}(t)\},\  B=B(t):=  G(t)\cap V,$$
$$G_{n,b}(t):=\text{ the geodesic tangent to } \gamma_{n,b} \text{ at } A, \ 
D=D(t):=G_{n,b}(t)\cap G(0),$$
see Fig. 7. By definition,  $Q=Q(t)=\gamma\cap V$. 
\begin{figure}[ht]
  \begin{center}
   \epsfig{file=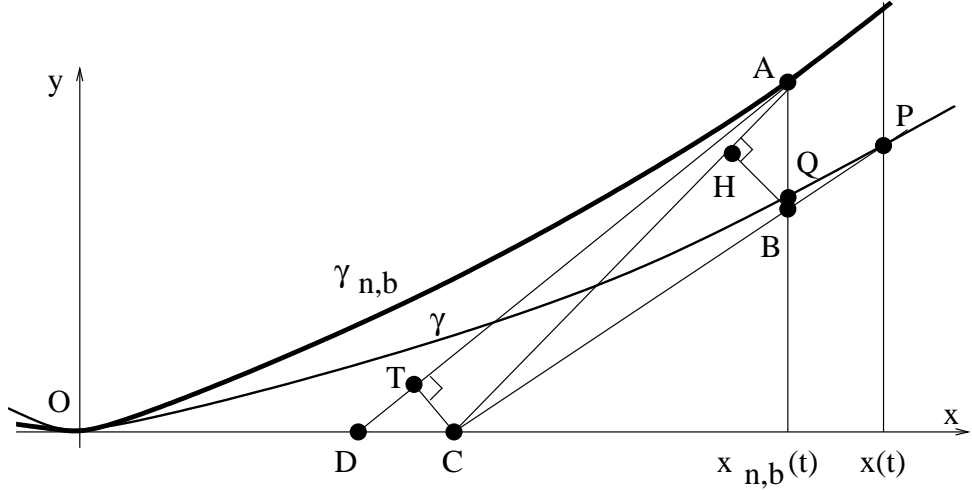}
    \caption{Auxiliary geodesics for calculation of the asymptotic of the difference $L_{n,b}(0,t)-L(0,t)$.}
    \label{fig:0}
  \end{center}
\end{figure}
 In what follows for any two points $E,F\in\Sigma$ close to $O$ 
the length of the geodesic segment connecting $F$ to $E$ will be denoted by $|EF|$. 
By definition, 
\begin{equation}L(0,t)=|OC|+|CP|-\la(O,P), \ L_{n,b}(0,t)= |OD|+|DA|-\la(O,A).\label{ltform}
\end{equation}
Recall that $\la(O,A)$, $\la(O,P)$ are lengths of arcs $OA$ and $OP$ of the curves 
$\gamma_{n,b}$ and $\gamma$ respectively. Set 
$$L_1=L_1(t):=|OC|+|CB|-\la(O,Q), \ L_2=L_2(t):=|OC|+|CA|-\la(O,A),$$
$$\Delta_1=\Delta_1(t):=L_1(t)-L(0,t)=\la(Q,P)-|BP|,$$
$$ \Delta_2=\Delta_2(t):=L_2(t)-L_1(t),$$
$$\Delta_3=\Delta_3(t):=L_{n,b}(0,t)-L_2(t):$$
\begin{equation} L_{n,b}(0,t)-L(0,t)=\Delta_1(t)+\Delta_2(t)+\Delta_3(t).\label{delty}
\end{equation}

In what follows we find   asymptotics of each $\Delta_j$.

\begin{proposition} \label{de1} One has 
\begin{equation}\Delta_1(t)=O(t^{2n-1})=O(t^{n+2}) \text{ whenever } n\geq3.
\label{d1}\end{equation}
\end{proposition}
\begin{proof}  In the curvilinear triangle $QPB$ with $QP\subset\gamma$, 
$PB$ being geodesic and $QB$ vertical segment one has 
$|PB|=O(x_{n,b}(t)-x(t))=O(t^{n-1})$, by (\ref{xbnt}). Its 
angle at $Q$ is $\frac{\pi}2+O(t)$.  Therefore, by (\ref{tricurve}), 
$$\Delta_1=\la(Q,P)-|PB|=O(|PB|^3)+O(t|PB|^2)=O(t^{3n-3})+O(t^{2n-1}).$$
The latter right-hand side is $O(t^{2n-1})=O(t^{n+2})$, since $n\geq3$.
\end{proof}
\begin{proposition} \label{de2} One has 
\begin{equation}\Delta_2(t)=\frac b{n+1}t^{n+1}+o(t^{n+1}).\label{d2}\end{equation}
\end{proposition}
\begin{proof} By definition,
$$\Delta_2(t)=|OC|+|CA|-\la(O,A)-(|OC|+|CB|-\la(O,Q))$$
\begin{equation}=
(|CA|-|CB|)-(\la(O,A)-\la(O,Q)),\label{d2ca}\end{equation}
\begin{equation}\la(O,A)-\la(O,Q)=s_{n,b}(x_{n,b}(t))-s(x_{n,b}(t))=\frac{nbt^{n+1}}{n+1}+o(t^{n+1}),\label{saq}\end{equation}
by (\ref{snb}) and (\ref{xte}). To find the asymptotics of the difference 
$|CA|-|CB|$, 
let us consider the height denoted by $BH$ of the geodesic triangle $ABC$, which splits it 
into two triangles $ABH$ and $CBH$, see Fig. 7. We use the following asymptotic formula for  
lengths of their sides:
\begin{equation} |AB|\simeq bt^n+o(t^n)\simeq |BH|,\label{abas}\end{equation}
\begin{equation} |CB|\simeq |CP|\simeq |CA|\simeq\frac t2\label{t2}\end{equation}
\begin{equation} |AH|\simeq bt^{n+1}+o(t^{n+1})\simeq |AC|-|BC|.\label{ahc}\end{equation}
\begin{proof} {\bf of (\ref{abas}).} The Euclidean distance in the coordinates 
$(x,y)$ 
between the points $A$ and $Q$ is  
$bx^n_{n,b}(t)+o(x^n_{n,b}(t))= bt^n+o(t^n)$, by construction. Therefore, the distance between them in the metric $g$ is asymptotic 
to the same quantity, since $g$ is Euclidean on $T_O\Sigma$. The Euclidean distance 
between the points $Q$ and $B$ is of order $O((x(P)-x(B))^2)\simeq O(t^{2(n-1)})=O(t^{n+1})$, 
by (\ref{xbnt}) and since $n\geq3$: $2(n-1)\geq n+1$ for $n\geq3$. The two latter 
statements together imply that $|AB|= bt^n+o(t^n)$; this is the first asymptotics in (\ref{abas}).
 
In  the proof of the second asymptotics in (\ref{abas}) and in what follows 
we use the two next claims.

{\bf Claim 1.} {\it The azimuths of the tangent vectors of the geodesic arcs $CA$, $CP$, $DA$  
at all their points  are uniformly asymptotically equivalent to  $t=t(P)$, as $t\to0$.}

\begin{proof}  Let us prove the above statement  for the geodesic arc $CP$; 
the proof  for the arcs $CA$ and $DA$ is analogous. The slope of 
the tangent vector to the curve $\gamma$ at the point $P$ is asymptotic to 
$x(P)=x(t)\simeq t$, and it is equal to the slope of the tangent vector of the geodesic 
$CP$ at $P$. On the other hand, let us 
apply formula (\ref{azim})  to the geodesic arc $\alpha=CP$: its right-hand side is 
a quantity of order $O(t)$. The length of the arc $CP$ is $O(t)$. 
Hence, the difference between the azimuths of tangent 
vectors at any two points of the geodesic arc $CP$ is of order $O(t^2)$. This proves the 
claim.
\end{proof}

{\bf Claim 2.} {\it The angle $A$ of the geodesic triangle $ABH$ is asymptotic to 
$\frac{\pi}2-t+O(t^2)$. Its angle $B$ is asymptotic to $t+O(t^2)$, 
and $|AH|\simeq t|AB|$.}

\begin{proof}
The  first statement of the 
claim follows from Claim 1 applied to $CA$ and the fact that the slopes of 
the tangent vectors to the geodesic arc $BA$ are uniformly 
$O(|BA|)=O(t^n)$-close to $\frac{\pi}2$. 
This follows from the second formula in (\ref{labu}) and 
formula (\ref{azim}) applied to the geodesic arc $BA$. 
The second statement 
of the claim follows from the first one and  (\ref{abc}).
\end{proof}

The first statement of Claim 2 implies that $|AB|\simeq |HB|$, which yields
 the second asymptotics in (\ref{abas}). Formula (\ref{abas}) is proved.
 \end{proof}
 
 \begin{proof} {\bf of (\ref{t2}).} The asymptotics $|CP|\simeq\frac{x(P)}2\simeq\frac t2$ 
 follows from Claim 1 and the fact that the height of the point $P$ over the $x$-axis is 
 asymptotic to $\frac{x^2(P)}2\simeq \frac{t^2}2$. The other asymptotics in (\ref{t2}) 
 follow from the above one, formula (\ref{abas}) and 
  the fact that $|BP|=O(t^{n-1})$ (follows from (\ref{xbnt})).\end{proof}
 
 \begin{proof} {\bf of (\ref{ahc}).} The geodesic triangle $ABH$ has right angle at $H$. 
 This together with Claim 2 and (\ref{abas}) 
 implies the first asymptotic formula in (\ref{ahc}). In the proof of the second formula in (\ref{ahc}) we use 
 the following claim. 
 
 {\bf Claim 3.} {\it The angle $\phi:=\angle BCH$ equals $2bt^{n-1}+o(t^{n-1})$.}
 
 \begin{proof} The triangle $BCH$ has right angle at $H$,
  $|BH|=bt^n+o(t^n)$, $|BC|\simeq\frac t2,$ 
 by (\ref{abas}) and (\ref{t2}). Hence,  
 $\phi\simeq |BH|\slash\frac t2=2bt^{n-1}+o(t^{n-1})$. 
  \end{proof}
  
  Now let us prove the second asymptotic formula in (\ref{ahc}).  One has 
  $$|HC|-|BC|\simeq \frac12BC\phi^2,$$
   by formula (\ref{abc}) applied to the 
  family of triangles $BCH$. The right-hand side in the latter formula is  $b^2t^{2n-1}+o(t^{2n-1})=O(t^{n+2})$, by (\ref{t2}) and Claim 3 and since 
  $2n-1\geq n+2$ for $n\geq3$. Thus, 
  \begin{equation}|HC|-|BC|=O(t^{n+2}),\label{hcbc}\end{equation} 
  $$|AC|-|BC|=(|HC|-|BC|)+|AH|=|AH|+O(t^{n+2})=bt^{n+1}+o(t^{n+1}),$$
  by the first formula in (\ref{ahc}) proved above. Formula (\ref{ahc}) is proved. 
  \end{proof}
 Substituting formulas (\ref{saq}) and (\ref{ahc}) to (\ref{d2ca}) yields
$$\Delta_2(t)=bt^{n+1}-\frac n{n+1}bt^{n+1}+o(t^{n+1})=\frac b{n+1}t^{n+1}
+o(t^{n+1}).$$
Proposition \ref{de2} is proved.
 \end{proof} 
 \begin{proposition} \label{de3} One has 
 \begin{equation}\Delta_3(t)= \frac{n-6}{12}bt^{n+1}+o(t^{n+1}).\label{d3}\end{equation}
 \end{proposition}
 \begin{proof} Recall that 
 $$\Delta_3(t)=L_{n,b}(0,t)-L_2(t)=|OD|+|DA|-\la(O,A)-(|OC|+|CA|-\la(O,A))$$
 \begin{equation}=|DA|-(DC+|CA|).\label{d3r}\end{equation}
 Here $DC$  is the "oriented length" $DC:=|OC|-|OD|$. 
 
  Let  $CT$ denote the height of the geodesic triangle $DCA$. To find an asymptotic formula for the right-hand side in (\ref{d3r}), we first find asymptotics of the 
 length of the height $CT$ and the angle $\angle DAC$.

 {\bf Claim 4.} {\it Let $\alpha:=\angle DAC$ denote the oriented angle between the geodesics $AD$ and $AC$: it is said to be positive, if 
 $D$ lies between $O$ and $C$, as at Fig.7. 
 One has $\alpha=\frac{6-n}3bt^{n-1}+o(t^{n-1})$.}
 
 \begin{proof} Consider the following 
 tangent lines of the geodesic arcs $AD$, $AC$, $BC$, $CP$ and the curve $\gamma$:
 $$\ell_1:=T_AAD=T_A\gamma_{n,b}, \ \ell_2:=T_AAC, \ \ell_3:=T_BBC,$$
 $$\ell_4:=T_Q\gamma, \ \ell_5:=T_PCP=T_P\gamma.$$
 We orient all these lines "to the right". One has 
 \begin{equation}\alpha\simeq\az(\ell_2)-\az(\ell_1),\label{diffl}\end{equation}
 by definition and since the Riemannian metric at the point $A$ written in the normal coordinates $(x,y)$  tends to the Euclidean one, as $t\to0$. Let us find asymptotic 
 formula for the above difference of azimuths by comparing azimuths of appropriate 
 pairs of lines $\ell_1,\dots,\ell_5$. One has 
 $$\az(\ell_4)-\az(\ell_1)=-nbt^{n-1}+o(t^{n-1}),$$
 since the above azimuth difference  is asymptotically equivalent to the 
 difference of the derivatives of the functions $h(x)$ and 
 $h_{n,b}(x)=h(x)+bx^n+o(x^{n})$ 
 at the same point $x=x(B)\simeq t$: hence, to $-nbx^{n-1}+o(x^{n-1})$. One has 
 $$\az(\ell_5)-\az(\ell_4)\simeq h'(x(t))-h'(x_{n,b}(t))\simeq x(t)-x_{n,b}(t)=
 \frac{2n}3bt^{n-1}+o(t^{n-1}),$$
 by (\ref{xbnt}) and since the function $h'(x)\simeq x$ has unit derivative at $0$,
 $$\az(\ell_3)-\az(\ell_5)=O(t(x(B)-x(P)))=O(t(x_{n,b}(t)-x(t)))=O(t^n),$$
  by (\ref{azim}) and (\ref{xbnt}),  
  $$\az(\ell_2)-\az(\ell_3)\simeq\angle BCA=2bt^{n-1}+o(t^{n-1}),$$
  by (\ref{azas}), (\ref{azim}) and Claim 3. The right-hand sides of the 
  above asymptotic formulas for azimuth differences are all of order $t^{n-1}$, 
  except for one, which is  $O(t^{n})$. Summing up all of them yields the statement of Claim 4:
 $$\alpha\simeq\az(\ell_2)-\az(\ell_1)=\frac{6-n}3bt^{n-1}+o(t^{n-1}).$$
  \end{proof}
  
  {\bf Claim 5.} {\it In the right triangle\footnote{We treat the lengths of sides of the  triangle 
  $CDT$ as oriented lengths (without module sign): we take them with the sign equal to $\sign(\alpha)$, 
  where $\alpha$ is the same, as in Claim 4.} 
  $CDT$ \ $\angle TDC\simeq t$, 
  $CT=\frac{6-n}6bt^n+o(t^n)$,}
  \begin{equation}CD\simeq DT=\frac{6-n}6bt^{n-1}+o(t^{n-1}), \ \ CD-DT=\frac{6-n}{12}bt^{n+1}+o(t^{n+1}).\label{dct}\end{equation}
  
  \begin{proof} The angle asymptotics follows from Claim 1. 
  The length asymptotics for the side $CT$ is found via the adjacent right triangle $ACT$, from the 
  formula $CT\simeq AC\angle CAT$ after substituting   $\angle CAT=\frac{6-n}3bt^{n-1}+o(t^{n-1})$ (Claim 4) and 
  $AC\simeq \frac t2$, see (\ref{t2}). This together with 
   formula (\ref{abc}) applied to 
  the right triangle $CDT$ implies  (\ref{dct}).  
  \end{proof}
   
   Now let us prove formula (\ref{d3}). Recall that 
   \begin{equation}\Delta_3(t)=|DA|-(DC+|CA|)=(DT-DC)+(|AT|-|AC|),\label{d3n}\end{equation} 
   see (\ref{d3r}). One has $DT-DC=\frac{n-6}{12}bt^{n+1}+o(t^{n+1})$, 
   by (\ref{dct}), and 
   $|AT|-|AC|=O(t^{n+2})$, analogously to formula (\ref{hcbc}). Substituting the 
   two latter formulas to (\ref{d3n}) yields to (\ref{d3}). Proposition \ref{de3} 
   is proved.
   \end{proof} 
   
   \begin{proof} {\bf of Lemma \ref{lcomp}.} Let us prove formula (\ref{lnb}). 
  Summing up formulas (\ref{d1}), (\ref{d2}), (\ref{d3})  and substituting 
  their sum to (\ref{delty}) yields to (\ref{lnb}):
  $$L_{b,n}(t)-L(0,t)=\Delta_1(t)+\Delta_2(t)+\Delta_3(t)=\frac b{n+1}t^{n+1}+
  \frac{n-6}{12}bt^{n+1}+o(t^{n+1})$$
  $$=(\frac1{n+1}+\frac{n-6}{12})bt^{n+1}+o(t^{n+1})=\frac{(n-2)(n-3)}{12(n+1)}bt^{n+1}
  +o(t^{n+1}).$$

 Let us prove formula (\ref{dlat}). Taking curves $\gamma$ and $\gamma_{n,b}$ with opposite 
  parameter $-t$ results in multiplying the coefficient $b$ by $(-1)^n$.  
  Therefore, applying formula  (\ref{lnb}) to the parameter $-t$  yields to 
  \begin{equation} L_{n,b}(-t,0)-L(-t,0)=(-1)^n\frac{(n-2)(n-3)}{12(n+1)}bt^{n+1}
  +o(t^{n+1}).\label{sta}\end{equation}
 Thus, for odd (even) $n$ the main asymptotic terms in (\ref{sta}) and (\ref{lnb}) are opposite (respectively, coincide). 
Hence, in the expression
 $$\Lambda_{n,b}(t)-\Lambda(t)=(L_{n,b}(0,t)-L(0,t))-(L_{n,b}(-t,0)-L(-t,0))$$
 they are added (cancel out),  and we get (\ref{dlat}). Lemma \ref{lcomp} is proved.
  \end{proof}

  \subsection{Dependence of functions $L(0,t)$ and $\La(t)$ on the metric}
  Here we prove the following lemma, which shows that the  $(n+1)$-jets  
   of the quantities $L(0,t)$ and $\La(t)$ depend  
  only on the $n$-jet of the metric.
  \begin{lemma} \label{lemli} Let $n\geq3$, $\Sigma$ be a two-dimensional 
  surface.  
  Let $O\in\Sigma$, and  let 
  $\gamma\subset\Sigma$ be a germ of $C^n$-smooth curve at $O$. 
Let $g$ and $\wt g$ be two $C^n$-smooth Riemannian metrics on $\Sigma$ having the 
same $n$-jet at $O$: $\wt g(q)-g(q)=o(\dist^n(q,O))$, as $q\to O$. Then the differences 
$L_{\wt g}(0,t)-L_g(0,t)$, $\La_{\wt g}(t)-\La_g(t)$ of quantities $L(0,t)$ and $\La(t)$ 
defined by the metrics $\wt g$ and $g$ are $o(t^{n+1})$.  
\end{lemma}

\begin{proof} Let $s$, $\wt s$, $t$, $\wt t$, $\kappa$, $\wt\kappa$ denote the natural and Lazutkin parameters centered at $O$, see (\ref{ltt}), 
and the geodesic curvature of the curve $\gamma$ defined by the metrics $g$ and $\wt g$ 
respectively. Let us rescale the metric by a constant factor so that $\kappa(O)=1$. Then $\wt\kappa(O)=1$, 
since $n\geq3$. Fix a coordinate system $(x,y)$ centered at $O$ so that the $x$-axis is 
tangent to the curve $\gamma$ at $O$ and 
$||\frac{\partial}{\partial x}(O)||=1$. 
Consider $x$ as a local parameter on $\gamma$. 
We consider the above quantities as functions of $x$; $s(0)=\wt s(0)=t(0)=\wt t(0)=0$. 

Let $x(t)$, $\wt x(t)$ denote the functions inverse to $t(x)$ and $\wt t(x)$ respectively. 
Let $\gamma(t)$, $\wt\gamma(t)$ 
denote the points of the curve $\gamma$ with $x$-coordinates $x(t)$ and $\wt x(t)$ respectively. Let now $s(t)$ and $\wt s(t)$ denote the natural length parameters of the metrics $g$ and $\wt g$, 
now considered as functions of the parameter $t$ defined by the  metric 
under question ($g$ or $\wt g$). 

\begin{proposition} One has $t\simeq x\simeq\wt t\simeq s\simeq\wt s$, 
\begin{equation}\wt s(x)-s(x)=o(x^{n+1}), \ \wt\kappa(x)-\kappa(x)=o(x^{n-1}), \ \wt t(x)-t(x)=o(x^n),\label{wtsk}\end{equation}
\begin{equation}\wt x(t)-x(t)=o(t^{n}), \ \dist(\gamma(t),\wt\gamma(t))=o(t^n),\label{wxt}\end{equation} 
\begin{equation} \wt s(t)-s(t)=o(t^n), \  \wt s'(t)-s'(t)=o(t^{n-1}).\label{wstt}\end{equation}
\end{proposition}
\begin{proof} The asymptotic equivalences follow from (\ref{ltt}). The first formula in (\ref{wtsk}) is obvious. The second one holds by definition and since the Christoffel symbols of 
the two metrics differ by a quantity $o(x^{n-1})$. The third formula follows from the second one. Formula (\ref{wxt}) follows from the third formula in (\ref{wtsk}). 
Formula (\ref{wstt}) follows from (\ref{wtsk}) and (\ref{wxt}). 
\end{proof}

Fix a small value $t\in\rr$, say, $t>0$.  Set 
$$P=\gamma(t), \ A=\wt\gamma(t).$$
Let $C$ ($\wt C$) be the point of intersection of the $g$-(respectively, $\wt g$-) geodesics $G(P)$, $G(O)$ tangent to $\gamma$ at $P$ and $O$.
Let $D$ ($\wt D$) be the analogous  points of intersection of the geodesics tangent to $\gamma$ at $A$ and $O$. See Fig. 8a). The distance (arc length) 
between points $E$ and $F$ in a metric $h$ will be denoted by $|EF|_h$ 
(respectively, $\la_h(E,F)$). One has 
$$ 
L_g(0,t)=|OC|_g+|CP|_g-\la_g(O,P), \ L_{\wt g}(0,t)=|O\wt D|_{\wt g}+|\wt DA|_{\wt g}-\la_{\wt g}(O,A),$$
by definition.   Set 
\begin{equation} \Delta_1(t):=|OC|_g+|CP|_g-|OD|_g-|DA|_g-(\la_g(O,P)-\la_g(O,A));
\label{de1g}\end{equation}
\begin{equation}\Delta_2(t):=(|OD|_g-|OD|_{\wt g})+(|DA|_g-|DA|_{\wt g})-
(\la_g(O,A)-\la_{\wt g}(O,A));\label{de2g}\end{equation}
\begin{equation}\Delta_3(t):= (|OD|_{\wt g}-|O\wt D|_{\wt g})+(|DA|_{\wt g}-|\wt DA|_{\wt g}).
\label{de3g}\end{equation}
One has 
\begin{equation}L_g(0,t)-L_{\wt g}(0,t)=\Delta_1+\Delta_2+\Delta_3.\label{deltygg}\end{equation}

{\bf Claim 1.} {\it One has $\Delta_1(t)=o(t^{n+1})$.}

\begin{proof} Let us introduce the point $B$ of intersection 
of the $g$-geodesic $PC$  with the vertical line  through $A$, 
see Fig. 8a): 
$x(B)=\wt x(t)$. One has 
\begin{equation} \Delta_1=(|OC|_g+|CB|_g-|OD|_g-|DA|_g)+(|BP|_g-\la_g(A,P)).\label{de1oc}
\end{equation}
Consider the curvilinear triangle $APB$ formed by the arc $AP$ of the curve 
$\gamma$, the $g$-geodesic $PB$ and the vertical segment $BA$. 
Its sides $AP$ and $BA$ have $g$-length $o(t^n)$, 
by definition and (\ref{wxt}).  Its angle  $B$  is $\frac{\pi}2+O(\wt x(t))=\frac{\pi}2+O(t)$, as in Claim 1 in Subsection 8.3. 
This together with (\ref{tricurve}) implies that the second bracket in (\ref{de1oc}) is $o(t^{n+1})$. Let us prove the same statement for the first bracket. It is 
equal to 
\begin{equation}|DC|_g+|CA|_g-|DA|_g+(|CB|_g-|CA|_g)=DC_g+|CA|_g-|DA|_g+o(t^{n+1}),
\label{1stbr}\end{equation}
since $||CB|_g-|CA|_g|\leq|BA|=O((x(P)-x(B))^2)=o(t^{n+1})$. 
Here $DC_g$ is the oriented length $|OC|_g-|OD|_g$. One has 
\begin{equation}DC_g+|CA|_g-|DA|_g=o(t^{n+1}).\label{diffan}\end{equation}
Indeed, consider the height $CT$ of the triangle $ADC$, 
which splits it into two triangles. One has $\angle CAD=O(x(A)-x(P))=o(t^n)$, as in the proof of Claim 4 in the previous subsection. 
This together with right triangle arguments using (\ref{abc})  analogous to those from the 
proof of Claim 5 (Subsection 8.3) 
implies (\ref{diffan}). Substituting (\ref{diffan}) to (\ref{1stbr}) and then substituting everything to (\ref{de1oc}) 
yields $\Delta_1(t)=o(t^{n+1})$. Claim 1 is proved.
\end{proof}

{\bf Claim 2.}  {\it One has $\Delta_2(t)=o(t^{n+1})$.}

\begin{proof} All the points in (\ref{de2g}) are $O(t)$-close to $O$. 
The $g$- and $\wt g$-distances between any two points (which will be denoted by $E$ 
and $F$)  differ by a quantity $o(t^{n+1})$. 
Indeed, the $\wt g$-length of the $g$-geodesic segment $EF$  differs 
from its $g$-length by  $o(t^{n+1})$, since the metrics differ by $o(t^n)$. 
The distance $|EF|_{\wt g}$ is no greater than the latter $\wt g$-length, 
and hence, no greater than $|EF|_g+o(t^{n+1})$. Applying the same 
arguments to interchanged metrics yields that the above distances 
differ by $o(t^{n+1})$. Similarly, $\la_g(O,A)-\la_{\wt g}(O,A)=o(t^{n+1})$. 
This proves the  claim.
\end{proof}

\begin{figure}[ht]
  \begin{center}
   \epsfig{file=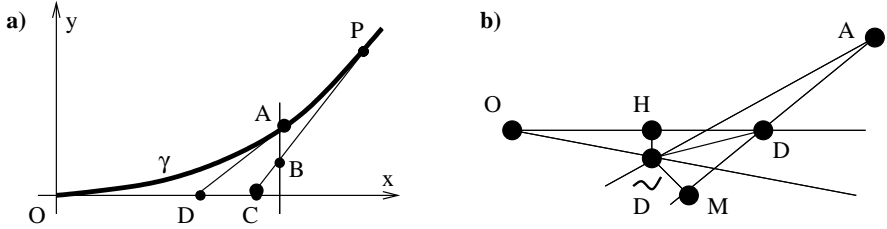}
    \caption{The curve $\gamma$, points $P$, $A$, $C$, $D$, $B$ (fig. a)).  
    The points $\wt D$, $H$, $M$; case 2) (fig. b)).}
    \label{fig:01}
  \end{center}
\end{figure}

Let $H$ and $M$ denote the points in the $\wt g$-geodesics $OD$ and $DA$ respectively that are $\wt g$-closest 
to $\wt D$: $\wt DH\perp OD$;  $\wt DM\perp DA$; see Fig. 8b). 

{\bf Claim 3.} {\it One has $|\wt DH|_{\wt g}=o(t^{n+1})$, $|\wt DM|_{\wt g}=o(t^{n+1})$.}

\begin{proof} The $\wt g$-geodesic $O\wt D$ is tangent to the $g$-geodesic $OD$ at $O$, and 
the metrics $g$ and $\wt g$ have the same $n$-jet at $O$. Therefore, their Christoffel symbols 
have the same $(n-1)$-jet, and hence, their difference is asymptotically dominated 
by the $g$-distance of $O$ in power $n-1$. This together with the equation of geodesics implies that the 
azimuths of the unit vectors tangent to both latter geodesics (as functions of the natural parameter 
based at $O$) differ by a quantity asymptotically dominated by $n$-th power to the $g$-distance to $O$. 
Therefore, the distance (in any metric) between points of the geodesics corresponding to the same natural parameter value  
is asymptotically dominated by the above distance in power $n+1$. Thus,   the distance 
of the point $\wt D$ to the geodesic $OD$ is $o(t^{n+1})$. Analogously, 
the same statement holds for 
distance to the $g$-geodesic $DA$. This proves the claim. 
\end{proof}

{\bf Claim 4.}  {\it One has $\Delta_3(t)=o(t^{n+1})$.} 

\begin{proof}  All the distances below are measured 
in the metric $\wt g$. One has 
\begin{equation}|O\wt D|-|OH|=O(\frac{|\wt D H|^2}{|O\wt D|})=o(t^{2n+1})=o(t^{n+1}),\label{o2n1}
\end{equation}
\begin{equation}|A\wt D|-|AM|=O(\frac{|\wt D M|^2}{|A\wt D|})=o(t^{2n+1})=o(t^{n+1}),\label{o2n2}
\end{equation}
by  (\ref{abc}) (applied to  the right $\wt g$-triangles $O\wt D H$ and $A\wt DM$) and Claim 3,
\begin{equation}|OD|-|OH|=\pm|DH|, \ |AD|-|AM|=\pm|DM|,\label{pmpm}
\end{equation}
see the cases of signs (which do not necessarily 
coincide) below. Taking sum of  equalities (\ref{pmpm}) and its difference 
with (\ref{o2n1}), (\ref{o2n2}) yields
\begin{equation}
\Delta_3(t)=(\pm)|DH|\pm|DM|+o(t^{n+1}).\label{de4g}\end{equation}

Case 1). In the right triangle $D\wt D H$ the angle $D$ is bounded from below (along some sequence 
of parameter values $t$ converging to 0). Then the same statement holds in the right triangle 
$\wt D MD$, since the angle between the geodesics 
$DA$ and $OD$ tends to 0 as $O(t)$. This implies that  $|DH|=O(|\wt D H|)=o(t^{n+1})$, 
and $|DM|=O(|\wt DM|)=o(t^{n+1})$, by Claim 3. This together with (\ref{de4g}) implies Claim 4 
(along the above sequence)

Case 2). In the right tringle $D\wt D H$ the angle $D$ tends to zero along some sequence of 
parameter values $t$ converging to 0, see Fig. 8b). 
Then the same holds in $\wt D MD$. 
In this case the signs in (\ref{de4g}) are different. For example, if $H$ lies between $O$ and $D$,  
then the angle $\angle\wt D D A$ is obtuse and $D$ lies between $M$ and $A$. The opposite case is treated analogously.  Let us denote 
$$\alpha(t):=\angle\wt D DH, \ \beta(t):=\angle\wt D DM; \ \alpha(t),\beta(t)\to0 \text{ as } t\to0.$$ 
Applying (\ref{abc}) to the above right triangles together with Claim 3 yields
$$|\wt D D|-|DH|=O(\alpha(t)|\wt DH|)=o(t^{n+1}),$$
$$ |\wt D D|-|DM|=O(\beta(t)|\wt DM|)=o(t^{n+1}).$$
Hence, $|DH|-|DM|=o(t^{n+1})$. 
This together with (\ref{de4g}) implies the asymptotics of Claim 4 (along the above sequence). Claim 4 is proved.
\end{proof}
 
 Claims 1, 2 and 4 together with (\ref{deltygg}) imply the statement of Lemma \ref{lemli} on the function $L$. 
 In its turn, it implies the same statement on $\La$. 
\end{proof}

\subsection{Taylor coefficients of $\La(t)$: analytic 
dependence on jets}

\begin{lemma} \label{leman} Let $(x,y)$ be  coordinates on 
a neighborhood of a point $O\in\Sigma$. 
Let a metric on $\Sigma$ be $C^n$-smooth, $n\geq3$, and let $\gamma$  be a germ of 
$C^n$-smooth curve on $\Sigma$ at $O$. Then the corresponding functions $L(0,t)$, $\La(t)$ are $O(t^3)$. They admit 
asymptotic Taylor expansions up to  $t^{n+1}$. Their coefficients at $t^{n+1}$ 
 are analytic functions of 
the pure $n$-jets of the metric and the curve $\gamma$. \end{lemma}
\begin{proof} The asymptotics $L(0,t),\La(t)=O(t^3)$ follows from 
Theorem \ref{lasyl}.  
 
 Case 1): the curve $\gamma$ and the metric are analytic. Consider the 
metric and the curve with variable Taylor coefficients of orders up to $n$; the other, higher Taylor 
coefficients are fixed. Consider  $L(0,t)$ and $\La(t)$  as functions in $t$ and in the latter variable 
Taylor coefficients. They are analytic in  $t$ and in the Taylor coefficients of order up to $n$: analytic 
on the product of a small disk centered at 0 with coordinate $t$ and a domain in the space of 
collections of Taylor coefficients. In more detail, complexifying everything we 
get that $L(0,t)$ has a well-defined holomorphic extension to complex domain. 
(The complexified lengths of segments in the definition of the function  
$L(0,t)$  become integrals of appropriate
holomorphic forms along paths.) Well-definedness  follows from the fact 
that through each point $C$ in a complex neighborhood of the real curve 
$\gamma$ there are two complex geodesics tangent to its complexification. 
This follows by  quadraticity of tangencies (non-vanishing of geodesic 
curvature) and Implicit Function Theorem. Analytic extendability to the 
locus $\{ t=0\}$ follows from the Erasing Singularity Theorem on 
bounded functions holomorphic on  complement to a hypersurface. 
Therefore, both functions admit a Taylor series in $t$ with coefficients being analytic functions in the  
Taylor coefficients of orders up to $n$ of the metric and the curve. They depend only on pure $n$-jets, 
since applying a translation of both the curve and the metric leaves $L(0,t)$ and $\La(t)$ invariant.

Case 2) of  general $C^n$-smooth metric $g$ and curve $\gamma$. 
Consider other, analytic metric $\wt g$ and curve $\wt\gamma$ representing their $n$-jets. The  functions $\wt L(0,t)$ and $\wt\La(t)$ defined by them are 
analytic and coincide with the  functions $L(0,t)$ and $\La(t)$ corresponding 
to $g$ and $\gamma$ up to  $o(t^{n+1})$. Indeed, if the $C^{n}$-smooth function $y=f(x)$ 
representing $\gamma$ as a graph changes in the same $n$-jet, i.e., by a quantity  $o(x^n)$, 
then $L(0,t)$, $\La(t)$ change by 
a quantity of order $o(t^{n+1})$. This follows from the results of Subsection 8.4 applied to $b=0$. 
 A similar statement holds for  
change of metric inside a given $n$-jet, by Lemma \ref{lemli}. 
This together with the discussion in Case 1) implies 
that $L(0,t)$ and $\La(t)$ have asymptotic Taylor expansions of order up to $t^{n+1}$ coinciding with those of 
$\wt L(0,t)$ and $\wt\La(t)$, and hence, having coefficients being analytic functions of the pure $n$-jets of $g$ and $\gamma$.  Lemma \ref{leman} is proved.
\end{proof}

\subsection{Proof of Lemma \ref{lmjet}} Let $\Sigma$ be a two-dimensional surface equipped with 
a $C^n$-smooth Riemannian metric $g$. 
Let $E\in\Sigma$, and let $V=V(E)\subset\Sigma$ be its 
small neighborhood. Let $(x,y)$ be (not necessarily normal) 
local coordinates on a domain containing $V$. 
Consider a $C^n$-smooth germ of curve $\gamma$ at a point $O\in V$ with positive geodesic curvature; the tangent line $T_O\gamma$ is not necessarily horizontal. The corresponding function $\La(t)$ admits an asymptotic Taylor expansion 
$$\Lambda(t)=\sum_{k=3}^{n+1}\wh\Lambda_kt^k+o(t^{n+1}).$$
 Its coefficients are analytic functions of the pure $n$-jets of the metric and  $\gamma$ at $O$ (Lemma \ref{leman}). Therefore, 
 without loss of generality we consider that $O$ is the origin in the coordinates 
 $(x,y)$, applying a translation, which changes neither $\La(t)$, nor the above 
 pure jets. Then $\gamma$  
is the graph of a $C^n$-function 
$$\gamma=\{  y= h(x)\}, \ \  h(x)=b_1x+\frac{b_2}2 x^2+\frac1{3!} b_3 x^3+\dots+\frac1{n!}b_n x^n
+o(x^{n}).$$
By definition, the coordinates of the pure jet $j^n_O\gamma$ are $(b_1,\dots, b_n)$. 

We already know that $\wh\La_{n+1}$ is an affine function in $b_n$, which follows from Lemma \ref{lcomp}, 
 see (\ref{dlat}). To obtain a precise formula for its coefficient at $b_n$, we use the following 
 proposition. 
 \begin{proposition} \label{prowu} Let $n\geq3$, $\Sigma$, $O$, 
 $(x,y)$, $h(x)$ be as above. Consider a family of tangent germs of 
 curves $\gamma_{n,b}=\{ y=h_{n,b}(x)\}$ at $O:=(x_0,y_0)$, $h_{n,b}(x)=h(x)+bx^n+o(x^n)$;  $h_{n,0}:=h$, $\gamma_{n,0}:=\gamma$. 
 Let $w\in T_O\Sigma$ 
 denote the orthogonal projection of the vector $\frac{\partial}{\partial y}$ 
 to $(T_O\gamma)^\perp$. Let $u=(1,b_1)\in T_O\gamma$: 
 the tangent vector to $\gamma$ with unit $x$-component.  Let $\kappa(O)$ denote the geodesic curvature of the curve $\gamma$ at $O$, which coincided with that  
of $\gamma_{n,b}$. Let $(\wt x,\wt y)$ be normal coordinates centered at $O$ such that the $\wt x$-axis is tangent to $\gamma$. Set 
$$\hat x:=\kappa(O)\wt x, \ \hat y:=\kappa(O)\wt y.$$
 In the coordinates $(\hat x,\hat y)$ the family of curves $\gamma_{n,b}$ is the family of graphs of functions 
 $\{ \hat y=\hat h_{n,b}(\hat x)\}$, set $\hat h_{n,0}:=\hat h$, such that 
 $\hat h(\hat x)=\frac{\hat x^2}2+O(\hat x^3)$, 
 \begin{equation} \hat h_{n,b}(\hat x)=\hat h(\hat x)+\mu_n b\hat x^n+o(\hat x^n), \ \mu_n=||w||||u||^{-n}{\kappa^{1-n}}(O).
 \label{wthh}\end{equation}
 \end{proposition}
 \begin{proof} Fix a point $A=(\wt x,0)$ on the $\wt x$-axis. Let $\ell$ denote the geodesic 
 through $A$ orthogonal to the $\wt x$-axis. We have to calculate the 
 gap (i.e., distance) $\wt\Delta(\wt x)$ 
 between the intersection points of the geodesic 
 $\ell$ with the curves $\gamma_{n,b}$ and $\gamma$. Let   
 $\Delta(\wt x)$ denote the gap between the points of their intersection with the vertical line $\{ x= x(A)\}$. 
 Their ratio $\wt\Delta(\wt x)\slash\Delta(\wt x)$ tends to the cosine of the angle between the 
 vector $\frac{\partial}{\partial y}\in T_O\Sigma$ and the 
 line $(T_O\gamma)^\perp$. One has $\Delta(\wt x)=||\frac{\partial}{\partial y}||bx^n+o(x^n)$. Hence,  
 by definition,  
  \begin{equation} \wt\Delta(\wt x)=||w||bx^n+o(x^n).\label{wtdx}\end{equation}
  One has $dx=\alpha d\wt x+\beta d\wt y$ on $T_O\Sigma$,  
  $\alpha=dx(\frac{\partial}{\partial\wt x})=||u||^{-1}$, by 
  definition; $x=\alpha\wt x+\beta\wt y+O(|\wt x|^2+|\wt y|^2)$. 
  One has 
  $\wt y=\frac{\kappa(O)}2\wt x^2=O(\wt x^2)$ along each curve $\gamma_{n,b}$, by (\ref{x122}). This together with 
  (\ref{wtdx}) implies that 
\begin{equation}\wt\Delta(\wt x)=||w||||u||^{-n}b\wt x^n+o(\wt x^n). \label{wtde}\end{equation}
Hence, in the coordinates $(\wt x,\wt y)$ 
$$\gamma_{n,b}=\{\wt y=\wt h_{n,b}(\wt x)\},  \ \ 
\wt h_{n,b}(\wt x)=\wt h_{n,0}(\wt x)+||w||||u||^{-n}b\wt x^n+o(\wt x^n).$$ 
Now rescaling to the coordinates $(\hat x,\hat y)$ yields that   $\gamma_{n,b}$ is a family of graphs of functions $\hat h_{n,b}(\hat x)$ satisfying (\ref{wthh}). The proposition 
is proved.
\end{proof} 
\begin{proposition} \label{propol} Consider the above family of curves $\gamma_{n,b}$ 
and the corresponding functions $\La^{n,b}(t)$, set $\La^{n,0}:=\La$. One has 
\begin{equation}\wh\La^{n,b}_{n+1}=\wh\La_{n+1}+\nu_nb, \ 
\nu_n:=\frac{(n-2)(n-3)}{6(n+1)}||w||(||u||\kappa(O))^{-n}.\label{whlanb}
\end{equation}
\end{proposition}
\begin{proof} The coordinates $(\hat x,\hat y)$ are normal coordinates for the rescaled metric $\hat g:=\kappa(O)g$. The 
common geodesic curvature at $O$  of the curves $\gamma_{n,b}$ in the metric $\hat g$ is equal to 1, by construction. Therefore, for the metric  
$\hat g$ one has  
$\wh\La^{n,b}_{n+1}-\wh\La_{n+1}=\frac{(n-2)(n-3)}{6(n+1)}\mu_nb$, by Lemma \ref{lcomp} and (\ref{wthh}). Rescaling 
the metric back to $g$ by the factor $\kappa^{-1}(O)$ rescales the functions $\La_{n,b}$ and their Taylor coefficients by the same factor (Remark \ref{rtinv}). This implies (\ref{whlanb}).
\end{proof}
 \begin{proposition} Let $\gamma$ be a germ of $C^n$-smooth curve at a point $O\in\Sigma$ lying in a chart with coordinates $(x,y)$. Let $\gamma$ 
 be  
 a graph $\{ y=h(x)\}$.  Let $b_1,\dots,b_n$ denote the  coordinates of the 
  pure $n$-jet $j^n_Oh$. Let 
 $w,u\in T_O\Sigma$ be the vectors from Proposition \ref{prowu}. 
 Then the Taylor coefficient $\wh\La_{n+1}$ of the corresponding function $\La(t)$ is equal to 
\begin{equation}\wh\La_{n+1}=\sigma_nb_n-P_n,\label{whla}\end{equation}
\begin{equation}\sigma_n=\frac{(n-2)(n-3)}{6(n+1)!}||w||\left(||u||\kappa(O)\right)^{-n},
\label{sinsin}
\end{equation}
 where $P_n$ is an analytic   function in 
 $b_1,\dots,b_{n-1}$ and in  the pure $n$-jet of the metric 
 at $O$. 
 \end{proposition}
 \begin{proof} The fact that $\wh\La_{n+1}$ depends on $b_n$ as an affine function with factor 
 $\sigma_n$ at $b_n$ follows from definition and  Proposition \ref{propol};  the $b$ from Proposition \ref{propol} 
 is $\frac1{n!}$ times the difference of the $b_n$-coordinates of jets of functions $h_{n,b}(x)$ and $h(x)$. 
 The function $P_n$ is thus independent on $b_n$ and hence, has the required type, by Lemma \ref{leman}.
 \end{proof}
 
\begin{proof} {\bf of Lemma \ref{lmjet} and its addendum.} All the statements of Lemma \ref{lmjet} 
and its addendum follow from the above proposition, except for the following points discussed below. 
Note that $\sigma_n$ depends only on the pure  2-jet of the curve $\gamma$ and the pure 1-jet of the metric, 
by definition. The function $P_n$ is an analytic function of the pure $n$-jet of the 
metric and the pure $(n-1)$-jet of the curve $\gamma$. Let us treat it as a function of 
a point and a pure $(n-1)$-jet of curve. We have to prove its smoothness. To this end, we use the assumption 
that the metric is $C^{n+1}$-smooth. (This is the only place in the proof where we use this assumption.) 
Then its pure $n$-jet is a $C^1$-smooth function of a point. This together with the above analyticity statement 
proves $C^1$-smoothness and finishes the proof of Lemma \ref{lmjet}.
\end{proof}

 \subsection{Proof of Theorems \ref{osc} and \ref{uniq4}} 
\begin{proof} {\bf of Theorem \ref{osc}.} Let $O\in\Sigma$. Let $(x,y)$ be local coordinates on a neighborhood 
$V=V(O)\subset\Sigma$. Let $\mcj^4_y(V)$ denote the space of 4-jets of curves,  
as in Lemma \ref{lmjet}. Let $J_2=(x,b_0,b_1,b_2)$, 
$\sigma_5=\sigma_5(J_2)$ and $h_5:=P_5(J_2;b_3,b_4)$ be the same, as in (\ref{lan+1}). 
Consider the field of kernels $K_4$  of the following 1-form $\nu_4$ on  $\mcj^4_y(V)$: 
$$\nu_4:=db_4-\sigma_5^{-1}h_5(x,b_0,b_1,b_2,b_3,b_4)dx; \ 
K_4:=\operatorname{Ker}(\nu_4).$$
Let $\mcd_4$ denote the canonical distribution on $\mcj^4_y(V)\simeq\mcf^4_y(V)$, see (\ref{candis}): 
$$\mcd_{4}=\operatorname{Ker}(db_0-b_1dx,  db_1-b_2dx, db_2-b_3dx, 
db_3-b_4dx).$$
Set 
\begin{equation}\mcp:=K_4\cap\mcd_{4}.\label{defdp}\end{equation}
This is a line field, since  the above intersections are obviously transverse and 
$\dim(\mcd_{4})=2$. It is $C^1$-smooth, since so are $\sigma_2$ and 
$h_5$ (Lemma \ref{lmjet}). Let $\gamma$ be an arbitrary $C^5$-smooth  germ of curve $\gamma$ based at a point  $A\in V$ such that the line $T_A\gamma$ is not parallel to the $y$-axis. 
Let $\gamma$ have  string Poritsky property. Then  $\Lambda(t)\equiv0$, 
hence, $\wh\Lambda_6=0$, thus, 
\begin{equation}\sigma_5(J_2)b_5-h_5(J_2;b_3,b_4)=0,\label{s5b}\end{equation}
by (\ref{lan+1}). On the other hand,  
the  $5$-jet extension of the curve $\gamma$ 
is  tangent to the canonical distribution $\mcd_{5}$, and hence, 
to the hyperplane field $\{ db_4=b_5dx\} $. This together with (\ref{s5b}) 
implies that its 4-jet extension  
is tangent to the hyperplane field  $\{ db_4=\frac{h_5}{\sigma_5}dx\}$. Thus, it is tangent to 
 the kernel field $K_4$, and hence, to   $\mcp=K_4\cap\mcd_{4}$. 
 This proves Theorem \ref{osc}. \end{proof}
 
\begin{proof} {\bf of Theorem \ref{uniq4}.}  Two germs of curves with string Poritsky property and the same 
4-jet correspond to  the same point in $\mcj^4$. Therefore, their 4-jet extensions  coincide with one and the same phase curve of the 
line field $\mcp$, by Theorem \ref{osc} and the Uniqueness Theorem for ordinary 
differential equations. Thus, the germs coincide. This proves Theorem \ref{uniq4}. 
\end{proof} 

\section{Acknowledgements} 

I wish to thank Sergei Tabachnikov, to whom this paper is much due, for introducing me into 
the topic of curves with Poritsky property and related areas. I wish to thank to him and to 
Misha Bialy and Maxim Arnold for helpful 
discussions. Most of results of the paper were obtained while I was visiting Mathematical Sciences Research Institute (MSRI) 
in Berkeley, California. I wish to thank MSRI for hospitality and support.

\end{document}